\documentclass[11pt]{amsart}

\usepackage{amsmath}
\usepackage{amssymb}
\usepackage[normalem]{ulem}
\usepackage[all]{xy}
\usepackage{longtable}
\usepackage{titletoc}
\usepackage{mathrsfs}
\usepackage{mathtools}
\usepackage{amscd}
\usepackage{appendix}
\usepackage{amsthm}
\usepackage{amsfonts}
\usepackage{tikz-cd}
\usepackage{array}
\usepackage{graphicx}
\usepackage{hyperref}
\usepackage{longtable}
\usepackage{xtab}
\usepackage{soul,xcolor}
\usepackage{upgreek}
\usepackage[total={15.875truecm,22.86truecm},centering]{geometry}
\usepackage{stackrel}
\usepackage{eucal}
\usepackage{bbm}
\usepackage[outline]{contour}
\usepackage{scalerel}
\usepackage{color}
\usepackage{nicematrix}

\newtheorem{thm}[equation]{Theorem}
\newtheorem{lem}[equation]{Lemma}
\newtheorem{prop}[equation]{Proposition}
\newtheorem{cor}[equation]{Corollary}

\theoremstyle{definition}
\newtheorem{defn}[equation]{Definition}

\theoremstyle{remark}
\newtheorem{rem}[equation]{Remark}
\newenvironment{Subsubsec}[1]{\vspace{\topsep}
\noindent\refstepcounter{equation}\theequation.\ \textit{#1.}}{\vspace{\topsep}}

\numberwithin{equation}{subsection}

\DeclareMathAlphabet{\mathpzc}{OT1}{pzc}{m}{n}
\DeclareMathAlphabet{\matheur}{U}{eur}{m}{n}

\definecolor{lava}{RGB}{207,16,32}
\definecolor{purple}{RGB}{148,0,211}

\renewcommand{\AA}{\mathbb{A}}
\newcommand{\FF}{\mathbb{F}}
\newcommand{\ZZ}{\mathbb{Z}}

\newcommand{\RR}{\mathbb{R}}

\newcommand{\GG}{\mathbb{G}}
\newcommand{\EE}{\mathbb{E}}
\newcommand{\CC}{\mathbb{C}}

\newcommand{\PP}{\mathbb{P}}

\newcommand{\NN}{\mathbb{N}}

\newcommand{\boz}{\boldsymbol{z}}
\newcommand{\bow}{\boldsymbol{w}}

\newcommand{\cfk}{\mathfrak{c}}
\newcommand{\dfk}{\mathfrak{d}}
\newcommand{\kfk}{\mathfrak{k}}

\newcommand{\pfk}{\mathfrak{p}}

\newcommand{\ufk}{\mathfrak{u}}

\newcommand{\Cfk}{\mathfrak{C}}

\newcommand{\Ifk}{\mathfrak{I}}
\newcommand{\Jfk}{\mathfrak{J}}

\newcommand{\Pfk}{\mathfrak{P}}

\newcommand{\Tfk}{\mathfrak{T}}

\newcommand{\Zfk}{\mathfrak{Z}}

\newcommand{\Ecal}{\CMcal{E}}

\newcommand{\Lcal}{\CMcal{L}}

\newcommand{\Ncal}{\CMcal{N}}
\newcommand{\Ocal}{\CMcal{O}}
\newcommand{\Pcal}{\CMcal{P}}
\newcommand{\Rcal}{\CMcal{R}}

\newcommand{\Jcal}{\CMcal{J}}
\newcommand{\Ical}{\CMcal{I}}

\newcommand{\Kcal}{\CMcal{K}}
\newcommand{\Xcal}{\CMcal{X}}

\newcommand{\Bscr}{\mathscr{B}}

\DeclareMathOperator{\GL}{GL}
\DeclareMathOperator{\Mat}{Mat}

\DeclareMathOperator{\End}{End}
\DeclareMathOperator{\res}{res}

\DeclareMathOperator{\im}{Im}

\DeclareMathOperator{\rank}{rank}
\DeclareMathOperator{\ord}{ord}

\def\rank{\operatorname{rank}}

\newcommand{\PGL}{\operatorname{PGL}}

\newcommand{\Pic}{\operatorname{Pic}}

\newcommand{\re}{\operatorname{Re}}

\newcommand\subfrac[2]{\genfrac{}{}{0pt}{}{#1}{#2}}

\newcommand{\assign}{\mathrel{\vcenter{\baselineskip0.5ex \lineskiplimit0pt
                     \hbox{\scriptsize.}\hbox{\scriptsize.}}}%
                     =}

\hypersetup{
	colorlinks=true,
	linkcolor=blue,
	citecolor=red,
	linktoc=page,
	linkbordercolor=blue,
	citebordercolor=red,
}

\begin{document}

\title{Several-variable Kronecker limit formula over global function fields}
\author[Fu-Tsun Wei]{Fu-Tsun Wei}
\address{Department of Mathematics, National Tsing Hua Univeristy, Taiwan}
\email{ftwei@math.nthu.edu.tw}


\date{\today}

\begin{abstract}
We establish Kronecker-type first and second limit formulas for ``non-holomorphic'' and ``Jacobi-type'' Eisenstein series over global function fields in the several-variable setting. Our main theorem demonstrates that the derivatives of these Eisenstein series can be understood as averaged integrals of certain period quantities along the associated ``Heegner cycles'' on Drinfeld modular varieties. A key innovation lies in our use of the Berkovich analytic structure of the Drinfeld period domains, which enables the parametrization of the Heegner cycles in question by Euclidean ``parallelepiped'' regions.
This approach also facilitates a unified and streamlined formulation and proof of our results. Finally, we apply these formulas to provide period interpretations of the ``Kronecker terms'' of Dedekind--Weil zeta functions and Dirichlet $L$-functions associated with ring and ray class characters. 
\end{abstract}

\maketitle

\section{Introduction}\label{Intro}

In the classical setting, the story of ``second terms'' can be traced back to the study of the logarithmic derivative of the Riemann zeta function at $s=0$.
Concretely,
let
\[
\zeta(s)\assign \sum_{n \in \NN} \frac{1}{n^s} \ \left(= \frac{1}{2}\sum_{0\neq n \in \ZZ}\frac{1}{|n|^s}\right), \quad \re(s)>1.
\]
Extending $\zeta(s)$ meromorphically to the complex $s$-plane, one knows that (see \cite[Corollary~9.13]{KKS3})
$\zeta(0) = -\frac{1}{2}$
and
$\zeta'(0) = -\frac{1}{2}\ln(2\pi)$,
which leads to the following equality
\[
\frac{\zeta'(0)}{\zeta(0)} = \ln(2\pi).
\]
In geometric terms, this says that the logarithmic derivative of $\zeta(s)$ at $s=0$ matches the ``log-length'' of the period of the unit circle (which is isomorphic to $\RR/\ZZ$).

Turning to the rank $2$ case, let $z \in \CC$ with  $\im(z)>0$. Consider the non-holomorphic Eisenstein series
\[
E(z,s)\assign \sum_{0 \neq \lambda \in \ZZ z + \ZZ} \frac{\im(z)^s}{|\lambda|^{2s}}, \quad \re(s)>1.
\]
Extending $E(z,s)$ to a meromorphic function on the complex $s$-plane,
the celebrated Kronecker limit formula states that (see \cite[Theorem~9.11 and 9.14]{KKS3})
\[
E(z,0)=-1
\quad \text{ and } \quad 
\frac{\partial}{\partial s} E(z,s)\bigg|_{s=0} = -\ln \im(z) - \frac{1}{6} \ln |\Delta(z)|,
\]
where $\Delta(z) = (2\pi)^{12} e^{2\pi i z}\prod_{n=1}^{\infty}(1-e^{2 \pi i n z})^{24}$ is the {\it modular discriminant function}.
If $\Ecal$ is a ``unit'' elliptic curve over $\CC$ (i.e.~of discriminant $1$) which is analytically isomorphic to the complex torus $\CC/(\ZZ z + \ZZ)$, 
then the period lattice of $\Ecal$ is $\Lambda_{\Ecal} = c \cdot (\ZZ z+\ZZ)$ for a complex number $c$ with $c^{12} = \Delta(z)$.
Therefore the limit formula implies
\[
\frac{E'(z,0)}{E(z,0)} = \ln\big(D(\Lambda_{\Ecal})\big),
\]
where $D(\Lambda_{\Ecal})\assign |c|^2 \cdot \im(z)\ (= \text{vol}(\{a\cdot cz+b\cdot c\mid 0\leq a,b \leq 1\}))$ is the co-volume of the period lattice $\Lambda_{\Ecal}$ of the unit elliptic curve $\Ecal$ analytically isomorphic to $\CC/(\ZZ z + \ZZ)$.

From this perspective, one sees that the Kronecker (first) limit formula admits a natural interpretation in terms of ``periods.''  Along with the second limit formula for “Jacobi-type” Eisenstein series (see \cite{Sie}), these formulas and their generalizations play a pivotal role in the arithmetic of modular forms and elliptic curves, contributing to various significant developments in number theory (see \cite{Bei}, \cite{B-G}, \cite{Col}, \cite{C-S}, \cite{Hecke}, \cite{L-M}, \cite{Yang}, and \cite{Zag}, etc.).

The main purpose of this paper is to revisit the Kronecker limit formula over global function fields within the framework of ``Berkovich context'' and to explore the aforementioned phenomena in a several-variable setting. By leveraging the Berkovich analytic structure of the Drinfeld period domain, our approach not only simplifies a range of analytic technicalities but also provides fresh insights and a unified perspective on these formulas.

\subsection{Preliminaries on Berkovich Drinfeld period domains}

Let $k$ be a global function field with a finite constant field $\FF_q$.
Fix a place $\infty$ of $k$, referred to the place at infinity,
and the corresponding normalized absolute value is denoted by $|\cdot|_{\infty}$.
Let $k_{\infty}$ be the completion of $k$ at $\infty$ and put $O_{\infty} \assign \big\{a \in k_{\infty} \ \big|\ |a|_{\infty}\leq 1\big\}$, the valuation ring in $k_{\infty}$.
Let $\CC_{\infty}$ be the completion of a chosen algebraic closure of $k_{\infty}$.

Take a positive integer $r$.
Let $R_{r,\infty}$ be the symmetric algebra of $k_{\infty}^r$ over $\CC_{\infty}$, and nonzero elements of $k_{\infty}^r$ are regarded as $k_{\infty}$-rational degree one elements in $R_{r,\infty}$.
The {\it Berkovich Drinfeld period domain of rank $r$ over $k_{\infty}$} is 
defined by (see Definition~\ref{defn: Omega})
\[
\Omega_{\infty}^r \assign 
\PP_{\CC_{\infty}}^{r-1,{\rm an}} \setminus \bigg(\bigcup_{\subfrac{\text{ $k_{\infty}$-rational}}{\text{hyperplane $H$}}} H^{\rm an}
\bigg),
\]
which can be identified with the set consisting of all equivalence classes of multiplicative seminorms on $R_{r,\infty}$ which are non-vanishing at every nonzero element in $k_{\infty}^r$ (see Remark~\ref{rem: Omega-id}).

Let $e_r = (0,\cdots \hspace{-0.03cm}, 0,1) \in k_{\infty}^r$.
Every point $z \in \Omega_{\infty}^r$ corresponds to a unique multiplicative seminorm on $R_{r,\infty}$ so that $e_r(z)=1 \in \CC_{\infty}(z)$, the residue field at $z$ over $\CC_{\infty}$ (see \eqref{eqn: Omega-1}).
The {\it building map} from $\Omega_{\infty}^r$ to the ``Euclidean'' Bruhat--Tits building $\Xcal^r(k_{\infty})$ (introduced in Remark~\ref{rem: build}~(2)) comes from the restriction of the seminorm $z$ to  $k_{\infty}^r$, which gives a non-archimedean norm $\nu_z$ (see \eqref{eqn: nu-z}).
The {\it imaginary part $\im(z)$ of $z$} $\in \Omega_{\infty}^r$ is defined as the volume $D(\nu_z)$ of the unit ball $O_{\infty}^r$ with respect to $\nu_z$ (see Section~\ref{sec: Imag}). This quantity plays a crucial rule in our study and provides a ``period interpretation'' for our main result presented in the following subsection.

\subsection{Main theorem}

Let $K$ be a finite extension of $k$ and $n \assign [K:k]$.
Write
\[
K\otimes_{k} k_{\infty} \cong \prod_{i=1}^{m} K_{\infty_i},
\]
where $\infty_1,...,\infty_m$ are the places of $K$ lying above $\infty$ and $K_{\infty_i}$ is the completion of $K$ at $\infty_i$ for $1\leq i \leq m$.
Put $n_i \assign [K_{\infty_i}:k_{\infty}]$.
Let $A$ be the ring of functions in $k$ regular away from $\infty$, and let $\Ocal$ be an $A$-order in $K$.
Take a positive integer $r$ and a projective $\Ocal$-module $\Lcal \subset K^r$.
For each $\boz = (z^{(1)},\cdots \hspace{-0.03cm},z^{(m)}) \in \prod_{i=1}^m \Omega_{\infty_i}^r$,
where $\Omega_{\infty_i}^r$ is the Berkovich Drinfeld period domain of rank $r$ over $K_{\infty_i}$,
we introduce the following {\it several-variable non-holomorphic Eisenstein series}:
\[
\EE^{\Lcal}_{\Ocal}(\boz,s)
\assign \Vert \Lcal \Vert_{\Ocal} \cdot \sum_{\lambda \in (\Lcal-\{0\})/\Ocal^{\times}}
\left(
\prod_{i=1}^m \frac{\im(z^{(i)})^s}{\nu_{z^{(i)}}(\lambda)^{rs}}
\right),
\]
where $\Vert \Lcal \Vert_{\Ocal} \assign |a|_{\infty}^{-rn} \cdot \#(\Ocal^r/a \cdot \Lcal)$ for an(y) $a \in A$ sufficiently large so that $a \Lcal \subset \Ocal^r$.
Also, for each $\bow \in K^r \setminus \Lcal$, the corresponding {\it several-variable Jacobi-type Eisenstein series} is defined by:
\[
\EE_{\Ocal}^{\Lcal}(\boz,\bow,s)
\assign \Vert \Lcal \Vert_{\Ocal} \cdot \sum_{\lambda_{\bow} \in (\bow+\Lcal)/\Ocal^{(1)}_{\Cfk_{\bow}}}
\left( \prod_{i=1}^{m}
\frac{\im(z^{(i)})^s}{\nu_{z^{(i)}}(\lambda_{\bow})^{rs}}
\right),
\]
where $\Cfk_{\bow} \assign \{\alpha \in \Ocal \mid \alpha \cdot \bow \in \Lcal\}$ and $\Ocal^{(1)}_{\Cfk}\assign\big\{u \in \Ocal^{\times}\ \big|\ u \equiv 1 \bmod \Cfk \big\}$ for every nonzero ideal $\Cfk$ of $\Ocal$.
The main theorem of this paper is presented as follows:

\begin{thm}\label{thm: KLF}
Keep the above notation.
The Eisenstein series $\EE^{\Lcal}_{\Ocal}(\boz,s)$ and $\EE^{\Lcal}_{\Ocal}(\boz,\bow,s)$ both converge absolutely for $\re(s)>1$ and have meromorphic continuation to the whole complex $s$-plane.
Moreover, put
\[
\widetilde{\EE}_{\Ocal}^{\Lcal}(\boz,s) \assign \frac{\dfk(\Ocal/A)^{\frac{rs}{2}}}{(rs)^{m-1}} \cdot \EE_{\Ocal}^{\Lcal}(\boz,s)
\quad \text{ and } \quad 
\widetilde{\EE}_{\Ocal}^{\Lcal}(\boz,\bow,s) \assign \frac{\dfk(\Ocal/A)^{\frac{rs}{2}}}{(rs)^{m-1}} \cdot \EE_{\Ocal}^{\Lcal}(\boz,\bow,s),
\]
where $\dfk(\Ocal/A)\in \NN$ is the ``discriminant quantity'' of $\Ocal$ over $A$ defined in \eqref{eqn: d(O/A)}.
Then the following equalities hold:
\begin{itemize}
    \item[(1)] (The first limit formula)
    \[    \widetilde{\EE}_{\Ocal}^{\Lcal}(\boz,0) = -\frac{\Rcal(\Ocal)}{q_{\Ocal}-1}
    \quad \text{ and } \quad 
    \frac{\partial}{\partial s}\widetilde{\EE}_{\Ocal}^{\Lcal}(\boz,s)\bigg|_{s=0}
    = -\frac{n_1\cdots n_m}{n(q_{\Ocal}-1)} \int_{\Zfk_{\Lcal}(\boz)} \ln \eta_{nr}^{Y_{\Lcal}}(\mathbf{z}) \ d \mathbf{z},
    \]
    where:
    \begin{itemize}
        \item[$\bullet$] $\Rcal(\Ocal)$ is the regulator of $\Ocal$ (see \eqref{eqn: R(O)});
        \item[$\bullet$] $q_{\Ocal}$ is the number of constants in $\Ocal$ (given in {\rm Theorem~\ref{Key}});
        \item[$\bullet$] $Y_{\Lcal}$ is the $A$-submodule in $k^{nr}$ corresponding to $\Lcal$ with respect to a fixed $k$-basis $\beta_{\circ,r}$ of $K^r$ (see \eqref{eqn: beta-0-r});
        \item[$\bullet$] $\Zfk_{\Lcal}(\boz)$ is the ``Heegner cycle'' associated with $\boz$ and $\Lcal$ on $\GL(Y_{\Lcal})\backslash \Omega_{\infty}^{nr}$ (see {\rm Remark~\ref{rem: HC}});
        \item[$\bullet$] $\eta_{nr}^{Y_{\Lcal}}$ is the ``absolute discriminant'' introduced in {\rm Definition~\ref{defn: eta}}.
    \end{itemize}
    \item[(2)] (The second limit formula)
    \[    \widetilde{\EE}_{\Ocal}^{\Lcal}(\boz,\bow,0) = 0
    \quad \text{ and } \quad 
    \frac{\partial}{\partial s}\widetilde{\EE}_{\Ocal}^{\Lcal}(\boz,\bow,s)\bigg|_{s=0}
    = -\frac{(n_1\cdots n_m)\cdot r}{q^{nr\deg \infty}-1} \cdot \int_{\Zfk_{\Lcal}(\boz;\Cfk_{\bow})} \ln \big|\ufk_{w}^{Y_{\Lcal}}(\mathbf{z})\big| \ d \mathbf{z},
    \]
    where:
    \begin{itemize}
        \item[$\bullet$] $\Zfk_{\Lcal}(\boz,\Cfk)$ is the ``Heegner cycle'' associated with $\boz$, $\Lcal$, and a nonzero ideal $\Cfk$ of $\Ocal$ on $\Gamma^{Y_{\Lcal}}(\Cfk)\backslash \Omega_{\infty}^{nr}$, and $\Gamma^{Y_{\Lcal}}(\Cfk) \subset \GL(Y_{\Lcal})$ is the congruence subgroup ``modulo $\Cfk$'' (see {\rm Remark~\ref{rem: HC-2}});
        \item[$\bullet$] $\ufk_{w}^{Y_{\Lcal}}(\mathbf{z})$ is the ``Drinfeld--Siegel unit'' associated with $w$ on $\Omega_{\infty}^{nr}$ introduced in {\rm Remark~\ref{rem: KLF-reform}~(2)}, and $w \in k^{nr}$ corresponds to $\bow \in K^r$ with respect to the fixed $k$-basis $\beta_{\circ,r}$ of $K^r$ (see \eqref{eqn: w}).
    \end{itemize}
\end{itemize}
\end{thm}

\begin{rem}\label{rem: KLF}
${}$

(1) When the extension $K/k$ is tamely ramified at the places 
    $\infty_1,...,\infty_m$, (for which $K$ is separable over $k$), the quantity $\dfk(\Ocal/A)$ is the norm of the discriminant ideal of $\Ocal$ over $A$ (see Remark~\ref{rem: HC}~(2)).

(2) We provide a more detailed description of the Heegner cycle $\Zfk_{\Lcal}(\boz)$ 
    as follows. 
    For $1\leq i \leq m$, let $\widetilde{\Omega}_{\infty_i}^r$ (resp.~$\widetilde{\Omega}_{\infty}^{nr}$) be the set of all seminorms on the symmetric algebra $R_{r,\infty_i}$ of $K_{\infty_i}^r$ (resp.~$R_{nr,\infty}$) which are non-vanishing on every nonzero element in $K_{\infty_i}^r$ (resp.~$k_{\infty}^{nr}$), and $\Ncal^r(K_{\infty_i})$ (resp.~$\Ncal^{r'}(k_{\infty})$) be the {\it Goldman--Iwahori space} consisting of all non-archimedean norms on $K_{\infty_i}^r$ (resp.~$k_{\infty}^{r'}$ for every $r' \in \NN$). Let $\widetilde{\Bscr}_i: \widetilde{\Omega}_{\infty_i}^r\rightarrow \Ncal^r(K_{\infty_i})$ be the map induced by restricting each seminorm $\tilde{z}^{(i)}$ on $R_{r,\infty_i}$ to a norm $\nu_{\tilde{z}^{(i)}}$ on $K_{\infty_i}^r$ (see \eqref{eqn: nu-z}), and put $\widetilde{\Bscr}=\prod_{i=1}^m \widetilde{\Bscr}_i$.
    Then the (product of the) ``restriction of scalars'' 
    \[
    \res: \prod_{i=1}^m \Ncal^r(K_{\infty_i})\rightarrow \prod_{i=1}^m \Ncal^{n_ir}(k_{\infty}) \rightarrow \Ncal^{nr}(k_{\infty})
    \]
    introduced in \eqref{eqn: res} (and Section~\ref{sec: prod-norm}), together with a natural inclusion $\tilde{\varsigma}: \Ncal^{nr}(k_{\infty})\rightarrow \widetilde{\Omega}_{\infty}^{nr}$ (see Remark~\ref{rem: nu-z}), gives rise to the following sequence of maps
    \begin{equation}\label{eqn: Z-map}
    \prod_{i=1}^m\widetilde{\Omega}_{\infty_i}^r
    \stackrel{\widetilde{\Bscr}}{\longrightarrow} \prod_{i=1}^m \Ncal^r(K_{\infty_i})\stackrel{\res}{\longrightarrow} \Ncal^{nr}(k_{\infty}) \stackrel{\tilde{\varsigma}}{\longrightarrow} \widetilde{\Omega}_{\infty}^{nr}\longrightarrow
    \Omega_{\infty}^{nr}\stackrel{g_{\circ,r}}{\longrightarrow} \Omega_{\infty}^{nr} \longrightarrow \GL(Y_{\Lcal})\backslash \Omega_{\infty}^{nr}.
    \end{equation}
    Here $g_{\circ,r} \in \GL_{nr}(k_{\infty})$ given in \eqref{eqn: g-0-r} is the ``transition matrix'' associated with the fixed ``global'' $k$-basis $\beta_{\circ,r}$ of $K^r$ and another fixed ``local'' identification 
    \[
    K^r\otimes_k k_{\infty} \cong \prod_{i=1}^m K_{\infty_i}^r \cong k_{\infty}^{nr} \quad \text{(see \eqref{eqn: iota-r})},
    \]
    and the left action of $g_{\circ,r}$ on $\widetilde{\Omega}_{\infty}^{nr}$ induces an automorphism (still denoted by $g_{\circ,r}$) on $\Omega_{\infty}^{nr}$.
    Let $\Upsilon: \prod_{i=1}^m\widetilde{\Omega}_{\infty_i}^r \rightarrow \GL(Y_{\Lcal})\backslash \Omega_{\infty}^{nr}$ be the composition of the above maps.
    The cycle $\Zfk_{\Lcal}(\boz)$ collects all images of the points $(\tilde{z}^{(1)},\cdots \hspace{-0.03cm}, \tilde{z}^{(m)}) \in \prod_{i=1}^m\widetilde{\Omega}_{\infty_i}^r$ under $\Upsilon$ where $\tilde{z}^{(i)}$ is equivalent to $z^{(i)}$ for $1\leq i \leq m$.
    In particular, $\Zfk_{\Lcal}(\boz)$ can be parametrized by a ``parallelepiped'' region (see Remark~\ref{rem: HC}):
    \[
    \Zfk_{\Lcal}(\boz) = \bigg\{[g_{\circ,r}\cdot \mathbf{z}(\boz;\mathbf{t})] \in \GL(Y_{\Lcal})\backslash \Omega_{\infty}^{nr} \ \bigg|\ \mathbf{t} \in \Tfk(\Ocal)\assign \frac{\RR^m}{\ln(\Ocal^{\times})+d(\RR)}\bigg\},
    \]
    whence its ``volume'' can be given by (see \eqref{eqn: vol-T})
    \[
    \text{vol}\big(\Zfk_{\Lcal}(\boz),d\mathbf{z}\big) = \text{vol}\big(\Tfk(\Ocal),d\mathbf{t}\big) = \frac{n}{n_1\cdots n_m} \cdot \Rcal(\Ocal).
    \]
    
    Similarly, the cycle $\Zfk_{\Lcal}(\boz;\Cfk_{\bow})$ in the second limit formula collects all images of the points in $\prod_{i=1}^m \widetilde{\Omega}_{\infty_i}^r$ equivalent to $\boz$ under the same map as in \eqref{eqn: Z-map}, but with $\GL(Y_{\Lcal})\backslash \Omega_{\infty}^{nr}$ replaced by $\Gamma^{Y_{\Lcal}}(\Cfk_{\bow})\backslash \Omega_{\infty}^{nr}$.

(3) Suppose $K=k$, which implies that $m=n=1$ and $\Ocal=A$.
    Let $Y = Y_{\Lcal}=\Lcal \subset k^r$ and 
    \begin{align*}
    \EE^Y(z,s) & \assign \Vert Y\Vert_A \cdot \sum_{0\neq \lambda \in Y} \frac{\im(z)^s}{\nu_z(\lambda)^{rs}} = (q-1)\cdot \EE_A^{Y}(z,s), \quad \forall z \in \Omega_{\infty}^r, \\
    \EE^Y(z,w,s) & \assign \Vert Y \Vert_A \cdot 
    \sum_{ \lambda \in Y} \frac{\im(z)^s}{\nu_z(\lambda-w)^{rs}} = \EE_A^Y(z,w,s), \quad \forall w \in k^r\setminus Y.
    \end{align*}
    Then our main theorem becomes the ``Berkovich version'' of the one-variable Kronecker limit formula established in \cite[(4.2) and (4.3)]{Wei20} (see Theorem~\ref{thm: KLM-1}):
\begin{align*}
\EE^Y(z,0)=-1 \quad \text{ and } \quad
\frac{\partial}{\partial s} \EE^Y(z,s)\bigg|_{s=0} & = - \ln \eta_r^{Y}(z) \\
& = -\ln \big(\Vert Y \Vert_A \cdot \im(z)\big) - \frac{r}{q^{r\deg \infty}-1}\ln |\Delta^Y(z)|; \\
\EE^Y(z,w,0)=0
\quad \text{ and } \quad 
\frac{\partial}{\partial s} \EE^Y(z,s)\bigg|_{s=0} & = - \frac{r}{q^{r\deg \infty}-1}\ln |\ufk_w^Y(z)| \\
& = - \frac{r}{q^{r\deg \infty}-1} \cdot \ln |\Delta^Y(z)| + r \ln | \kfk_w^Y(z)|,
\end{align*}
where $\Delta^Y$ is the {\it Drinfeld discriminant function on $\Omega_{\infty}^r$} introduced in {\rm Definition~\ref{defn: Delta}} and $\kfk_w^Y$ is the {\it Klein form} associated with $w$ with respect to $Y$ defined in \eqref{eqn: Klein}.

(4) As pointed out in Remark~\ref{rem: covolume}, we may regard the absolute discriminant $\eta^Y_r(z)$ as the ``covolume'' of the period lattice of a(ny) Drinfeld $A$-module with ``unit discriminant'' which is isomorphic to  $\varphi^{Y,z}$, the specialization of the ``universal'' Drinfeld $A$-module of rank $r$ at $z$ (see \eqref{eqn: specialization}).
Also, from Remark~\ref{rem: KLF}~(1), we get
\[
\frac{(\widetilde{\EE}_{\Ocal}^{\Lcal})'(\boz,0)}{\widetilde{\EE}_{\Ocal}^{\Lcal}(\boz,0)}
= \frac{1}{{\rm vol}\big(\Zfk_{\Lcal}(\boz),d\mathbf{z}\big)} \int_{\Zfk_{\Lcal}(\boz)} \ln \eta^{Y_{\Lcal}}_{nr}(\mathbf{z})\  d \mathbf{z}.
\]
In other words,
the logarithmic derivative of $\widetilde{\EE}_{\Ocal}^{\Lcal}(\boz,s)$ at $s=0$ represents the average integral of the ``log-covolumes'' of the period lattices associated with the ``unit'' Drinfeld $A$-modules of rank $nr$ corresponding to the points along the Heegner cycle $\Zfk_{\Lcal}(\boz)$ on $\GL(Y_{\Lcal})\backslash \Omega_{\infty}^{nr}$.
\end{rem}

\begin{Subsubsec}{Historical development in the function field setting}
Previous investigations of the Kronecker limit formula in the function field setting have largely been confined to the one-variable case. The study originated with Gekeler's work, where he established a connection between an ``improper'' Eisenstein series on $\GL_2(k_{\infty})$ over rational function fields and Drinfeld discriminant functions \cite{Gek}. Subsequently, P\'al \cite[Section 4]{Pal} extended Gekeler’s ideas to a special family of rank-$2$ modular units.

For higher rank, Kondo \cite{Kon} and Kondo–Yasuda \cite{K-Y} developed the second limit formula by investigating Jacobi-type Eisenstein series on $\mathrm{GL}_r(\mathbb{A}_k)$, where $\mathbb{A}_k$ is the ad\`ele ring of the global function field $k$. A conceptual link between these (one-variable) Eisenstein series and the automorphic ``mirabolic'' Eisenstein series was then illustrated in \cite{Wei17} for rank 2 and in \cite{Wei20} for higher rank, providing a comprehensive account of the phenomenon in the one-variable setting.

By contrast, the present paper is the first to address the several-variable context. Our main theorem significantly broadens the ``Kronecker term'' phenomenon by offering a period-theoretic interpretation that applies in a more general, several-variable framework.
\end{Subsubsec}

\subsection{Strategy of the proof}

To prove Theorem~\ref{thm: KLF}, we first translate the ``rigid-analytic'' version of the one-variable Kronecker limit formula for arbitrary rank $r$ in \cite[(4.1) and (4.2)]{Wei20} into the ``Berkovich'' setting (see Theorem~\ref{thm: KLM-1}).
For the case of several-variable, the most crucial step is to express the Eisenstein series $\EE_{\Ocal}^{\Lcal}(\boz,s)$ and $\EE_{\Ocal}^{\Lcal}(\boz,\bow,s)$ in question in terms of integrals of the corresponding one-variable Eisenstein series along the Heegner cycles $\Zfk_{\Lcal}(\boz)$ and $\Zfk_{\Lcal}(\boz;\Cfk_{\bow})$, respectively.
Specifically as shown in Theorem~\ref{Key} and \ref{Key-2}, we have:
\begin{align*}
\dfk(\Ocal/A)^{\frac{rs}{2}}\cdot \EE_{\Ocal}^{\Lcal}(\boz,s) &= \frac{(rs)^{m-1}}{q_{\Ocal}-1} \cdot \frac{n_1\cdots n_m}{n}  \cdot \int_{\Tfk(\Ocal)} \EE^{Y_{\Lcal}}\big(g_{\circ,r}\cdot {\bf z}(\boz;{\bf t}), s\big) d \mathbf{t}
,    \\
\dfk(\Ocal/A)^{\frac{rs}{2}} \cdot \EE^{\Lcal}_{\Ocal}(\boz,\bow,s) &=
(rs)^{m-1}\cdot \frac{n_1\cdots n_m}{n} \cdot 
\int_{\Tfk(\Ocal^{(1)}_{\Cfk_{\bow}})}
\EE^{Y_{\Lcal}}\big(g_{\circ,r}\cdot \mathbf{z}(\boz;\mathbf{t}),w,s\big)d \mathbf{t}.
\end{align*}
These integral representation are achieved by the key integration formula of the quantity $\prod_{i=1}^m \big(D(\nu_{z^{(i)}})^s/ \nu_{z^{(i)}}(x^{(i)})^{rs}\big)$
for nonzero vectors $x^{(i)} \in  K_{\infty_i}^{r}$, $1\leq i \leq m$, derived in Corollary~\ref{cor: sev-one}.
Finally, the proof of Theorem~\ref{thm: KLF} is completed by applying the formula in the one-variable case.

\begin{rem}
The Berkovich analytic structure of the Drinfeld period domain provides more ``points'', which assures the existence of a natural inclusion from the Euclidean building $\Xcal^r(k_{\infty})$ into $\Omega_{\infty}^r$.
This enables us to define the Heegner cycle $\Zfk_{\Lcal}(\boz)$ for each $\boz$ when the finite extension $K$ of $k$ is arbitrary.
Moreover, integrating along these cycles, which can be identified with ``parallelepiped'' region in a Euclidean space, surprisingly makes the formula much cleaner, which gives a very concrete ``period interpretation'' of the ``Kronecker terms'' of these several-variable Eisenstein series.

Consequently, when $r=1$, the formula in Theorem~\ref{thm: KLF} can be applied to the derivative of the $L$-functions associated with Dirichlet characters on class groups of $K$, which is shown in the next subsection.
\end{rem}

\subsection{Application to Dirichlet \texorpdfstring{$L$}{L}-functions}

Keep the above notation. The {\it Dedekind--Weil zeta function} of $\Ocal$ is given by
\[
\zeta_{\Ocal}(s) \assign \sum_{\subfrac{\text{invertible ideal}}{\Ifk \subset \Ocal}} \frac{1}{\Vert \Ifk\Vert_{\Ocal}^{s}}, \quad \forall \re(s)>1.
\]
Note that $\prod_{i=1}^m\Omega_{\infty_i}^1$ consists of a single point, say $\boz_K$.
Then we may write (see \eqref{eqn: L-E-1})
\[
\zeta_{\Ocal}(s) = \sum_{[\Ifk] \in \Pic(\Ocal)} \EE^{\Ifk^{-1}}_{\Ocal}(\boz_K,s).
\]
Therefore the limit formula in Theorem~\ref{thm: KLF} implies that (see Corollary~\ref{cor: zeta}):

\begin{cor}\label{cor: app-zeta}
Put $\widetilde{\zeta}_{\Ocal}(s) \assign s^{1-m}\cdot \dfk(\Ocal/A)^{s/2}\cdot \zeta_{\Ocal}(s)$. Then
\[
\widetilde{\zeta}_{\Ocal}(0) = -\frac{\#\Pic(\Ocal) \cdot \Rcal(\Ocal)}{q_{\Ocal}-1}
\quad \text{ and } \quad 
\frac{\partial}{\partial s} \widetilde{\zeta}_{\Ocal}(s)\bigg|_{s=0}
= -\frac{n_1\cdots n_m}{n\cdot (q_{\Ocal}-1)} \sum_{[\Ifk] \in \Pic(\Ocal)} \int_{\Zfk_{\Ifk}(\boz_K)}\ln \eta_n^{Y_{\Ifk}}(\mathbf{z}) d\mathbf{z}.
\]
Consequently, the logarithmic derivative of $\widetilde{\zeta}_{\Ocal}(s)$ at $s=0$ is equal to the following average:
\[
\frac{\widetilde{\zeta}'_{\Ocal}(0)}{\widetilde{\zeta}_{\Ocal}(0)} = \frac{1}{\#\Pic(\Ocal) \cdot {\rm vol}\big(\Tfk(\Ocal),d\mathbf{t}\big)}
\sum_{[\Ifk] \in \Pic(\Ocal)}
\int_{\Tfk(\Ocal)} \ln \eta_n^{Y_{\Ifk}}\big(g_{\circ}\cdot \mathbf{z}(\boz_K;\mathbf{t})\big)d \mathbf{t},
\]
where $g_{\circ}=g_{\circ,1} \in \GL_n(k_{\infty})$ is given in \eqref{eqn: g-beta}.
\end{cor}

Similarly, let $\chi: \Pic(\Ocal)\rightarrow \CC^{\times}$ be a ring class character.
The {\it Dirichlet $L$-function} associated with $\chi$ is given by
\begin{align*}
L_{\Ocal}(s,\chi) & \assign
\sum_{\subfrac{\text{invertible ideal}}{\Ifk \subset \Ocal}} \frac{\chi(\Ifk)}{\Vert \Ifk\Vert_{\Ocal}^{s}}, \quad \forall \re(s)>1 \\
& = \sum_{[\Ifk] \in \Pic(\Ocal)} \chi([\Ifk]) \cdot \EE_{\Ocal}^{\Ifk^{-1}}(\boz_K,s).
\end{align*}
Hence we get (see Corollary~\ref{cor: D-L}):

\begin{cor}\label{cor: app-ring}
Put $\widetilde{L}_{\Ocal}(s,\chi)\assign s^{1-m}\cdot \dfk(\Ocal/A)^{s/2}\cdot \zeta_{\Ocal}(s)$. Suppose $\chi: \Pic(\Ocal)\rightarrow \CC^{\times}$ is non-trivial.
Then
\[
\widetilde{L}_{\Ocal}(0,\chi) = 0
\quad \text{ and } \quad 
\widetilde{L}'_{\Ocal}(0,\chi) = -\frac{n_1\cdots n_m}{n\cdot (q_{\Ocal}-1)}
\sum_{[\Ifk] \in \Pic(\Ocal)}
\chi([\Ifk]) \cdot \int_{\Tfk(\Ocal)} \ln \eta_n^{Y_{\Ifk^{-1}}}\big(g_{\circ}\cdot \mathbf{z}(\boz;\mathbf{t})\big) d \mathbf{t}.
\]
\end{cor}

Finally, let $\Ocal_K$ be the integral closure of $A$ in $K$, and $\Cfk$ be a nonzero proper ideal of $\Ocal_K$ (i.e.\ $\Cfk \subsetneq \Ocal_K$).
Let $\Jcal(\Cfk)$ be the ideal group generated by the ideals of $\Ocal_K$ coprime to $\Cfk$, and $\Pcal_{\Cfk}$ be the subgroup generated by principal ideals $(\alpha)$ with $\alpha \equiv 1 \bmod \Cfk$.
Let ${\rm Cl}(\Ocal_K,\Cfk)\assign \Jcal(\Cfk)/\Pcal_{\Cfk}$, called the {\it ray class group of $\Ocal_K$ for the modulus $\Cfk$}.
For each ray class character $\chi: {\rm Cl}(\Ocal_K,\Cfk)\rightarrow \CC^{\times}$, the {\it Dirichlet $L$-function associated with $\chi$} is
\begin{align*}
L_{\Ocal_K}^{\Cfk}(s,\chi) &\assign \sum_{\subfrac{\Ifk \subset \Ocal_K}{\Ifk \in \Jcal(\Cfk)}} \frac{\chi(\Ifk)}{\Vert \Ifk \Vert_{\Ocal_K}}, \quad \re(s)>1 \\
&= \Vert \Cfk \Vert_{\Ocal_K}^{-s} \cdot \sum_{[\Ifk] \in {\rm Cl}(\Ocal_K,\Cfk)} \chi([\Ifk])
\cdot \EE_{\Ocal_K}^{\Cfk \Ifk^{-1}}(\boz_K,\bow_K,s),
\quad (see \eqref{eqn: ray-L})
\end{align*}
where $\bow_K = 1 \in K$, and the representative ideal $\Ifk$ is chosen to be integral.
Let $w_K \in k^n$ be the element corresponding to $\bow_K$ with respect to the fixed $k$-basis $\beta_{\circ}$.
The second limit formula implies that (see Corollary~\ref{cor: LF-2}):

\begin{cor}\label{cor: app-ray}
Keep the above notation. 
Put $\widetilde{L}_{\Ocal_K}^{\Cfk}(s,\chi) \assign s^{1-m}\cdot \dfk(\Ocal/A)^{s/2}\cdot L_{\Ocal_K}^{\Cfk}(s,\chi)$. 
Then $\widetilde{L}_{\Ocal_K}(0,\chi) = 0$ and 
\[
(\widetilde{L}_{\Ocal_K}^{\Cfk})'(0,\chi) = 
- \frac{n_1\cdots n_m}{q^{n\deg \infty}-1} \cdot 
\sum_{[\Ifk] \in {\rm Cl}(\Ocal_K,\Cfk)}
\chi([\Ifk]) \cdot
\int_{\Tfk(\Ocal_{K,\Cfk}^{(1)})}\ln
\big|\ufk_{w_K}^{Y_{\Cfk \Ifk^{-1}}}\big(g_{\circ}\cdot \mathbf{z}(\boz_K;\mathbf{t})\big)\big|d\mathbf{t}.
\]
\end{cor}

\begin{rem}\label{rem: AI}
Similarly to the one-variable case discussed in \cite{Wei20}, the several-variable Eisenstein series considered here can also be linked to the automorphic mirabolic Eisenstein series.
For completeness, we provide the concrete description of this link in the appendix.
In particular, the resulting limit formulas naturally lend themselves to applications in the study of special values of $L$-functions associated with automorphic cuspidal representations. These formulas could provide deeper geometric insights of these $L$-values through the Berkovich analytic structure on Drinfeld period domains. Furthermore, when combined with techniques such as ``automorphic inductions'' and ``base changes,'' we expect that our theorem would offer a richer period-theoretic perspective of various special  $L$-values.

In this paper, we focus on offering a detailed examination of the period integration aspects of the ``Kronecker terms'' for the several-variable Eisenstein series. Applications to automorphic  $L$-functions will be pursued in a subsequent paper.
\end{rem}

\subsection{Contents}

This paper is organized as follows.

In Section~\ref{sec: GIS}, we begin with a review of the necessary properties of non-archimedean norms on finite-dimensional vector spaces over local fields and the Goldman--Iwahori spaces. This includes the introduction of the lattice discriminant of norms, the product of norms, and the restriction of scalars.

Section~\ref{sec: BDPD} provides a brief overview of Berkovich affine and projective spaces, followed by the Berkovich description of the Drinfeld period domains. The building map is introduced in Section~\ref{sec: build}, along with its natural section and the imaginary part of points on Drinfeld period domains.

In Section~\ref{sec: DM-DPD}, we define the ``universal'' Drinfeld $A$-module of rank $r$ over $\Omega_{\infty}^r$ and introduce the Klein form and Drinfeld discriminant functions. When specializing this Drinfeld module at a point on $\Omega_{\infty}^r$, we introduce the associated period lattice and the absolute discriminant.

Section~\ref{sec: KLF-1} revisits the one-variable Kronecker-type (first and second) limit formulas in the ``Berkovich'' setting. The whole proof is provided for the sake of completeness. A key new contribution here is the introduction of a theta series on the Goldman--Iwahori spaces, which allows us to express the non-holomorphic Eisenstein series \( \EE^Y(z, s) \) as a Mellin transform of this theta series, demonstrating its absolute convergence.

In Section~\ref{sec: KLF-s}, we extend the discussion to Eisenstein series in several variables. By expressing these several-variable Eisenstein series as suitable integrals of the corresponding one-variable series, we derive Theorem~\ref{thm: KLF}.

Section~\ref{sec: KT-app} explores applications of our limit formulas to Dirichlet  $L$-functions. Corollary~\ref{cor: app-zeta}, \ref{cor: app-ring}, and \ref{cor: app-ray} are presented in Sections~\ref{sec: app-zeta}, \ref{sec: app-ring}, and \ref{sec: app-ray}, respectively.

Finally, Appendix~\ref{sec: Appendix} provides an ``automorphic'' interpretation of our several-variable Eisenstein series in terms of mirabolic Eisenstein series on general linear groups.

\subsection*{Acknowledgments}
The author expresses his deepest gratitude to Jing Yu for his invaluable feedback and steady encouragement. He also extends his sincere thanks to Shih-Yu Chen for many insightful discussions.
The work is supported by the NSTC grant 109-2115-M-007-017-MY5, 113-2628-M-007-003, and the National Center for Theoretical Sciences.

\section{Goldman--Iwahori space}\label{sec: GIS}

This section provides a brief review of the Goldman--Iwahori spaces introduced in \cite{G-I} and the properties of non-archimedean norms to be used.

\subsection{Definition and basic properties}
Let $(F, |\cdot|_F)$ be a non-archimedean local field.
Set $O_F \assign \big\{a \in F \ \big|\  |a|_F \leq 1\big\}$ and $\pfk_F \assign \big\{a \in F \ \big|\ |a|_F < 1\big\}$.
Let $\FF_F \assign O_F/\pfk_F$ and $q_F \assign \#(\FF_F)$.
We fix a generator $\pi_F$ of $\pfk_F$ once and for all, and normalize the absolute value satisfying that $|\pi_F|_F = q_F^{-1}$.

Recall that a \textit{norm} on a finite dimensional vector space $V$ over $F$ is a map $\nu : V\rightarrow \RR_{\geq 0}$ satisfying:
\begin{itemize}
\item $\nu(u) = 0$ if and only if $u = 0$;
\item $\nu(au) = |a|_F\cdot \nu(u)$ for every $a \in F$;
\item $\nu(u_1+u_2) \leq \max\big(\nu(u_1),\nu(u_2)\big)$.
\end{itemize}
Let $\Ncal(V)$ be the set of all the norms on $V$, called the \textit{Goldman--Iwahori space associated with $V$}.

Suppose $\dim_F(V)=r$.
Given $\nu \in \Ncal(V)$,
a basis $\beta = \{u_1,...,u_r\}$ of $V$ is called \textit{orthogonal} with respect to $\nu$ if
$$ \nu(a_1 u_1+\cdots +a_r u_r) = \max\big(\nu(a_i u_i) \mid 1\leq i \leq r\big), \quad \forall a_1,...,a_r \in F.$$

\begin{prop}\label{prop: e-ob}
Let $V$ be a finite dimensional vector space over $F$, and let $L$ be an $O_F$-lattice of $V$, i.e.\ $L$ is a free $O_F$-submodule of $V$ with full rank.
For each $\nu \in \Ncal(V)$, there exists an $O_F$-base $\beta$ of $L$ which is orthogonal with respect to $\nu$.
\end{prop}

This is asserted in \cite[Lemma~(4.2) and Remark~(1) above]{Tag}.
Here we provide a proof for the sake of completeness. First, we define the following invariants depending on the given $\nu$ and $L$: for each $\epsilon>0$, let
\[
B_\nu(\epsilon)\assign\{u \in V \mid \nu(u)\leq \epsilon\},
\]
which is an $O_F$-lattice of $V$.
Let $r = \dim_F(V)$. For each integer $i$ with $1\leq i \leq r$,
let $\epsilon_{L,i}$ be the minimal positive real number so that $B_\nu (\epsilon_{L,i}) \cap L$ contains a direct summand $L_i$ of $L$ with $\rank_{O_F}(L_i)=i$.
The existence of $\epsilon_{L,i}$ follows from the finite images of all primitive elements in $L$ under $\nu$. 
Moreover, we have that:

\begin{lem}\label{lem: e-i}
Keep the above notation.
Then
\[
0<\min\big(\nu(u)\ \big|\ \text{\rm primitive } u \in L\big) =\epsilon_{L,1}\leq \cdots \leq \epsilon_{L,r} = \max\big(\nu(u)\ \big|\ u \in L\big),
\]
and $B_\nu(\epsilon_{L,1})\subset L$.
\end{lem}

\begin{proof}
The first statement is straightforward.
To show the second statement, take a nonzero $u \in B_{\nu}(\epsilon_{L,1})$, and let $\ell \in \ZZ$ which is minimal so that $\pi_F^\ell u \in L$.
Then $\pi_F^\ell u$ is primitive in $L$, for which we get
\[
\epsilon_{L,1}\leq \nu(\pi_F^\ell u) \leq q^{-\ell} \cdot \epsilon_{L,1}.
\]
Hence $\ell \leq 0$, and $u = \pi_F^{-\ell} \cdot (\pi_F^{\ell}u) \in L$.
This implies that $B_{\nu}(\epsilon_{L,1})\subset L$.
\end{proof}

\subsubsection*{Proof of Proposition~\ref{prop: e-ob}}
Let $L'$ be a direct summand of $L$ contained in $B_{\nu}(\epsilon_{L,1})$ which has maximal rank, and $i_1 = \rank_{O_F}(L')$.
Then $\nu(u) = \epsilon_{L,1}$ for every primitive $u \in L'$.
Moreover, let $\{u_1,...,u_{i_1}\}$ be an $O_F$-base of $L'$.
Then for every $a_1,...,a_{i_1} \in F$,
we get that
\[
\nu(a_1u_1+\cdots+a_{i_1}u_{i_1})
= \epsilon_{L,1}\cdot \max(|a_1|_F,...,|a_{i_1}|_F)
=\max\big(\nu(a_1u_1),...,\nu(a_{i_1}u_{i_1})\big).
\]
Write $L = L'\oplus L''$, where $L''$ is a free $O_F$-module of rank $r-i_1$.
We claim that 
\[
\nu(u'+u'') = \max\big(\nu(u'),\nu(u'')\big), \quad \forall u' \in L',\ u'' \in L''.
\]
Then we may apply the induction process on $L''$ to extend $\{u_1,...,u_{i_1}\}$ to a base of $L$ which is orthogonal with respect to $\nu$.

Suppose there exist $u' \in L'$ and $u''\in L''$ such that $\nu(u'+u'') < \max\big(\nu(u'),\nu(u'')\big)$.
Then $\nu(u'')=\nu(u')\leq \epsilon_{L,1}$.
Without loss of generality, we may assume that $u'+u''$ is primitive in $L$.
If $u''$ is primitive in $L''$, then
$L'\oplus O_F u'' \subset B(\epsilon_{L,1})$ is a direct summand of $L$ with rank $i_1+1>i_i$, which is impossible.
Thus $u''$ is not primitive in $L''$, which implies that $u'$ must be primitive in $L'$.
Hence $\nu(u'')=\nu(u') = \epsilon_{L,1}$.
This says that $u'+u''$ is a primitive element of $L$ satisfying that
\[
\nu(u'+u'')<\max\big(\nu(u'),\nu(u'')\big) = \epsilon_{L,1},
\]
which contradicts to that $\epsilon_{L,1} = \min\big(\nu(u)\ \big|\ \text{primitive }u\in L\big)$
in Lemma~\ref{lem: e-i}.
Therefore the claim holds, and the proof is complete.
\hfill $\Box$\\
${}$

Set
\[
D_\nu(L)\assign \prod_{i=1}^r \epsilon_{L,i},
\]
called the {\it lattice discriminant of $L$ with respect to $\nu$}.
We may regard this quantity as the ``volume'' of the lattice $L$ with respect to $\nu$ by the following:

\begin{lem}\label{lem: LD}
Keep the above notations.
Let $\beta=\{u_1,...,u_r\}$ be an $O_F$-base of $L$ which is orthogonal with respect to $\nu$ and $\nu(u_1)\leq \cdots \leq \nu(u_r)$.
Then $\nu(u_i) = \epsilon_{L,i}$ for $1\leq i \leq r$, whence $\prod_{i=1}^r \nu(u_i) = D_\nu(L)$.
\end{lem}

\begin{proof}
It is clear that $\nu(u_i)\geq \epsilon_{L,i}$ for $1\leq i \leq r$.
Suppose $\nu(u_{i_0})>\epsilon_{L,i_0}$ for some $i_0$.
We may assume $i_0$ to be minimal, i.e.\ $\nu(u_i)=\epsilon_{L,i}$ if $i<i_0$.
Then the orthogonality of $\beta$ implies that
\[
L\cap B_\nu(\epsilon_{L,i_0}) \subset (O_F u_1 \oplus \cdots \oplus O_F u_{i_0-1}) \oplus (\pfk_F u_{i_0} \oplus \cdots \oplus \pfk_F u_{r}).
\]
Hence the rank of every direct summand of $L$ contained in $L\cap B_\nu(\epsilon_{L,i_0})$ it at most $i_0-1$, which is a contradiction.
Therefore
\[
\nu(u_i)=\epsilon_{L,i} \quad 1\leq i \leq r.
\]
\end{proof}

\begin{defn}\label{defn: LD}
When $V=F^r$, we put $\Ncal^r(F)\assign \Ncal(F^r)$.
Moreover, for each $\nu \in \Ncal^r(F)$, set $D(\nu) = D_\nu(O_F^r)$, called the {\it lattice discriminant of $\nu$}.
\end{defn}
There is a natural left action of $\GL_r(F)$ on $\Ncal^r(F)$: for $g \in \GL_r(F)$ and $\nu \in \Ncal^r(F)$,
\begin{equation}\label{eqn: action on norms}
(g * \nu) (x) \assign \nu (x g), \quad \forall x \in F^r \text{ (viewing as row vectors)}.
\end{equation}
The lattice discriminant satisfies the following transformation law:

\begin{lem}\label{lem: D-trans}
Given $g \in \GL_r(F)$ and $\nu \in \Ncal^r(F)$, one has
$$ D(g * \nu) = |\det g|_F\cdot D(\nu).$$
\end{lem}

\begin{proof}
This is derived in \cite[Proposition (4.4)]{Tag}.
We provide a proof here for the sake of completeness.
Let $\{u_1,...,u_r\}$ be an orthogonal $O_F$-base of $O_F^r$ with respect to $\nu$, and suppose that $\nu(u_1)\leq \cdots \leq \nu(u_r)$.
For $1\leq i \leq r$, put $\epsilon_i \assign \nu(u_i)$ and let $\delta_i \assign \min\big(\delta \in \ZZ\ \big| \pi_F^{\delta_i}u_i \in B_\nu(1)\big)$.
Then
\[
B_{\nu}(1) = O_F \pi_F^{\delta_1}u_1 \oplus \cdots \oplus O_F \pi_F^{\delta_r}u_r.
\]
On the other hand, let $\{u_1',...,u_r'\}$ be an orthogonal $O_F$-base of $O_F^r$ with respect to $g*\nu$ and 
$(g*\nu)(u_1')\leq \cdots \leq (g*\nu)(u_r')$.
Then
$\{u_1'g,...,u_r'g\}$ is an orthogonal $O_F$-base of $O_F^rg$ with respect to $\nu$.
In particular, put $\epsilon_i'\assign (g*\nu)(u_i')$ and 
$\delta_i'\assign \min\big(\delta \in \ZZ\ \big| \pi_F^\delta u_i'g \in B_\nu(1)\big)$ for $1\leq i \leq r$.
We get another expression of $B_\nu(1)$:
\[
B_\nu(1) = O_F \pi_F^{\delta_1'}u_1'g\oplus \cdots \oplus O_F \pi_F^{\delta_r'}u_r'g.
\]
Regarding $u_i,u_i'$, $1\leq i \leq r$, as row vectors in $F^r$, the above two expressions of $B_\nu(1)$ implies that there exists $\kappa \in \GL_r(O_F)$ satisfying that
\[
\kappa \cdot \begin{pmatrix}
    \pi_F^{\delta_1'} & & \\ & \ddots & \\
    & & \pi_F^{\delta_r'}
\end{pmatrix}
\cdot \begin{pmatrix}
    u_1' \\
    \vdots \\
    u_r'
\end{pmatrix}
\cdot g
= \begin{pmatrix}
    \pi_F^{\delta_1} & & \\ & \ddots & \\
    & & \pi_F^{\delta_r}
\end{pmatrix}
\cdot \begin{pmatrix}
    u_1 \\
    \vdots \\
    u_r
\end{pmatrix} \quad \in \GL_r(F).
\]
In particular, one has $|\det g|_F = \prod_{1\leq i \leq r} |\pi_F|_F^{\delta_i-\delta_i'}$.

Finally, Lemma~\ref{lem: LD} assures that
\begin{align*}
D(g*\nu) = \prod_{i=1}^r \nu(u_i'g) & = \prod_{i=1}^r |\pi_F|_F^{-\delta_i'}\nu(\pi^{\delta_i'}u_i'g) = \Big(\prod_{i=1}^r|\pi_F|_F^{-\delta_i'}\Big) \cdot D_\nu\big(B_\nu(1)\big) \\
& = \Big(\prod_{i=1}^r|\pi_F|_F^{-\delta_i'}\Big) \cdot
\Big(\prod_{i=1}^r \nu(\pi_F^{\delta_i}u_i)\Big)
= \prod_{i=1}^r |\pi_F|_F^{\delta_i-\delta_i'} \cdot \nu(u_i)
= |\det g|_F \cdot D(\nu),
\end{align*}
whence completes the proof.
\end{proof}

\begin{rem}\label{rem: build}
${}$
\begin{itemize}
\item[(1)] 
Given a norm $\nu \in \Ncal^r(F)$ and an orthogonal $O_F$-base $\beta = \{u_1,...,u_r\}$ of $O_F^r$ with respect to $\nu$ satisfying $\nu(u_1)\leq \cdots \leq \nu(u_r)$, there exists an open subgroup $\Kcal$ of $\GL_r(O_F)$ so that $\{u_1 \kappa, ... , u_r \kappa\}$ is also an orthogonal $O_F$-base of $O_F^r$ with respect to $\nu$ for every $\kappa \in \Kcal$ (see \cite[Proposition~2.28]{RTW}).
\item[(2)] Two norms $\nu_1$ and $\nu_2$ in $\Ncal^r(F)$ are said to be equivalent (denoted by $\nu_1 \sim \nu_2$) if $\nu_1 = c  \nu_2$ for some $c \in \RR_{>0}$.
Denote by $\Xcal^r(F) \assign \Ncal^r(F)/\!\!\sim$ the set of equivalence classes of all the norms on $F^r$.
It is known that $\Xcal^r(F)$ can be identified with the realization of the Bruhat--Tits building associated with $\PGL_r(F)$, see \cite[Theorem 2.22]{RTW}.
\end{itemize}
\end{rem}

\subsection{Integral formula of the product of norms}\label{sec: prod-norm}

let $r_1,...,r_m$ be positive integers. Given $\nu = (\nu^{(1)}, \cdots \hspace{-0.03cm} ,\nu^{(m)}) \in \prod_{i=1}^m \Ncal^{r_i}(F)$ and $\mathbf{t} = (t_1,\cdots \hspace{-0.03cm},t_m) \in \RR^m$,
define a norm $\nu(\mathbf{t})$ on $\prod_{i=1}^m F^{r_i}$ by
$$ \nu(\mathbf{t})(x_1,\cdots \hspace{-0.03cm},x_m) \assign \max\Big(\exp(t_i) \nu^{(i)}(x_i)\ \Big|\  1\leq i \leq m\Big),
\quad \forall (x_1,\cdots \hspace{-0.03cm},x_m) \in \prod_{i=1}^m F^{r_i}.$$
Here $\exp(t) = \sum_{n=0}^\infty t^n/n!$ is the usual exponential function on $\RR$.
For $\mathbf{t} = (t_1,\cdots \hspace{-0.03cm},t_m) \in \RR^m$, the above construction gives an embedding from $\prod_{i=1}^m \Ncal^{r_i}(F)$ to $\Ncal^r(F)$, where $r = r_1+\cdots +r_m$.
Note that the identification $\prod_{i=1}^m F^{r_i} = F^r$ also induces an embedding from $\prod_{i=1}^m \GL_{r_i}(F)$ into $\GL_r(F)$.
It can be checked that the embedding $\big(\nu \mapsto \nu(\mathbf{t})\big)$
from $\prod_{i=1}^m \Ncal^{r_i}(F)$ to $\Ncal^r(F)$ is equivariant under the action of $\prod_{i=1}^m \GL_{r_i}(F)$ for each $\mathbf{t} \in \RR^m$.
Moreover, one has that
$$D\big(\nu(\mathbf{t})\big) = \prod_{i=1}^m \exp(r_i t_i) D(\nu^{(i)}), \quad \forall \mathbf{t} = (t_1,\cdots \hspace{-0.03cm},t_m) \in \RR^m.$$

In the following, we derive an integral formula for a particular product of norms $\nu^{(1)}$,...,$\nu^{(m)}$, which is a key technical aspect of this paper.

\begin{prop}\label{prop: integral-form}
Let $r_1,...,r_m$ be positive integers and put $r = r_1+\cdots+r_m$.
Given a tuple of norms $\nu = (\nu^{(1)},\cdots \hspace{-0.03cm},\nu^{(m)}) \in \prod_{i=1}^m \Ncal^{r_i}(F)$ and $s \in \CC$ with $\re(s)>0$,
we have that for every $(x_1,\cdots \hspace{-0.03cm},x_m) \in \prod_{i=1}^mF^{r_i} = F^r$ with $\nu^{(i)}(x_i) \neq 0$ for $1\leq i \leq m$, the following equality holds:
$$
\int_{\RR^m/d(\RR)} \frac{D\big(\nu(\mathbf{t})\big)^s}{\nu(\mathbf{t})(x_1,\cdots \hspace{-0.03cm},x_m)^{rs}} d \mathbf{t}
= s^{1-m} \cdot \frac{r}{r_1\cdots r_m} \cdot \prod_{i=1}^m \frac{D(\nu^{(i)})^s}{\nu^{(i)}(x_i)^{r_is}}
$$
Here
$d(\RR) \assign \{(t,\cdots \hspace{-0.03cm} ,t) \in \RR^m \mid t \in \RR\}$,
and the measure $d\mathbf{t}$ on the quotient space $\RR^m/d(\RR)$ is induced from the standard Lebesgue measures on $\RR^m$ and $\RR$.
\end{prop}

\begin{proof}
We shall prove the result by induction on $m$. 
The case when $m=1$ holds clearly.
Suppose $m\geq 2$ and the statement holds for $m-1$.
For $\mathbf{t}' = (t_2,\cdots \hspace{-0.03cm},t_m) \in \RR^{m-1}$, let
$$\nu'(\mathbf{t}')(x_2,\cdots \hspace{-0.03cm},x_m) \assign \max\Big(\exp(t_i)\nu^{(i)}(x_i) \mid 2\leq i \leq m\Big), \quad \forall (x_2,\cdots \hspace{-0.03cm},x_m) \in \prod_{i=2}^m F^{r_i}.$$
Put $r' \assign r_2+\cdots + r_m$.
Then for $(x_1,\cdots \hspace{-0.03cm},x_m) \in \prod_{i=1}^m F^{r_i}$ with $\nu^{(i)}(x_i) \neq 0$, $1\leq i \leq m$, the identification $\big(\mathbf{t}'\leftrightarrow (0,\mathbf{t}')\big)$ between $\RR^{m-1}$ and $\RR^m/d(\RR)$ gives us that
\begin{eqnarray}
&&\int_{\RR^m/d(\RR)} \frac{D\big(\nu(\mathbf{t})\big)^s}{\nu(\mathbf{t})(x_1,\cdots \hspace{-0.03cm},x_m)^{rs}} d \mathbf{t} 
\ = \ 
\int_{\RR^{m-1}}\frac{D(\nu^{(1)})^s \cdot D\big(\nu'(\mathbf{t}')\big)^s}{\max\big(\nu^{(1)}(x_1),\nu'(\mathbf{t}')(x')\big)^{rs}} d\mathbf{t}'
\nonumber\\
&=& D(\nu^{(1)})^s\cdot \int_{\RR^{m-1}/d(\RR)} \left(\int_{\RR} \max\Big(\nu^{(1)}(x_1), \exp(t)\nu'(\mathbf{t}')(x')\Big)^{-rs} e^{r'ts} dt\right)  D\big(\nu'(\mathbf{t}')\big)^s d\mathbf{t}' ,\nonumber
\end{eqnarray}
where $x' \assign (x_2,\cdots \hspace{-0.03cm},x_m) \in \prod_{i=2}^m F^{r_i} = F^{r'}$.
Put $c = \ln \nu^{(1)}(x_1) - \ln \nu'(\mathbf{t}')(x')$. 
For $\re(s)>0$, we get that
\begin{eqnarray}
&&\hspace{-3cm}  \int_{\RR} \max\Big(\nu^{(1)}(x_1), e^t\nu'(\mathbf{t}')(x')\Big)^{-rs} e^{r'ts} dt \nonumber \\
&=& \int_c^\infty \nu'(\mathbf{t}')(x')^{-rs} e^{-r_1 ts} dt + \int_{-\infty}^c \nu^{(1)}(x_1)^{-rs} e^{r'ts} dt \nonumber  \\
&=& \nu^{(1)}(x_1)^{-r_1s} \nu'(\mathbf{t}')(x')^{-r' s} \left(\frac{1}{r_1s} + \frac{1}{r's}\right) \nonumber \\
&=& s^{-1}\cdot \frac{r}{r_1\cdot r'} \cdot \frac{1}{\nu^{(1)}(x_1)^{r_1s} \cdot \nu'(\mathbf{t}')(x')^{r' s}}. \nonumber
\end{eqnarray}
Hence
\begin{eqnarray}
&& \hspace{-4cm}\int_{\RR^m/d(\RR)} \frac{D\big(\nu(\mathbf{t})\big)^s}{\nu(\mathbf{t})(x_1,\cdots \hspace{-0.03cm},x_m)^{rs}} d \mathbf{t} \nonumber\\
&=&s^{-1}\cdot \frac{r}{r_1\cdot r'} \cdot \frac{D(\nu^{(1)})^s}{\nu^{(1)}(x_1)^{r_1s}} \cdot \int_{\RR^{m-1}/d(\RR)}
\frac{D\big(\nu'(\mathbf{t}')\big)^s}{\nu'(\mathbf{t}')(x')^{r's}} d \mathbf{t}' \nonumber \\
\text{(induction hypothesis)} \quad &=& s^{-1} \cdot \frac{r}{r_1\cdot r'}\cdot \frac{D(\nu^{(1)})^s}{\nu^{(1)}(x_1)^{r_1s}} \cdot
\left(s^{2-m} \cdot \frac{r'}{r_2\cdots r_m} \cdot \prod_{i=2}^m \frac{D(\nu^{(i)})^s}{\nu^{(i)}(x_i)^{r_is}}\right) \nonumber \\
&=& s^{1-m} \cdot\frac{r}{r_1\cdots r_m} \cdot \prod_{i=1}^m \frac{D(\nu^{(i)})^s}{\nu^{(i)}(x_i)^{r_is}}. \nonumber 
\end{eqnarray}
Therefore the proof is complete by induction.
\end{proof}

\subsection{Restriction of scalars}

Let $E$ be a finite extension over $F$.
The absolute value of $F$ extends uniquely to $E$, which is still denoted by $|\cdot|_F$ (in order to distinguish with the normalized one $|\cdot|_E$ on $E$). In particular, let $N_{E/F}:E\rightarrow F$ be the norm map from $E$ to $F$. Then
\[
|\alpha|_E = |N_{E/F}(\alpha)|_F, \quad \forall \alpha \in E.
\]

Put $n = [E:F]$ and $\dfk(E/F) \assign |\pi_F|_F^{f(e-1)}$, where $e$ and $f$ are the ramification index and the residue degree of $E/F$, respectively.
Note that $\pi_F^{f(e-1)} O_F$ is equal to the discriminant ideal of $O_E/O_F$ when $E/F$ is tamely ramified.
Fix a uniformizer $\pi_E$ of $E$ and an element $\xi \in O_E$ so that $\FF_E = \FF_F(\bar{\xi})$, where $\bar{\xi}$ is the image of $\xi$ in the residue field $\FF_E$.
Then
\begin{equation}\label{eqn: E-F-id}
O_E = \bigoplus_{0\leq i < f}\bigoplus_{0\leq j < e} O_F \cdot \xi^i \pi_E^j
\end{equation}
\begin{align}\label{eqn: res-ind}
\quad \text{ and } \quad 
\left| \sum_{0\leq i < f}\sum_{0\leq j < e} a_{ij}\xi^i \pi_E^j\right|_F & = \max\left(\big|a_{ij}\xi^i \pi_E^j\big|_F \ \Big|\ 0\leq i < f,\ 0\leq j < e \right) \notag \\
& =
\max\left(|a_{ij}|_F\cdot q_F^{j/e} \ \Big|\ 0\leq i < f,\ 0\leq j < e \right)
, \quad \forall a_{ij} \in F.
\end{align}
Thus, the orthogonal $O_F$-base $\big\{\xi^i\pi_E^j\ \big|\ 0\leq i<f,\ 0\leq j < e\big\}$ of $O_E$ gives an isomorphism
\begin{equation}\label{eqn: E-F-iso}
\text{$\iota_{E/F}: E \stackrel{\sim}{\rightarrow} F^n$ 
which extends to an isomorphism $\iota_{E/F,r}: E^r \stackrel{\sim}{\rightarrow} F^{nr}$ for every $r \in \NN$.}    
\end{equation}
This isomorphism
induces an embedding 
\begin{equation}\label{eqn: res}
\text{res}: \Ncal^r(E)\hookrightarrow \Ncal^{nr}(F),
\end{equation}
where
for every norm $\nu_E$ in $ \Ncal^r(E)$ and 
$(a_{ij\ell})_{0\leq i<f,0\leq j<e,1\leq \ell \leq r} \in F^{nr}$,
the norm $\text{res}(\nu_E)$ is defined by
\[
\text{res}(\nu_E)\Big((a_{ij\ell})_{0\leq i<f,0\leq j<e,1\leq \ell \leq r}\Big)
\assign
\nu_E
\left(\sum_{0\leq i < f}\sum_{0\leq j < e} a_{ij1}\xi^i \pi_E^j, \cdots, \sum_{0\leq i < f}\sum_{0\leq j < e} a_{ijr}\xi^i \pi_E^j\right)^{1/n}.
\]

As $\nu_E(c\cdot x) = |c|_E \cdot \nu_E(x) = |c|_F^n \cdot  \nu_E(x)$ for every $c \in F$ and $x \in E^r$, the exponent $1/n$ in the above definition makes $\text{res}(\nu_E)$ a well-defined norm on $F^{nr}$.
We call $\res(\nu_E)$ the \emph{restriction of scalars of $\nu_E$} (with respect to the fixed pair $\xi$ and $\pi_E$).

\begin{rem}\label{rem: ind-base}
In the archimedean setting, $\CC$ has a natural $\RR$-basis $\{1,\sqrt{-1}\}$, providing a canonical expression as $\RR^2$.
However, in the nonarchimedean case, no canonical identification of $E$ with $F^n$ when $n>1$.
Therefore for introducing the restriction of scalars, it is necessary to fix specific $\xi$ and $\pi_E$ first for the isomorphism $\iota_{E/F,r}: E^r \stackrel{\sim}{\rightarrow} F^{nr}$ for every $r \in \NN$.
In particular, suppose another $\xi'$ and $\pi_E'$ are chosen.
Write 
\begin{equation}\label{eqn: u-i}
\big\{\xi^i\pi_E^j\ \big|\ 0\leq i <f,\ 0\leq j < e\big\} = \big\{u_{t}\ \big|\ 1\leq t \leq n\big \}
\end{equation}
and
\[
\big\{(\xi')^i(\pi_E')^j\ \big|\ 0\leq i <f,\ 0\leq j < e\big \} = \big \{u'_{t}\ \big|\ 1\leq t \leq n\big\}.
\]
There exists $\kappa \in \GL_{n}(O_{F})$ (in fact, within a parahori subgroup) so that
\[
\begin{pmatrix} u_1' \\ \vdots \\ u_{n}' \end{pmatrix}
=
\kappa \cdot \begin{pmatrix} u_1 \\ \vdots \\ u_{n} \end{pmatrix}.
\]
Put $\kappa_r \assign {\rm diag}(\kappa,\cdots \hspace{-0.03cm},\kappa) \in \GL_{nr}(O_{F})$.
Let $\res':\Ncal^r(E) \rightarrow \Ncal^{nr}(F)$ be the restriction of scalars associated with the pair $(\xi',\pi_E')$.
Then we can see that
\[
\res'(\nu_E) = \kappa_r * \res(\nu_E), \quad \forall \nu_E \in \Ncal^r(E).
\]
In other words, the restriction of scalars $\res: \Ncal^r(E) \rightarrow \Ncal^{nr}(F)$ depends on the chosen $\xi$ and $\pi_E$ when $r>1$.
When $r=1$, one observes that $\Ncal^1(E) = \big\{c\, |\cdot|_E\ \big|\ c \in \RR_{>0}\big\}$ and $\res'(|\cdot|_E) = \res(|\cdot|_E)$ (by the formula \eqref{eqn: res-ind}).
Hence the dependency disappears.
\end{rem}

\begin{lem}\label{lem: lattice-dis}
Given $\nu_E \in \Ncal^r(E)$, the following equality holds:
$$D\big(\text{\rm res}(\nu_E)\big) = \dfk(E/F)^{r/2} \cdot D(\nu_E).$$
\end{lem}

\begin{proof}
Let $\{u_1,...,u_n\}$ be the $F$-basis of $E$ chosen in \eqref{eqn: u-i}.
For each $b \in E$, let $\tilde{b} \in \Mat_n(F)$ so that
\[
\begin{pmatrix}
     u_1 b \\ \vdots \\ u_n b
\end{pmatrix}
= \tilde{b} \cdot \begin{pmatrix} u_1 \\ \vdots \\ u_n \end{pmatrix}. 
\]
This gives an injective homomorphism (with respect to the fixed $\iota_{E/F}$)
\[
\GL_r(E) \hookrightarrow \GL_{nr}(F), \quad g = 
\begin{pmatrix} 
b_{11} & \cdots & b_{1r} \\
\vdots & & \vdots \\
b_{r1} & \cdots & b_{rr}
\end{pmatrix}
\longmapsto 
\tilde{g} \assign 
\begin{pmatrix} 
\tilde{b}_{11} & \cdots & \tilde{b}_{1r} \\
\vdots & & \vdots \\
\tilde{b}_{r1} & \cdots & \tilde{b}_{rr}
\end{pmatrix},
\]
which satisfies that
\[
N_{E/F}(\det g) = \det \tilde{g}, \quad 
\text{ where $N_{E/F}:E\rightarrow F$ is the norm map of $E$ over $F$.}
\]
Moreover, one can check that
\[
\res(g*\nu_E) = \tilde{g}*\res(\nu_E), \quad \forall \nu_E \in \Ncal^r(E) \text{ and } g \in \GL_r(E).
\]

Now, let $\{v_1,...,v_r\}$ be the standard basis of $E^r$.
Then there exists $g_0 \in \GL_r(E)$ so that $\{v_1,...,v_r\}$ is an orthogonal $O_E$-base of $O_E^r$ with respect to $g_0*\nu_E$.
Then with respect to the notation in~\eqref{eqn: u-i}, we know that
\[
\big\{u_t v_{\ell}\ \big|\ 1\leq t \leq n,\ 1\leq \ell \leq r\big\}
\ =\ 
\big\{\xi^i \pi_E^j v_\ell \ \big|\ 0\leq i < f,\ 0\leq j < e,\ 1\leq \ell \leq r\big\}
\]
corresponds to an orthogonal $O_F$-base of $O_F^{nr}$ with respect to $\text{res}(g_0*\nu_E)$.
Thus,
\begin{eqnarray*}
D\big(\text{res}(\nu_E)\big) 
&= & |\det \tilde{g}_0|_F^{-1} \cdot D\big(\tilde{g}_0 * \res(\nu_E)\big)
\ = \ |N_{E/F}(\det g_0)|_F^{-1} \cdot D \big(\res(g_0*\nu_E)\big) \\
&= & |\det g_0|_E^{-1} \cdot  \prod_{i=0}^{f-1} \prod_{j=0}^{e-1} \prod_{\ell=1}^r (g_0*\nu_E)\big(\xi^i\pi_E^j v_\ell\big)^{\frac{1}{n}}\\
&=& |\det g_0|_E^{-1} \cdot |\pi_E|_F^{\frac{ef(e-1)r}{2}} \cdot D(g_0*\nu_E)  \\
&=& |\pi_F|_F^{\frac{f(e-1)r}{2}} \cdot D(\nu_E) \ = \  \dfk(E/F)^{r/2} \cdot D(\nu_E), 
\end{eqnarray*}
whence the proof is complete.
\end{proof}

Combining with Proposition~\ref{prop: integral-form}, we obtain that:

\begin{cor}\label{cor: sev-one}
Given a positive integer $r$ and finite extensions $E^{(i)}$ over $F$ with degree $n_i$ for $1\leq i\leq m$, let $n = n_1+\cdots + n_m$.
By abuse of notation, we still denote 
\[
\text{\rm res} : \prod_{i=1}^m \Ncal^r(E^{(i)}) \hookrightarrow \prod_{i=1}^m \Ncal^{n_ir}(F) 
\]
to be the embedding induced from the restriction of scalars on each component.
Then for every $\nu_E = (\nu_E^{(1)},\cdots \hspace{-0.03cm},\nu_E^{(m)}) \in \prod_{i=1}^m\Ncal^r(E^{(i)})$
and $(x_1,\cdots \hspace{-0.03cm},x_m) \in \prod_{i=1}^m F^{n_ir}$, the following equality holds for $s \in \CC$ with $\re(s)>0$:
$$
\int_{\RR^m/d(\RR)} \frac{D\big(\text{\rm res}(\nu_E)(\mathbf{t})\big)^s}{\text{\rm res}(\nu_E)(\mathbf{t})(x_1,\cdots \hspace{-0.03cm},x_m)^{nrs}} d \mathbf{t}
= (rs)^{1-m} \cdot \frac{n}{n_1\cdots n_m} \cdot \prod_{i=1}^m \dfk(E^{(i)}/F)^{\frac{rs}{2}} \frac{D(\nu_E^{(i)})^{s}}{\nu_E^{(i)}(x^{(i)})^{rs}}
$$
Here $x^{(i)} \in (E^{(i)})^r$ corresponds to $x_i \in F^{n_ir}$ via the isomorphism $\iota_{E^{(i)}/F,r}: (E^{(i)})^r \cong F^{n_ir}$ as in \eqref{eqn: E-F-iso} for $1\leq i \leq m$, i.e.~$\iota_{E^{(i)}/F,r}(x^{(i)}) = x_i$.
\end{cor}

\begin{proof}
\begin{eqnarray*}
&& \hspace{-3cm} \int_{\RR^m/d(\RR)} \frac{D\big(\text{\rm res}(\nu_E)(\mathbf{t})\big)^s}{\text{\rm res}(\nu_E)(\mathbf{t})(x_1,\cdots \hspace{-0.03cm},x_m)^{nrs}} d \mathbf{t} \\
\text{(Proposition~\ref{prop: integral-form})} \quad  &=& 
s^{1-m}\cdot \frac{nr}{(n_1r)\cdots (n_mr)}
\cdot \prod_{i=1}^m\frac{D\big(\text{res}(\nu_E^{(i)})\big)^s}{\text{res}(\nu_E^{(i)})(x_i)^{n_ir s}} \\
\text{(Lemma~\ref{lem: lattice-dis})}\quad \quad \quad  &=&
(rs)^{1-m}\cdot \frac{n}{n_1\cdots n_m} \cdot 
\prod_{i=1}^m \dfk(E^{(i)}/F)^{\frac{rs}{2}} \frac{D(\nu_E^{(i)})^{s}}{\nu_E^{(i)}(x^{(i)})^{rs}}
\end{eqnarray*}
as desired.
\end{proof}

\section{Berkovich Drinfeld period domain}\label{sec: BDPD}

In this section, we review the definition of Drinfeld period domains as Berkovich spaces and their connection with Goldman--Iwahori spaces through the building maps. For further details, we refer the reader to \cite{Ber1}, \cite{RTW12}, and \cite{RTW}.

\subsection{Berkovich affine and projective spaces}

Let $\CC_F$ be the completion of a chosen algebraic closure of $F$.
Given a positive integer $r$, let $R(r,\CC_F)$ be the symmetric algebra of $\CC_F^r$. We may identify $R(r,\CC_F)$ with the polynomial ring $\CC_F[e_1,\cdots \hspace{-0.03cm},e_r]$, where $\{e_1,...,e_r\}$ is the standard $\CC_F$-basis of $\CC_F^r$.
A \emph{multiplicative seminorm on $R(r,\CC_F)$ over $\CC_F$} is a map $|\cdot|: R(r,\CC_F) \rightarrow \RR_{\geq 0}$
satisfying that
\begin{itemize}
    \item[(1)] $|c| = |c|_F \quad \forall c \in \CC_F$.
    \item[(2)] $|f+g|\leq \max(|f|,|g|) \quad \forall f,g \in R(r,\CC_F)$;
    \item[(3)] $|f\cdot g| = |f|\cdot |g| \quad \forall f,g \in R(r,\CC_F)$.
\end{itemize}
Recall that the {\it Berkovich affine space} $\AA^{r,\text{an}}_{\CC_F}$ consists of all the multiplicative seminorms on $R(r,\CC_F)$ (see \cite[\S 1.5]{Ber90}). 
For each $x \in \AA^{r,\text{an}}_{\CC_F}$, we also write
\[
|f|_x \assign x(f), \quad \forall f \in R(r,\CC_F).
\]
Then $\AA^{r,\text{an}}_{\CC_F}$ is equipped the weakest topology so that the real-valued function $(x \mapsto |f|_x)$ on $\AA^{r,\text{an}}_{\CC_F}$ is continuous for every $f \in R(r,\CC_F)$ .

Put $\wp_x \assign \big\{f \in R(r,\CC_F)\ \big|\ |f|_x = 0\big\}$, which is a prime ideal of $R(r,\CC_F)$.
Then $|\cdot|_x$ induces a multiplicative norm on
$R(r,\CC_F)/\wp_x$ (still denoted by $|\cdot|_x$, see \cite[Definition~2.2.1 (ii)]{Tem14}):
\[
|\bar{f}|_x \assign \inf\big\{|f+g|_x\ \big| \ g \in \wp_x\big\}, \quad \forall \bar{f}=f+\wp_x \in R(r,\CC_F)/\wp_x.
\]
We may extend $| \cdot |_x$ to a non-archimedean absolute value on the fraction field of $R(r,\CC_F)/\wp_x$, and its completion is denoted by $\CC_F(x)$.
The absolute value of every $g(x) \in \CC_F(x)$ is written by $|g(x)|$ for convention.
In particular, for each $f \in R(r,\CC_F)$, its image in $\CC_F(x)$ is denoted by $f(x)$, whence
\[
|f(x)| = |\bar{f}|_x= |f|_x.
\]

Two such seminorms $x_1$ and $x_2$ are equivalent if and only if there exists $c \in \RR_{>0}$ so that $x_1(f) = c^n x_2(f)$ for every $n\in \ZZ_{\geq 0}$ and homogeneous $f \in R(r,\CC_F)$ with $\deg f = n$.
The Berkovich projective space $\PP^{r-1,\text{an}}_{\CC_F}$ is (identified with) the set of equivalence classes of the multiplicative seminorms on $R(r,\CC_F)$ which are nonzero at $e_i$ for some $1\leq i \leq r$, see \cite[\S~3.1.1]{RTW}.
In particular, for every multiplicative seminorm $x$ on $R(r,\CC_{F})$ with $|e_r|_x\neq 0$, there exists a unique multiplicative seminorm $\underline{x}$ on $R(r,\CC_F)$ which is equivalent to $x$ so that 
\[e_r(\underline{x})=1 \quad \in \CC_F(\underline{x}).
\]

Given a (usual) $F$-rational hyperplane $H \subset \PP^{r-1}_F$,
we may write 
$$H = \big\{(c_1:\cdots: c_r) \in \PP^{r-1}_F \ \big|\ a_1c_1+\cdots +a_r c_r = 0 \big\}$$
for a unique $F$-rational point $a = (a_1:\cdots: a_r) \in \PP^{r-1}(F)$.
In this case, $H$ is denoted by $H_a$.
For $z \in \PP^{r-1,\text{an}}_{\CC_F}$, take a multiplicative seminorm $x$ on $R(r,\CC_F)$ representing $z$.
We say that \emph{$z$ lies in $H$} if 
\[
|a_1e_1+\cdots+ a_r e_r|_x = 0.
\]
Put 
$H^{\text{an}} \assign \big\{z \in \PP^{r-1,\text{an}}_{\CC_F}\ \big|\ \text{$z$ lies in $H$}\big\}$, regarded as the analytification of $H$.

\begin{defn}\label{defn: Omega}
(See \cite{Ber1}.) The \emph{Berkovich Drinfeld period domain of rank $r$ over $F$} is defined by
$$\Omega^r_F \assign \PP^{r-1,\text{an}}_{\CC_F} \setminus \bigcup_{a \in \PP^{r-1}(F)} H_a^{\text{an}}.$$
\end{defn}

\begin{rem}\label{rem: Omega-id}
Let 
$$\widetilde{R}(r,\CC_F) \assign R(r,\CC_F)\big[(a_1e_1+\cdots+ a_r e_r)^{-1} \mid (a_1,\cdots \hspace{-0.03cm},a_r) \in F^r-\{0\}\big],$$
and let $\widetilde{\Omega}^r_F$ be the set of all the multiplicative seminorms on $\widetilde{R}(r,\CC_F)$.
The embedding $R(r,\CC_F) \hookrightarrow \widetilde{R}(r,\CC_F)$ induces an injective map $\widetilde{\Omega}^r_F \hookrightarrow \AA^{r,\text{an}}_{\CC_F}$.
Then we may identify $\Omega^r_F$ with the set of equivalence classes in $\widetilde{\Omega}^r_F$.
\end{rem}

Note that the ring homomorphism from $R(r-1,F)$ to $\widetilde{R}(r,\CC_F)$ given by
\[
e_i \longmapsto \frac{e_i}{e_r}, \quad 1\leq i \leq r-1,
\]
induces an embedding $\Omega_F^r \hookrightarrow \AA^{r-1,\text{an}}_{\CC_F}$.
More precisely, we may identify $\Omega_F^r$ with
\begin{equation}\label{eqn: Omega-1}
\left\{
z \in \AA^{r-1,\text{an}}_{\CC_F}\ \Big|\ 
|a_1e_1+\cdots+a_{r-1}e_{r-1}+a_r|_z \neq 0, \quad \forall 0\neq (a_1,\cdots \hspace{-0.03cm},a_r) \in F^r
\right\}.
\end{equation}
In other words, every $z \in \Omega_F^r$ can be regarded as a point in $\AA_{\CC_{F}}^{r,{\rm an}}$ so that $e_r(z)=1$.


The atlas on $\Omega_F^r$ can be chosen as the collections of the affinoid spaces associated with all simplices of $\Xcal^r(F)$ via the ``building map''  introduced in the next subsection.

\subsection{Building map}\label{sec: build}

Given $\tilde{z} \in \widetilde{\Omega}^r_F$, one may associate a norm $\nu_{\tilde{z}}$ on $F^r$ by
\begin{equation}\label{eqn: nu-z}
\nu_{\tilde{z}}(a_1,\cdots \hspace{-0.03cm},a_r) \assign |a_1e_1+\cdots + a_r e_r|_{\tilde{z}}, \quad \forall (a_1,\cdots \hspace{-0.03cm},a_r) \in F^r.
\end{equation}
Define $\tilde{\Bscr}: \widetilde{\Omega}^r_F \rightarrow \Ncal^r(F)$ by sending $\tilde{z}$ to $\nu_{\tilde{z}}$ for every $\tilde{z} \in \widetilde{\Omega}^r_F$.
It is clear that if $\tilde{z}_1,\tilde{z}_2 \in \widetilde{\Omega}^r_F$ are equivalent, then so are $\nu_{\tilde{z}_1}$ and $\nu_{\tilde{z}_2}$.
Thus $\tilde{\Bscr}$ induces a map
$\Bscr: \Omega^r_F \rightarrow \Xcal^r(F)$, the so-called \emph{building map}.

Note that the right action of $\GL_r(F)$ on $F^r$ (by right multiplication when viewing elements in $F^r$ as row vectors) 
extends to a right action of $\GL_r(F)$ on the symmetric algebra $R(r,\CC_F)$ (and also on $\widetilde{R}(r,\CC_F)$).
The corresponding left actions of $\GL_r(F)$ on $\AA^{r,\text{an}}_{\CC_F}$ and $\widetilde{\Omega}^r_F$ can be illustrated as follows: for every $x \in \AA^{r,\text{an}}_{\CC_F}$ (resp.\ $\widetilde{\Omega}^r_F$) and $\gamma \in \GL_r(F)$, the point $\gamma \cdot x$ satisfies
\[
|f|_{\gamma \cdot x} \assign |f \cdot \gamma|_x, \quad \forall f \in R(r,\CC_F) \quad \text{ (resp.\ $\widetilde{R}(r,\CC_F)$)}.
\]
This induces a left action of $\PGL_r(F)$ on $\Omega^r_F$.
On the other hand, recall that we have a left action of $\PGL_r(F)$ on $\Xcal^r(F)$ induced from the left action of $\GL_r(F)$ on $\Ncal^r(F)$ given in \eqref{eqn: action on norms}.
It is known in \cite[(4.2) Proposition]{D-H} that:

\begin{prop}
The map $\widetilde{\Bscr}: \widetilde{\Omega}_F^r \rightarrow \Ncal^r(F)$ is $\GL_r(F)$-equivariant.
Consequently, the building map $\Bscr: \Omega_F^r \rightarrow \Xcal^r(F)$ is $\PGL_r(F)$-equivariant.
\end{prop}

\begin{rem}\label{rem: nu-z}
(see \cite[p.~3]{Ber1}.) For each $\nu \in \Ncal^r(F)$, one may associate a muliplicative seminorm $x_\nu$ on $R(r,\CC_F)$ by the following:
choose an orthogonal basis $\{u_1,...,u_r\}$ of $F^r$ with respect to $\nu$, and
set
$$\bigg|\sum_{\underline{i} \in \ZZ_{\geq 0}^r} a_{\underline{i}} \underline{u}^{\underline{i}}\bigg|_{x_\nu}
\assign \max\big(|a_{\underline{i}}|_F \cdot \nu(\underline{u})^{\underline{i}}\ \big|\ \underline{i} \in \ZZ_{\geq 0}^r \big), \quad \forall \sum_{\underline{i}} a_{\underline{i}} \underline{u}^{\underline{i}} \in R(r,\CC_F).
$$
Here for $\underline{i} = (i_1,\cdots \hspace{-0.03cm},i_r) \in \ZZ_{\geq 0}^r$,
we put
\[
\underline{u}^{\underline{i}} \assign u_1^{i_1}\cdots u_r^{i_r} \quad \text{ and } \quad
\nu(\underline{u})^{\underline{i}} \assign 
\nu(u_1)^{i_1}\cdots \nu(u_r)^{i_r}.
\]
Observe that $x_\nu$ does not depend on the chosen orthogonal basis $\{u_1,...,u_r\}$.
Since
\[
|a_1u_1+\cdots + a_r u_r|_{x_\nu} = \nu(a_1u_1+\cdots + a_r u_r) > 0 
\quad \text{ for every nonzero }
a = (a_1,\cdots \hspace{-0.03cm},a_r) \in F^r,
\]
we may extend $x_\nu$ to a multiplicative norm on $\widetilde{R}(r,\CC_F)$, i.e.\ $x_\nu \in \widetilde{\Omega}_F^r$.
This gives an embedding $\tilde{\varsigma} : \Ncal^r(F) \hookrightarrow \widetilde{\Omega}_F^r$ so that 
$\widetilde{\Bscr} \circ \tilde{\varsigma} = \text{id}_{ \Ncal^r(F)}$, which induces an embedding $\varsigma: \Xcal^r(F) \hookrightarrow \Omega_F^r$ satisfying
$\Bscr \circ \varsigma = \text{id}_{ \Xcal^r(F)}$.
This ensures the surjectivity of $\Bscr$.
In particular, $\Xcal^r(F)$ is homeomorphic to the image of $\varsigma$, and $\varsigma(\Xcal^r(F))$ is closed in $\Omega^r_F$ (see \cite[Proof of Theorem 1: Step 2]{Ber1}).
Moreover,
for each $x$ in $\widetilde{\Omega}^r_F$ with $\widetilde{\Bscr}(x) = \nu$, one has that
\[
|f|_x \leq |f|_{x_\nu}, \quad \forall f \in \tilde{R}(r,\CC_F).
\]
\end{rem}

\begin{rem}\label{rem: atlas}
Note that $\Omega_F^r$ is a ``strict $F$-analytic space'' in the sense of \cite[\S~3.1]{Ber90}.
According to \cite[1.6.1~Theorem]{Ber93}, it is uniquely determined by its ``rigid $F$-analytic structure'' $\Omega_F^{r,\rm{rig}}$ introduced in \cite[\S~1.6]{Ber93}.
Consequently, the building map $\Bscr$ allows us to select a suitable atlas on $\Omega_F^r$.
We refer the reader to \cite[Proposition~6.2]{Drin} for an explicit selection of such an atlas.
\end{rem}

\begin{Subsubsec}{Imaginary part}\label{sec: Imag}
Given $z \in \Omega_F^r$, which is identified with a point on $\widetilde{\Omega}_F^r$ satisfying
$e_r(z) =1$,
the norm $\nu_{z}$ on $F^r$ corresponding to $z$ is uniquely determined by $z$.
We set the {\it imaginary part of $z$} to be
\[
\im(z) \assign D(\nu_{z}),
\]
where $D(\nu_z)$ is the lattice discriminant of the norm $\nu_z$ introduced in Definition~\ref{defn: LD}.
Moreover, for every 
$g =
\begin{pmatrix}
    a_{11}&\cdots & a_{1r} \\
    \vdots &  & \vdots \\
    a_{r1}& \cdots & a_{rr}
\end{pmatrix}
\in \GL_r(F)$,
define the following {\it automorphy factor}
\[
j(g ,z) \assign (a_{r1}e_1+\cdots + a_{rr}e_r)(z) \quad \in \CC_F(z).
\]
Then we have an identification between $\CC_F(g \cdot z)$
and $\CC_{\infty}( z)$ induced from the automorphism of the symmetric algebra $R(r,\CC_F)$ associated with $g $, i.e.~
\begin{equation}\label{eqn: field-id}
\CC_F(g \cdot z) \ni \quad f(g \cdot z) =\joinrel= \frac{(f\cdot g)(z)}{j(g ,z)^{\deg f} } \quad \in \CC_F(z), \quad \forall \text{ homogeneous } f \in R(r,\CC_F).
\end{equation}
Moreover,
\[
|j(g,z)| = |a_{r1}e_1+\cdots + a_{rr}e_r|_{z}
= \nu_{z}(a_{r1}e_1+\cdots + a_{rr}e_r),
\]
and the following transformation law holds:
\end{Subsubsec}

\begin{lem}\label{lam: trans-law}
For every $z \in \Omega_F^r$ and $g \in \GL_r(F)$, we have that
\[
\im(g \cdot z) = \frac{|\det g|_F}{|j(g,z)|^r}\cdot \im(z).
\]
\end{lem}

\begin{proof}
Identifying $z$ and $g\cdot z$ as points in $\AA_{\CC_{\infty}}^{r,{\rm an}}$ so that $e_r(z)=e_r(g\cdot z) =1$,
one observes that
\[
\nu_{g\cdot z} = |j(g,z)|^{-1}\cdot (g*\nu_z).
\]
Therefore by Lemma~\ref{lem: D-trans}, we get
\[
\im(g\cdot z) = D(\nu_{g\cdot z}) = \frac{1}{|j(g,z)|^r} \cdot D(g*\nu_{z})=\frac{|\det g|_F}{|j(g,z)|^r} \cdot D(\nu_{z}) = \frac{|\det g|_F}{|j(g,z)|^r} \cdot \im(z)
\]
as desired.
\end{proof}

\section{Drinfeld modules over Drinfeld period domains and discriminants}\label{sec: DM-DPD}


Following \cite[1.5.3~Definition and 1.5.5~Remark]{Ber90}, an \emph{analytic function on $\Omega_F^r$} is a map
\[
f : \Omega_F^r \rightarrow \coprod_{z \in \Omega_F^r} \CC_F(z)
\]
with $f(z) \in \CC_F(z)$ which is a ``local limit'' of rational functions, i.e.\ for each $z \in \Omega_F^r$, there exists an ``admissible'' open neighborhood $U_z$ of $z$ with the following property:
given $\epsilon > 0$, there exist $g,h \in R(r,\CC_F)$ such that for every $z' \in U_z$, we have that
\[
h(z') \neq 0 \quad \text{ and } \quad \left|f(z')-\frac{g(z')}{h(z')}\right|<\epsilon.
\]
Since $\Omega_F^r$ is determined by $\Omega_F^{r,{\rm rig}}$ as noted in Remark~\ref{rem: atlas}, 
the ring of analytic functions on $\Omega_F^r$ coincides with the ring of rigid analytic functions on $\Omega_F^{r,\text{rig}}$.
In particular, when the characteristic of $F$ is positive, {\it Drinfeld modular forms} of rank $r$ on $\Omega_F^{r,{\rm rig}}$ 
form a fundamental class of rigid analytic functions in function field arithmetic. These functions serve as natural analogs of classical modular forms in the equal-characteristic setting over function fields and have been the subject of extensive recent study. We refer the reader to \cite{BBP}, \cite{Gek-1}, \cite{Gek-2}, \cite{Gek-3}, \cite{Gek-4}, \cite{Gek-5}, \cite{Gek-6} for comprehensive discussions. Consequently, we may regard them as ``Berkovich'' analytic functions on $\Omega_F^r$ as well.

In the following, 
we explicitly construct ``Klein forms'' on $\Omega_F^r$ in the positive characteristic setting, which are specific invertible analytic functions closely related to the ``coefficient forms'', and introduce the universal Drinfeld modules over Drinfeld period domains.

\subsection{Klein forms on Drinfeld period domains}

Let $\FF_q$ be a finite field with $q$ elements, where $q$ is a power of a prime number $p$.
Let $k$ be a global function field with the constant field $\FF_q$, i.e.\ $k$ is a finitely generated field extension over $\FF_q$ with transcendence degree one and $\FF_q$ is algebraically closed in $k$.
Fix a place $\infty$ of $k$, referred to the place at infinity.
The valuation of $k$ with respect to $\infty$ is denoted by $\ord_\infty$.
Let $k_\infty$ be the completion of $k$ at $\infty$.
Set
\[
O_\infty \assign \{\alpha \in k_\infty \mid \ord_\infty(\alpha)\geq 0\},\quad
\pfk_\infty \assign \{\alpha \in O_\infty \mid \ord_\infty(\alpha)>0\}, \quad \text{and} \quad 
\FF_\infty \assign \frac{O_\infty}{\pfk_\infty}.
\]
Put $q_\infty \assign \#(\FF_\infty)$.
The absolute value on $k_\infty$ is normalized by
\[
|\alpha|_\infty \assign q_\infty^{-\ord_\infty(\alpha)}, \quad \forall \alpha \in k_\infty.
\]
Let $\CC_\infty$ be the completion of a chosen algebraic closure of $k_\infty$.\\

Fix a positive integer $r$.
Let $\Omega_\infty^r$ be the Berkovich Drinfeld period domain of rank $r$ over $k_\infty$, 
Throughout this section, we set
\[
R_{r,\infty} = R(r,\CC_\infty)
\quad \text{ and } \quad 
\widetilde{R}_{r, \infty} = \widetilde{R}(r,\CC_\infty).
\]
Given $\alpha = (c_1,\cdots \hspace{-0.03cm},c_r) \in \CC_\infty^r$, we put
\[
w_{\alpha} \assign \left(c_1 e_1+\cdots + c_r e_r\right)/e_r \quad \in \widetilde{R}_{r,\infty},
\]
which can be regarded as a rational function on $\Omega_\infty^r$.
Moreover, let $A$ be the subring of $k$ consisting of all elements in $k$ regular away from $\infty$.
Let $Y \subset k^r$ be a projective $A$-module of rank $r$.
Then for each $z \in \Omega_{\infty}^r$,
the $A$-lattice
\begin{equation}\label{eqn: Lambda-z}
\Lambda^Y_z \assign \big\{w_{\lambda}(z) \ \big| \ \lambda \in Y\big\} \subset \CC_{\infty}(z) 
\end{equation}
satisfies that
\begin{equation}\label{eqn: asymp}
\#\big\{w_\lambda(z) \in \Lambda_z^Y \ \big|\ |w_\lambda(z)| \leq q^{\epsilon} \big\} \asymp
q^{r \epsilon}, \quad \epsilon \in \RR_{>0}.
\end{equation}
This assures that for each $\alpha \in \CC_\infty^r $,
\[
e_{\alpha}^Y \assign
w_{\alpha} \cdot \prod_{0\neq \lambda \in Y} \left(1-\frac{w_{\alpha}}{w_{\lambda}}\right),
\]
is a well-defined  analytic function on the Berkovich space $\Omega_\infty^r$ by the asymptotic estimation~\eqref{eqn: asymp}, which is invertible when $\alpha \in k_{\infty}^r\setminus Y$ and $0$ if $\alpha \in Y$.

\begin{lem}\label{lem: klein}
Given $ \alpha, \alpha' \in \CC_{\infty}^r$, $\lambda \in Y$, and $\varepsilon \in \FF_q$, one has
\[
e^Y_{\varepsilon \alpha+\alpha' + \lambda} = \varepsilon \cdot e^Y_{\alpha}+e^Y_{\alpha'}.
\]
Moreover, the following transformation law holds:
\[
e_{\alpha}^Y(\gamma \cdot z) = j(\gamma,z)^{-1} \cdot e_{\alpha\gamma}^Y(z), \quad \forall \gamma \in \GL(Y),\ z \in \Omega_\infty^r.
\]
\end{lem}

\begin{proof}
Given $z \in \Omega_{\infty}^r$, consider the following exponential function
\begin{align*}
\exp_z^Y(\xi) & \assign \xi \cdot \prod_{0\neq \lambda \in Y}\left(1-\frac{\xi}{w_{\lambda}(z)}\right) \\
&= \lim_{\epsilon \rightarrow \infty} \  \xi \cdot \prod_{\subfrac{0\neq \lambda \in Y}{|w_{\lambda}(z)|\leq q^{\epsilon}}}\left(1-\frac{\xi}{w_{\lambda}(z)}\right),
\quad \forall \xi \in \CC_{\infty}(z).
\end{align*}
As $\Lambda_z^{Y,\epsilon} \assign \big\{w_\lambda(z) \in \Lambda_z^Y \ \big|\ |w_\lambda(z)| \leq q^{\epsilon} \big\}$ is a finite dimensional vector space over $\FF_q$ for every $\epsilon \in \RR_{\geq 0}$,
the polynomial
${\rm x} \cdot \prod_{0\neq w_{\lambda}(z) \in \Lambda_z^{Y,\epsilon}} \big(1-{\rm x}/w_{\lambda}(z)\big) \in \CC_{\infty}(z)[\rm x]$ is $\FF_q$-linear, whence so is the function  $\exp_z^Y$.
Therefore
\begin{align*}
e_{\varepsilon \alpha + \alpha' + \lambda}^Y(z)
= \exp_{z}^Y\big(w_{\varepsilon\alpha+\alpha'+\lambda}(z)\big)
&=\exp_{z}^Y\big(\varepsilon w_{\alpha}(z)+w_{\alpha'}(z)+w_{\lambda}(z)\big) \\
&= \varepsilon \exp_{z}^{Y}\big(w_\alpha(z)\big) + \exp_{z}^{Y}\big(w_{\alpha'}(z)\big)
= \varepsilon \cdot e_{\alpha}^Y(z) + e_{\alpha'}^Y(z).
\end{align*}

Moreover, via the identification \eqref{eqn: field-id} one has
\[
w_{\alpha}(g \cdot z) = j(g,z)^{-1} \cdot w_{\alpha g}(z), \quad \forall \alpha \in \CC_{\infty}^r \text{ and } g \in \GL_r(k_{\infty}).
\]
Hence for every $\gamma \in \GL(Y)$ and $z \in \Omega_{\infty}^r$,
\begin{align*}
e_{\alpha}^Y(\gamma \cdot z) = w_{\alpha}(\gamma \cdot z) \cdot \prod_{0\neq \lambda \in Y}\left(1-\frac{w_{\alpha}(\gamma \cdot z)}{w_{\lambda}(\gamma \cdot z)}\right)
& = \frac{w_{\alpha \gamma}(z)}{j(\gamma,z)} \cdot \prod_{0\neq \lambda \in Y}
\left(1-\frac{w_{\alpha\gamma}(z)}{w_{\lambda \gamma }( z)}\right) \\
&=
\frac{w_{\alpha \gamma}(z)}{j(\gamma,z)} \cdot \prod_{0\neq \lambda \in Y}
\left(1-\frac{w_{\alpha\gamma}(z)}{w_{\lambda}( z)}\right)
= \frac{e_{\alpha \gamma }^Y(z)}{j(\gamma,z)}.
\end{align*}
\end{proof}

\begin{rem}\label{rem: Klein}
${}$
\begin{itemize}
    \item[(1)] Let $O(\Omega_{\infty}^r)$ be the ring of analytic functions on $\Omega_{\infty}^r$.
    Then the first assertion of Lemma~\ref{lem: klein} induces an $\FF_q$-linear map from $\CC_{\infty}^r/Y$ to $O(\Omega_{\infty}^r)$ which sends a coset $\alpha + Y \in \CC_{\infty}^r/Y$ to $e_{\alpha}^Y$.
    \item[(2)] Given $\alpha \in k_{\infty}^r \setminus Y$, the \textit{Klein form of $\alpha$ with respect to $Y$} is defined by
\begin{equation}\label{eqn: Klein}
\kfk_{\alpha}^Y \assign (e_{\alpha}^Y)^{-1} = w_\alpha^{-1} \cdot \prod_{0\neq \lambda \in Y} \left(1-\frac{w_\alpha}{w_\lambda}\right)^{-1},
\end{equation}
which is an invertible analytic function on $\Omega_{\infty}^r$ satisfying
\[
\kfk_{\alpha}^Y(\gamma\cdot z) = j(\gamma,z) \cdot \kfk_{\alpha \gamma }^Y(z), \quad \forall \gamma \in \GL(Y) \text{ and } z \in \Omega_{\infty}^r.
\]
These particular analytic functions appear in the construction of the ``Drinfeld $A$-module of rank $r$ over $\Omega_{\infty}^r$''
in the next subsection.
\end{itemize}

\end{rem}

\subsection{Drinfeld modules over Drinfeld period domains}

Recall that $O(\Omega_{\infty}^r)$ denotes the ring of analytic functions on the Berkovich space $\Omega_{\infty}^r$.
Given $a \in A$, by Remark~\ref{rem: Klein} (1) we know that 
\[
\big\{e_{\alpha}^Y\ \big|\  \alpha \in a^{-1}Y/Y\big \} = \{0\} \cup \big\{ (\kfk^Y_{\alpha})^{-1} \ \big| \  0\neq \alpha \in a^{-1}Y/Y\big\}
\]
is an $\FF_q$-vector subspace of $O(\Omega_{\infty}^r)$ with dimension $r\cdot \deg a$, where $\deg a \assign \log_{q}(\#(A/aA))$.
Thus the polynomial
\[
\varphi^{Y,\Omega_{\infty}^{r}}_a({\rm x})\assign a{\rm x} \cdot \prod_{0\neq \alpha \in a^{-1}Y/Y} (1-\kfk^Y_{\alpha}\cdot  {\rm x}) \quad \in O(\Omega_{\infty}^r)[{\rm x}]
\]
is $\FF_q$-linear, for which there exist ``coefficient forms'' $g^Y_{a,1},...,g^Y_{a,r\deg a} \in O(\Omega_{\infty}^r)$ so that
\[
\varphi_a^{Y,\Omega_{\infty}^{r}}({\rm x}) = a{\rm x} + \sum_{i=1}^{r\deg a } g^Y_{a,i} {\rm x}^{q^i}.
\]
Moreover, for every $a,b \in A$, one has that
\begin{equation}\label{eqn: phi-mult}
\varphi_a^{Y,\Omega_{\infty}^{r}} \circ \varphi_b^{Y,\Omega_{\infty}^{r}} = \varphi_{ab}^{Y,\Omega_{\infty}^{r}} = 
\varphi_b^{Y,\Omega_{\infty}^{r}} \circ \varphi_a^{Y,\Omega_{\infty}^{r}}.
\end{equation}
This can be confirmed by verifying the corresponding relations among the coefficient forms evaluated at points in the rigid analytic structure $\Omega_{\infty}^{r,{\rm rig}}$, which is checked in \cite[p.~235]{Ros}. 
This gives an $\FF_q$-algebra homomorphism $\varphi^{Y,\Omega_{\infty}^{r}}: A \rightarrow \End_{\FF_q}(\GG_{a/O(\Omega_{\infty}^{r})})$, called the {\it Drinfeld $A$-module of rank $r$ over $\Omega_{\infty}^r$}.

Now, for each point $z \in \Omega_{\infty}^r$ (regarded as a point on $\AA_{\CC_{\infty}}^{r-1,{\rm an}}$ through the identification in Remark~\ref{rem: Omega-id}), the evaluation map from $O(\Omega_{\infty}^r)$ to the residue field $\CC_{\infty}(z)$ sends $\varphi^{Y,\Omega_{\infty}^r}$ to a Drinfeld $A$-module $\varphi^{Y,z}$ of rank $r$ over $\CC_{\infty}(z)$: for each $a \in A$,
\begin{equation}\label{eqn: specialization}
\varphi_a^{Y,z}({\rm x}) = a{\rm x} + \sum_{i=1}^{r\deg a} g_{a,i}^Y(z) {\rm x}^{q^i} \quad \in \CC_{\infty}(z)[{\rm x}].
\end{equation}
We call $\varphi^{Y,z}: A \rightarrow \End_{\FF_q}(\GG_{a/\CC_{\infty}(z)})$ the {\it specialization of $\varphi^{Y,\Omega_{\infty}^r}$ at $z$}.

\begin{rem}\label{rem: coe-tran}
From the transformation law of the Klein forms in Remark~\ref{rem: Klein}~(2), one has that given $a \in A$ and $1\leq i \leq r\deg a$,
\[
g^Y_{a,i}(\gamma \cdot z) = j(\gamma,z)^{q^{i}-1} \cdot g^Y_{a,i}(z), \quad \forall \gamma \in \GL(Y) \text{ and } z \in \Omega_{\infty}^r.
\]
\end{rem}

For each nonzero $a \in A$,
define
\[
\Delta^Y_a \assign g^Y_{a,r\deg a} =
a \cdot \prod_{0\neq \alpha \in \frac{1}{a}Y/Y} \kfk^Y_{\alpha},
\]
which is an invertible analytic function on $\Omega_{\infty}^r$. Comparing the leading coefficients on both sides in \eqref{eqn: phi-mult}, we have that
\[
(\Delta^Y_{a_2})^{|a_1|^r_\infty-1} = (\Delta^Y_{a_1})^{|a_2|^r_\infty-1}, \quad \forall \text{ nonzero }a_1,a_2 \in A.
\]
Therefore following the same approach in \cite[Chapter IV, Proposition 5.15]{Gek1}, we get:

\begin{lem}\label{lem: Delta}
There exists an analytic function $\Delta^Y$ on $\Omega^r_\infty$, which is unique up to multiplying with a $(q_\infty^r-1)$-th root of unity, satisfying that
\[
(\Delta^Y)^{|a|_\infty^r -1} = (\Delta^Y_a)^{q_\infty^r-1}, \quad \forall \text{ nonzero } a \in A.
\]
\end{lem}

\begin{defn}\label{defn: Delta}
The analytic function $\Delta^Y$ on $\Omega_\infty^r$ is called \emph{the Drinfeld discriminant function associated with $Y$}.
\end{defn}

\begin{rem}\label{rem: can-Delta}
The recent work of Gekeler in \cite[Corollary/Definition 7.6]{Gek25} shows that there is a ``canonical'' choice of $\Delta^Y$ with respect to a fixed ``sign function''.
\end{rem}

The transformation law of the coefficient forms in Remark~\ref{rem: coe-tran} implies that:

\begin{lem}\label{lem: delta-tran}
Given $\gamma \in \GL(Y)$ and $z \in \Omega_\infty^r$, we have that
\[
\Delta^Y(\gamma \cdot z) = j(\gamma,z)^{q^r-1} \cdot \Delta^Y(z) \quad (\in \CC_{\infty}(z)).
\]
\end{lem}

\begin{Subsubsec}{Absolute discriminant}
Given $z \in \Omega^r_{\infty}$, recall the exponential function defined in the proof of Lemma~\ref{lem: klein}:
\[
\exp_{z}(\xi) \assign \xi \cdot \prod_{0\neq \lambda \in Y}\left(1-\frac{\xi}{w_{\lambda}(z)}\right), \quad \quad \forall \xi \in \CC_{\infty}(z),
\]
We know that the function $\exp_z$ is $\FF_q$-linear and entire on the non-archimedean complete field $\CC_{\infty}(z)$
which satisfies (cf.~\cite[p.~235]{Ros})
\[
\exp_z(a \cdot \xi)
= \varphi_a^{Y,z}\big(\exp_z(\xi)\big), \quad \forall a \in A \ \text{ and }\  \xi \in \CC_{\infty}(z).
\]
Also, we may call
$\Lambda_z=\big\{w_{\lambda}(z)\mid \lambda \in Y\big\} \subset \CC_{\infty}(z)$
the {\it period lattice of $\varphi^{Y,z}$}.

Set
\[
\Vert Y \Vert_A \assign |a|_{\infty}^{-r} \cdot \#(A^r/aY) \quad \text{ for nonconstant $a \in A$ so that $aY \subset A^r$},
\]
which is independent of the chosen $a \in A$.
Following the discussion in \cite[Remark~2.10]{Wei20}, we may regard the quantity $\Vert Y\Vert_A \cdot \im(z)$ as the ``covolume'' of the period lattice $\Lambda_z$.
This leads us to the following definition:
\end{Subsubsec}

\begin{defn}\label{defn: eta}
For each $z \in \Omega_{\infty}^r$, the {\it absolute 
discriminant of $z$} is defined by
\[
\eta_r^Y(z)\assign \Vert Y \Vert_A \cdot \im(z) \cdot |\Delta^Y(z)|^{\frac{r}{q_{\infty}^r-1}} \quad \in \RR_{>0}.
\]
\end{defn}

\begin{rem}\label{rem: covolume}
Let $\overline{\CC_{\infty}(z)}$ be an algebraic closure of $\CC_{\infty}(z)$, and extend the absolute value $|\cdot|_z$ to $\overline{\CC_{\infty}(z)}$.
Take $c \in \overline{\CC_{\infty}(z)}^{\times}$ with $c^{q_{\infty}^r-1} = \Delta^Y(z)$, and consider the Drinfeld $A$-module $\psi^c$ of rank $r$ over $\overline{\CC_{\infty}(z)}$ defined by 
\[
\psi^c_a({\rm x}) \assign c \cdot \varphi^{Y,z}_a \big(c^{-1}{\rm x}\big) = a{\rm x} + \sum_{i=1}^{r\deg a} \big(g_{a,i}^{Y}(z) c^{1-q^{i}}\big) x^{q^i} \quad \in \overline{\CC_{\infty}(z)}[{\rm x}], 
\]
which has ``unit discriminant'', i.e.
\[
\Big(g_{a,r\deg a}^Y(z)c^{1-q^{r\deg a}}\Big)^{q_{\infty}^r-1} = \Delta^Y_a(z)^{q_{\infty}^r-1} \cdot \Delta^Y(z)^{1-q^{r\deg a}} = 1 \quad \text{ for every nonzero $a \in A$,}
\]
and the period lattice of $\psi^c$ becomes 
\[
c\cdot \Lambda_z \subset \overline{\CC_{\infty}(z)}.
\]
Therefore in this case, the value $\eta_r^{Y}(z)$ can be regarded as the ``covolume'' of the period lattice $c\cdot \Lambda_z$ of the ``unit'' Drinfeld $A$-module $\psi^c$ which is isomorphic to $\varphi^{Y,z}$ over $\overline{\CC_{\infty}(z)}$.
\end{rem}

In the next section, we shall extend the Kronecker limit formula established in \cite[equation~(4.2) and (4.3)]{Wei20} to the ``Berkovich'' version, which connects the derivative of the ``non-holomorphic'' Eisenstein series on $\Omega_{\infty}^r$ with the absolute discriminant $\eta_r^Y$.

\section{Kronecker-type limit formula: one variable case}\label{sec: KLF-1}

Here we translate the proof of the Kronecker-type limit formula in \cite[Theorem~4.4]{Wei20}
into ``Berkovich'' setting.
First, we introduce the non-holomorphic Eisenstein series on the Berkovich Drinfeld period domain $\Omega_{\infty}^r$ and derive its analytic properties.

\subsection{Non-holomorphic Eisenstein series}

Let $Y\subset k^r$ be a projective $A$-module of rank $r$. Recall that
\[
\Vert Y \Vert_A =|a|_{\infty}^{-r} \cdot \#(A^r/aY) \quad \text{ for nonconstant $a \in A$ so that $aY \subset A^r$},
\]
which is independent of the chosen $a \in A$.
Given $z \in \Omega_\infty^r$, we are interested in the following Eisenstein series:
$$
\EE^Y(z,s) \assign \sum_{0\neq (c_1,\cdots \hspace{-0.01cm},c_r) \in Y} \frac{\Vert Y \Vert_A^s \cdot  \im(z)^s}{|c_1e_1+\cdots +c_re_r|_{z}^{rs}}= \Vert Y \Vert_A^s \cdot  \sum_{0\neq \lambda \in Y} \frac{\im(z)^s}{\nu_{z}(\lambda)^{rs}}.
$$

\begin{lem}\label{lem: converge}
The Eisenstein series $\EE^Y(z,s)$ converges absolutely for $\re(s)>1$.
\end{lem}

\begin{proof}
Let $\nu=\nu_{z}$ be the norm on $k_\infty^r$ corresponding to $z$.
For $y \in \RR_{>0}$, let 
\[
B_\nu(y) \assign \{ x \in k_\infty^r \mid \nu(x) \leq y\}.
\]
Then $B_\nu(y)$ is an $O_\infty$-lattice in $k_\infty^r$.
Take $\mathbf{1}_{B_\nu(y)}$ to be the characteristic function of $B_\nu(y)$ on $k_{\infty}^r$.
For $0 \neq x \in  k^r$ and $\re(s) > 0$, one observes that
$$
\int_0^\infty \mathbf{1}_{B_\nu(y)}(x) y^{-rs} \frac{dy}{y}
= \frac{1}{rs} \cdot \nu(x)^{-rs}.
$$
Since $Y$ is discrete in $k_\infty^r$, one has
$$y_0 \assign \inf\{\nu(x) \mid 0 \neq x \in Y\} > 0.$$
Define the following theta series:
$$\Theta^Y(\nu,y) \assign \sum_{\lambda \in Y} \mathbf{1}_{B_\nu(y)}(\lambda) \quad \in \RR_{\geq 1}.$$
There exists $C \in \RR_{>0}$ so that $\Theta^Y(\nu,y) \leq  C \cdot y^r$ for all $y \geq y_0$.
Thus
\begin{eqnarray}
\int_0^\infty \left|\left(\Theta^Y(\nu,y) - 1\right) y^{-rs} \right| \frac{dy}{y}
&\leq & C \cdot \int_{y_0}^\infty y^{r(1-\re(s))} \frac{dy}{y} \nonumber \\
&= & C \cdot \frac{y_0^{r(1-\re(s))}}{r(\re(s)-1)} \quad \text{ as $\re(s)>1$.}\nonumber
\end{eqnarray}
This implies that for $\re(s) > 1$, the Mellin transform of $\Theta^Y$ converges absolutely and equals to 
\begin{eqnarray}
\int_0^\infty \left(\Theta(\nu,y) - 1\right) y^{-rs} \frac{dy}{y}
&=& \sum_{0\neq \lambda \in Y} \int_0^\infty \mathbf{1}_{B_\nu(y)}(\lambda)\, y^{-rs} \frac{dy}{y} \nonumber \\
&=& \frac{1}{rs} \sum_{0\neq \lambda \in Y} \frac{1}{\nu(\lambda)^{rs}} = \frac{\Vert Y \Vert_A ^{-s}\im(z)^{-s}}{rs} \cdot \EE^Y(z,s).\nonumber 
\end{eqnarray}
This completes the proof.
\end{proof}

\begin{rem}
${}$
\begin{itemize}
    \item[(1)] It is straightforward that
    $
    \EE^Y(\gamma \cdot z, s) = \EE^Y(z,s)$ for every $\gamma \in \GL(Y)$.
    \item[(2)] 
    Since $\im(z) = D(\nu_z)$ for every $z \in \Omega_{\infty}^r$, we may write
    \[
    \EE^Y(z,s) = \Vert Y \Vert_A^s \cdot \sum_{0\neq \lambda \in Y} \frac{D(\nu_z)^s}{\nu_z(\lambda)^{rs}}
    \]
Hence the Eisenstein series $\EE^Y(\cdot ,s)$ actually factors through the building map $\Bscr$, and $\EE^Y(z',s) = \EE^Y(z,s)$ when $\Bscr(z') = \Bscr(z)$ (i.e.~$\nu_{z'}=\nu_{z}$).
\end{itemize}
\end{rem}

The following lemma gives the meromorphic continuation of $\EE^Y(z,s)$ to the whole complex $s$-plane:

\begin{lem}\label{lem: mero}
Given $s\in \CC$ with $\re(s) > 1$ and $a \in A\setminus \FF_q$, the following equality holds:
$$
\EE^Y(z,s) = \frac{\Vert Y \Vert_A^s \im(z)^s}{|a|_\infty^{rs}-|a|_\infty^{r}} \sum_{0\neq w \in \frac{1}{a}Y/Y} \left[\frac{1}{\nu_z(w)^{rs}} 
+ \sum_{\substack{\lambda \in Y \\ \nu_z(\lambda) \leq \nu_z(w)}}'
 \left(\frac{1}{\nu_z(\lambda - w)^{rs}} - \frac{1}{\nu_z(\lambda)^{rs}}\right)\right].
$$
In particular, one has $\EE^Y(z,0) = -1$.
\end{lem}

\begin{proof}
We first observe that
\[
|a|_\infty^{-rs}\cdot \EE^Y(z,s) = \frac{\Vert Y\Vert_A^s}{\Vert a Y \Vert_A^s} \cdot \EE^{aY}(z,s).
\]
Hence
\begin{align*}
\EE^Y(z,s)-|a|_\infty^{-rs}\EE^Y(z,s)
&= \Vert Y \Vert_A^s\cdot \sum_{w \in Y- aY}\frac{\im(z)^s}{\nu_z(w)^{rs}} \\
&= \Vert Y \Vert_A^s \cdot \sum_{0\neq w \in Y/aY}\left(
\frac{\im(z)^s}{\nu_z(w)^{rs}} + \sum_{0\neq \lambda \in aY} \frac{\im(z)^s}{\nu_z(\lambda-w)^{rs}}
\right).
\end{align*}
Therefore we get that
\begin{eqnarray}
&& (1-|a|_\infty^{r-rs}) \cdot \EE^Y(z,s) \nonumber \\
&=& \Big(\EE^Y(z,s) - |a|_\infty^{-rs} \EE^Y(z,s)\Big) - (|a|_\infty^r-1) \cdot |a|_\infty^{-rs}\EE^Y(z,s) \nonumber \\
&=& \Vert Y \Vert_A^s \cdot \sum_{0\neq w \in Y/aY} \left[ \left(\frac{\im(z)^s}{\nu_z(w)^{rs}} + \sum_{0\neq  \lambda \in aY} \frac{\im(z)^s}{\nu_z(\lambda -w)^{rs}}\right) - \frac{\EE^{aY}(z,s)}{\Vert aY\Vert_A^s} \right] \nonumber \\
&=&\Vert Y \Vert_A^s \im(z)^s \cdot \sum_{0\neq w \in Y/aY} \left[\frac{1}{\nu_z(w)^{rs}} 
+ \sum_{0\neq \lambda \in aY}
 \left(\frac{1}{\nu_z(\lambda - w)^{rs}} - \frac{1}{\nu_z(\lambda)^{rs}}\right)\right] \nonumber \\
&=&\frac{\Vert Y \Vert_A^s \im(z)^s}{|a|_\infty^{rs}}
\cdot 
\sum_{0\neq w \in \frac{1}{a}Y/Y} \left[\frac{1}{\nu_z(w)^{rs}} 
+ \sum_{0\neq \lambda \in Y}
 \left(\frac{1}{\nu_z(\lambda - w)^{rs}} - \frac{1}{\nu_z(\lambda)^{rs}}\right)\right].
\nonumber 
\end{eqnarray}
The result then follows from the fact that $\nu_z(\lambda-w) = \nu_z(\lambda)$ if $\nu_z(\lambda) > \nu_z(w)$.
\end{proof}

Moreover, for each $w = a_1e_1+\cdots + a_r e_r \in k_\infty^r-Y$,
define the \emph{Jacobi-type} Eisenstein series:
\[
\EE^Y(z,w,s)\assign \sum_{(c_1,\cdots \hspace{-0.03cm},c_r) \in Y} \frac{\Vert Y \Vert_A^s \im(z)^s}{\big|(c_1-a_1)e_1+\cdots + (c_r-a_r)e_r\big|_{z}^{rs}} = \Vert Y \Vert_A^s \cdot \sum_{\lambda \in Y} \frac{\im(z)^s}{\nu_{z}(\lambda-w)^{rs}},
\]
which also converges absolutely for $\re(s)>1$ because of the following equality:
\begin{equation}\label{eqn: mero-JE}
\Vert Y \Vert_A^{-s} \cdot \left(\EE^Y(z,w,s)-\EE^Y(z,s)\right)
=\frac{\im(z)^s}{\nu_{z}(w)^{rs}} +\sum_{\subfrac{0\neq \lambda \in Y}{\nu_{z}(\lambda)\leq \nu_{z}(w)}}\left(\frac{\im(z)^{s}}{\nu_{z}(\lambda-w)^{rs}}-\frac{\im(z)^{s}}{\nu_{z}(\lambda)^{rs}}\right).
\end{equation}
Since the right hand side of the above equality is a finite sum, we also have the meromorphic continuation of $\EE^Y(z,w,s)$ to the whole complex $s$-plane
with $\EE^Y(z,w,0) = 0$.
We now derive a Kronecker-type limit formula for $\EE^Y(z,s)$ and $\EE^Y(z,w,s)$:

\begin{thm}\label{thm: KLM-1}
Given $z \in \Omega_\infty^r$ and $w \in \CC_\infty^r$, we have that
\begin{eqnarray}\label{eqn: KLF-1}
\frac{\partial}{\partial s}\EE^Y(z,s)\Big|_{s=0} & =& -\ln\big(\Vert Y\Vert_A \cdot \im(z)\big) - \frac{r}{q^r-1}\cdot \ln |\Delta^Y(z)|,\\
\label{eqn: KLF-2}
\frac{\partial}{\partial s}\EE^Y(z,w,s)\Big|_{s=0} & =& - \frac{r}{q_\infty^r-1}\cdot \ln |\Delta^Y(z)| + r \ln |\kfk^Y_w(z)|.
\end{eqnarray}
\end{thm}

\begin{proof}
From the expression of $\EE^Y(z,s)$ in Lemma~\ref{lem: mero}, we get
\begin{align*}
&\hspace{0.1cm}\frac{\partial}{\partial s} \EE^Y(z,s)\Big|_{s=0} \\
&= - \ln\big(\Vert Y\Vert_A \cdot \im(z)\big) - \frac{r}{|a|_\infty^r-1} \cdot \left[\ln |a|_\infty - \sum_{w \in \frac{1}{a}Y/Y}' \left(\ln \nu_z(w)
+ \sum_{\substack{\lambda \in Y \\ \nu_z(\lambda) \leq \nu_z(w)}}'
 \ln \nu_z\Big(1 - \frac{w}{\lambda}\Big)\right)\right]\\
&= -\ln\big(\Vert Y\Vert_A \cdot \im(z)\big)-\frac{r}{|a|_\infty^r-1} \cdot \Big(\ln|a|_\infty + \sum_{w \in \frac{1}{a}Y/Y}' \ln |\kfk_w^Y(z)|\Big) \\
&= -\ln\big(\Vert Y\Vert_A \cdot \im(z)\big) - \frac{r}{|a|_\infty^r-1} \cdot \ln |\Delta_a^Y(z)|
= -\ln\big(\Vert Y\Vert_A \cdot \im(z)\big) - \frac{r}{q_\infty^r-1} \cdot \ln |\Delta^Y(z)|
\end{align*}
where the last equality comes from Lemma~\ref{lem: Delta}.
Moreover, the description of $\EE^Y(z,w,s)$ in the equation \eqref{eqn: mero-JE} implies that
\begin{align*}
\frac{\partial}{\partial s} \EE^Y(z,w,s)\Big|_{s=0}
&= \frac{\partial}{\partial s} \EE^Y(z,s)\Big|_{s=0}
+ \ln\big(\Vert Y\Vert_A \cdot \im(z)\big) - r\ln \nu_{z}(w) - r \cdot 
\sum_{\subfrac{0\neq \lambda \in Y}{\nu_{z}(\lambda)\leq \nu_{z}(w)}}
\ln \nu_{z}\Big(1-\frac{w}{\lambda}\Big)
 \\
&=
- \frac{r}{q_\infty^r-1}\cdot \ln |\Delta^Y(z)| + r \ln |\kfk^Y_w(z)|
\end{align*}
as desired.
\end{proof}

\begin{rem}\label{rem: KLF-reform}
${}$
\begin{itemize}
\item[(1)] The first limit formula \eqref{eqn: KLF-1} can be reformulated into
\[
\frac{\partial}{\partial s}\EE^Y(z,s)\Big|_{s=0} = -\ln \eta_r^Y(z),
\]
where $\eta_r^Y$ is the absolute discriminant introduced in Definition~\ref{defn: eta}.
\item[(2)]
When $w$ is in $k^r\setminus Y$,
put $\ufk_w^Y\assign (\kfk_w^Y)^{q_\infty^r-1}/\Delta^Y$, which is a ``Drinfeld--Siegel unit'' on $\Omega_{\infty}^r$ (see \cite[Section~4.1]{Wei20}).
Then the second-limit formula in \eqref{eqn: KLF-2} can be rewritten as
\[
\frac{\partial}{\partial s} \EE^Y(z,w,s)\Big|_{s=0} = \frac{r}{q^r_\infty-1} \ln |\ufk_w^Y(z)|.
\]
\item[(3)]
The formulas in Theorem~\ref{thm: KLM-1} can be also extended to an ``adelic'' version as demonstrated in \cite[Theorem~4.4]{Wei20}. However, as the primary focus of this paper is unrelated to automorphic forms, we choose to present the formulas in their current form, which serves as a precise analogue of the classical case. This approach is sufficient for the purposes of our discussion.
\end{itemize}
\end{rem}

%

\section{Kronecker-type limit formula: several variable case}\label{sec: KLF-s}

Let $K$ be a finite extension of $k$. As the constant field of $k$ is perfect, one has that (see \cite[Theorem~7.6]{Ros})
\[
K_\infty \assign K\otimes_k k_\infty \cong \prod_{i=1}^m K_{\infty_i},
\]
where $\infty_1,...,\infty_m$ are the places of $K$ lying above $\infty$, and $K_{\infty_i}$ is the completion of $K$ at $\infty_i$ for $1\leq i \leq m$. Put $n_i \assign [K_{\infty_i}:k_\infty]$.
Let $\Ocal$ be an $A$-order in $K$.
Given a positive integer $r$ and a projective $\Ocal$-module $\Lcal$ of rank $r$ in $K^r$,
we are interested in the following ``non-holomorphic'' Eisenstein series in several variables: let $\Omega_{\infty_i}^r$ be the Berkovich Drinfeld period domain of rank $r$ over $K_{\infty_i}$ for $1\leq i \leq m$.
For $\boz = (z^{(1)},\cdots \hspace{-0.03cm},z^{(m)}) \in \prod_{i=1}^m \Omega_{\infty_i}^r$, define
\[
\EE_{\Ocal}^{\Lcal}(\boz,s) \assign 
\Vert \Lcal \Vert_{\Ocal}^s \cdot \sum_{\lambda \in (\Lcal-\{0\})/\Ocal^\times} \left(\prod_{i=1}^m \frac{\im(z^{(i)})^s}{|\lambda|_{z^{(i)}}^{rs}}\right), \quad \forall \re(s)>1.
\]
Here 
$\Vert \Lcal \Vert_{\Ocal} \assign |a|_{\infty}^{-rn} \cdot \#(\Ocal^r/a \Lcal)$ for nonzero $a \in A$ with $a \Lcal \subset \Ocal^r$, which is independent of the chosen $a \in A$.

Following a similar argument as Lemma~\ref{lem: converge}, it can be verified that $\EE_{\Ocal}^{\Lcal}(\boz,s)$ converges absolutely for $\re(s)>1$.
Here we derive a conceptual link between $\EE_{\Ocal}^{\Lcal}(\boz,s)$ and a corresponding ``one-variable'' Eisenstein series, which assures the convergence property of $\EE_{\Ocal}^{\Lcal}(\boz,s)$ as well.

\subsection{Integral representation of the non-holomorphic Eisenstein series}

For the rest of this paper, we fix an identification
\begin{equation}\label{eqn: iota-K}
\iota_{K/k} \assign \prod_{i=1}^m \iota_{i} : \prod_{i=1}^m K_{\infty_i} \stackrel{\sim}{\longrightarrow} k_{\infty}^{n_i}
\end{equation}
once and for all (where $\iota_i \assign \iota_{K_{\infty_i}/k_{\infty}}:K_{\infty_i}\cong k_{\infty}^{n_i}$ is as in \eqref{eqn: E-F-iso}),
and extends to  
\begin{equation}\label{eqn: iota-r}
\iota_{K/k,r} \assign \prod_{i=1}^m \iota_{i,r}: \prod_{i=1}^m K_{\infty_i}^r \stackrel{\sim}{\longrightarrow} \prod_{i=1}^m k_{\infty}^{n_ir} \quad 
\text{ for every $r \in \NN$,}
\end{equation}
where 
\[
\iota_{i,r}(x^{(i)}_1,\cdots \hspace{-0.03cm}, x^{(i)}_r) \assign \big(\iota_i(x^{(i)}_1),\cdots \hspace{-0.03cm}, \iota_i(x^{(i)}_r)\big),
\quad \forall (x^{(i)}_1,\cdots \hspace{-0.03cm}, x^{(i)}_r) \in K_{\infty_i}^r.
\]
Recall that we let $\text{res}: \prod_{i=1}^m\Ncal^r(K_{\infty_{i}}) \hookrightarrow \prod_{i=1}^m \Ncal^{n_ir}(k_{\infty})$ be the \lq\lq restriction of scalars\rq\rq\ in Corollary~\ref{cor: sev-one}.
For $\mathbf{t} = (t_1,\cdots \hspace{-0.03cm},t_m) \in \RR^m$ and $\nu_{\boz} = (\nu_{z^{(1)}},\cdots \hspace{-0.03cm},\nu_{z^{(m)}}) \in \prod_{i=1}^m\Ncal^r(K_{\infty_{i}})$,
the norm $\text{res}(\nu_{\boz})(\mathbf{t})$ on $\prod_{i=1}^m k_{\infty}^{n_ir} \cong k_{\infty}^{nr}$ is given by:
\begin{eqnarray}
\text{res}(\nu_{\boz})(\mathbf{t}) (x_1,\cdots \hspace{-0.03cm},x_m) &=& \max\Big(\exp(t_i)\cdot \text{res}(\nu_{z^{(i)}})(x_i)\Big) \nonumber \\
&=& \max\Big(\exp(t_i)\cdot \nu_{z^{(i)}}\big(x^{(i)}\big)^{1/n_i}\Big), \quad \forall (x_1,\cdots \hspace{-0.03cm},x_m) \in \prod_{i=1}^m k_{\infty}^{n_ir}, \nonumber
\end{eqnarray}
where $x^{(i)} = \iota_{i,r}^{-1}(x_i) \in (K_{\infty_{i}})^r$.
By Remark~\ref{rem: nu-z}, the section $\tilde{\varsigma}:\Ncal^{nr}(k_{\infty})\rightarrow \widetilde{\Omega}_{\infty}^{nr}$ sends
${\rm res}(\nu_{\boz})({\bf t})$ to a point in $\widetilde{\Omega}_{\infty}^{nr}$. We denote this point by $\tilde{\bf z}(\boz;{\bf t})$ and let ${\bf z}(\boz;{\bf t})$ be the corresponding class in $\Omega_{\infty}^{nr}$.

On the other hand, regarding $K$ as a $k$-vector space of dimension $n$, we fix a $k$-basis
\begin{equation}\label{eqn: beta-0}
\beta_{\circ} = \{\lambda_1,...,\lambda_{n}\} \quad \text{of $K$ once and for all.}
\end{equation}
For each $r \in \NN$, let $\{e_1,...,e_r\}$ be the standard $K$-basis of $K^r$.
Then
\begin{equation}\label{eqn: beta-0-r}
\beta_{\circ,r} \assign \big\{ \lambda_t e_\ell \ \big|\ 1\leq t \leq n,\ 1\leq \ell \leq r\} \quad 
\text{ is a $k$-basis of $K^r$.}
\end{equation}
From the embedding 
$$K^r \hookrightarrow  K^r\otimes_k k_{\infty} \cong \prod_{i=1}^m K^r_{\infty_{i}} \ \underset{\iota_{K/k,r}}{\overset{\sim}{\longrightarrow}} \ \ \prod_{i=1}^m k_{\infty}^{n_ir} \cong k_{\infty}^{n_ir},$$
each $\lambda_t e_{\ell}$ can be regarded as a row vector in $k_{\infty}^{nr}$ via $\iota_{K/k,r}$. Define
\begin{equation}\label{eqn: g-0-r}
g_{\circ,r} =
\begin{pmatrix}
\iota_{K/k,r}(\lambda_1 e_1) \\
\vdots \\
\iota_{K/k,r}(\lambda_n e_1) \\
\vdots \\
\vdots \\
\iota_{K/k,r}(\lambda_1 e_r) \\
\vdots \\
\iota_{K/k,r}(\lambda_n e_r)
\end{pmatrix}  =
\begin{pNiceArray}[first-row]{ccccccc}
\multicolumn{3}{c}{\overbrace{\hspace{3cm}}^{n_1r\ \text{columns}}} & & 
\multicolumn{3}{c}{\overbrace{\hspace{3cm}}^{n_mr\ \text{columns}}} \\
\iota_{1}(\lambda_1) & \cdots & 0 &  & \iota_{m}(\lambda_1) & \cdots  & 0 \\
\vdots  &   &\vdots  & \cdots \cdots & \vdots &  &\vdots  \\
\iota_{1}(\lambda_n)& \cdots & 0 & & \iota_{m}(\lambda_n) & \cdots & 0 \\
 & \vdots & & & & \vdots & \\
 & \vdots & & & & \vdots & \\
0 & \cdots & \iota_{1}(\lambda_1) & & 0 & \cdots & \iota_{m}(\lambda_1) \\
\vdots &  & \vdots & \cdots \cdots & \vdots &  & \vdots \\
0 & \cdots & \iota_{1}(\lambda_n) & & 0 & \cdots & \iota_{m}(\lambda_n) 
\end{pNiceArray}
\ \in  \GL_{nr}(k_{\infty}).
\end{equation}
In particular, one has that
\begin{equation}\label{eqn: g-beta}
g_{\circ} \assign g_{\circ,1} = 
\begin{pmatrix} \iota_{K/k}(\lambda_1) \\ \vdots \\ \iota_{K/k}(\lambda_{n})\end{pmatrix} =
\begin{pNiceArray}[first-row, code-for-first-row = \color{gray}]{ccc}
\RowStyle{\scriptstyle} \RowStyle{\text{\small}} n_1 \text{ col.} & & \RowStyle{\scriptstyle} \RowStyle{\text{\small}} n_m \text{ col.} \\
\iota_{1}(\lambda_1) & \cdots & \iota_{m}(\lambda_1) \\
\vdots & &  \vdots \\
\iota_{1}(\lambda_n) & \cdots & \iota_{m}(\lambda_n)
\end{pNiceArray}
\in \GL_{n}(k_{\infty}).
\end{equation}

\begin{lem}\label{lem: norm-comp}
Following the notations as above, we have
$$|\det g_{\circ,r}|_{\infty}\cdot \Vert \Lcal\Vert_{\beta_{\circ,r}} = \Vert \Lcal \Vert_{\Ocal}\cdot \cfk(\Ocal) \cdot q^{(g(K)-1)[\FF_K:\FF_q]r-nr(g(k)-1)}.$$
Here:
\begin{itemize}
    \item $\Vert \Lcal \Vert_{\beta_{\circ,r}}\assign |a|_{\infty}^{-rn} \cdot \#\big((A\lambda_1e_1+\cdots + A\lambda_{n}e_r)/a\Lcal\big)$ for a(ny) nonzero $a \in A$ so that $a \Lcal \subset A\lambda_1e_1+\cdots+A\lambda_{n}e_r$;
    \item $\cfk(\Ocal)\assign \#(\Ocal_K/\Ocal)$ where $\Ocal_K$ is the integral closure of $A$ in $K$;
    \item $g(K)$ and $g(k)$ are the genera of $K$ and $k$, respectively;
    \item $\FF_K$ is the constant field of $K$ (i.e.\ the algebraic closure of $\FF_q$ in $K$).
\end{itemize}
\end{lem}

\begin{proof}
It is observed that $|\det g_{\circ,r}|_{\infty} \cdot \Vert \Lcal \Vert_{\beta_{\circ,r}} \cdot \Vert \Lcal \Vert_{\Ocal}^{-1}$ is independent of the chosen $k$-basis $\beta_{\circ}$ and the $\Ocal$-module $\Lcal$.
Thus by Remark~\ref{rem: ind-base}, we may assume that $r=1$, $\Lcal = \Ocal$, and $g_{\circ} \in \GL_{n}(O_{\infty})$ without loss of generality. 
In this case, the lattice $O_{\infty} \lambda_1 + \cdots + O_{\infty}\lambda_n$ in $K\otimes_k k_{\infty} \cong \prod_{i=1}^m K_{\infty_i}$ is actually equal to $\prod_{i=1}^{m}O_{K_{\infty_i}}$.
Therefore for $a \in A$ with $|a|_{\infty}$ sufficiently large so that 
\[
a\lambda_1,...,a\lambda_n \in \Ocal
\quad \text{ and } \quad  k_{\infty}^n = \Ocal + (O_{\infty} a \lambda_1 +\cdots + O_{\infty} a \lambda_n),
\]
we obtain that
\[
\frac{\Ocal_K \cap (O_{\infty} a\lambda_1 + \cdots + O_{\infty} a \lambda_n)}{\Ocal \cap (O_{\infty} a\lambda_1 + \cdots + O_{\infty} a \lambda_n)} \cong \frac{\Ocal_K}{\Ocal},
\]
whence
\begin{align*}
\frac{|\det g_{\circ}|_{\infty}\cdot \Vert \Ocal \Vert_{\beta_{\circ}}}{\Vert \Ocal\Vert_{\Ocal}}
&= |a|_{\infty}^n \cdot \#\left(\frac{\Ocal}{A a\lambda_1+\cdots + A a\lambda_n}\right)^{-1} \\
\text{(by Rieman--Roch theorem for $k$)}\hspace{0.2cm}  &= |a|_{\infty}^n \cdot q^{n(1-g(k))} \cdot \#\Big(\Ocal \cap a \cdot (O_{\infty} \lambda_1 + \cdots O_{\infty}\lambda_n) \Big)^{-1} \\
&= |a|_{\infty}^n \cdot q^{n(1-g(k))} \cdot \frac{\cfk(\Ocal)}{\#\Big(\Ocal_K \cap a \cdot (O_{\infty} \lambda_1 + \cdots O_{\infty}\lambda_n) \Big)} \\
\text{(by Rieman--Roch theorem for $K$)}\hspace{0.2cm}  &=
|a|_{\infty}^n \cdot q^{n(1-g(k))} \cdot \cfk(\Ocal)\cdot |a|_{\infty}^{-n} \cdot q^{[\FF_K:\FF_q]\cdot (g(K)-1)} \\
&= \cfk(\Ocal)\cdot q^{[\FF_K:\FF_q]\cdot (g(K)-1)-n(g(k)-1)}
\end{align*}
as desired.

\end{proof}

Finally, we let $d:\RR \rightarrow \RR^m$ be the diagonal embedding and define $\ln: K^{\times}\rightarrow \RR^m$ by
\[
\ln(\alpha) \assign (\ln|\alpha|_1,\cdots \hspace{-0.03cm},\ln|\alpha|_m), \quad \forall \alpha \in K^\times,
\]
where $|\cdot|_i : K_{\infty_{i}}\rightarrow \RR_{\geq 0}$ is the extension of the absolute value $|\cdot|_{\infty}$ on $k_{\infty}$ to $K_{\infty_i}$.
We point out that the normalized absolute value on $K_{\infty_i}$ is $|\cdot|_{i}^{n_i}$ for $1\leq i \leq m$.
Set 
\[
\Tfk(\Ocal)\assign \frac{\RR^m}{\ln(\Ocal^\times) + d(\RR)},
\]
which is a compact set by the Dirichlet unit theorem.
Moreover, let $d\mathbf{t}$ be the quotient measure on $\RR^m/d(\RR)$ with respect to the usual Lebesgue measure on $\RR^m$ and $d(\RR)$.
One can check that
\begin{equation}\label{eqn: vol-T}
\text{vol}(\Tfk(\Ocal),d\mathbf{t}) = \frac{n}{n_1\cdots n_m} \cdot \Rcal(\Ocal),
\end{equation}
where $\Rcal(\Ocal)$ is the (usual) regulator of $\Ocal$ defined as follows.
Let $\FF_{\Ocal}\assign \FF_K \cap \Ocal$ and $\{u_1,...,u_{m-1}\}$ be a set of generators of $\Ocal^\times/\FF_{\Ocal}^\times$. Then $\Rcal(\Ocal)\assign 1$ if $m=1$ and
\begin{equation}\label{eqn: R(O)}
\Rcal(\Ocal)\assign  \Big|\det\begin{pmatrix} n_1\ln |u_1|_1 & \cdots &n_{m-1}\ln |u_1|_{m-1} \\ \vdots & & \vdots \\ n_1\ln|u_{m-1}|_1 & \cdots &n_{m-1} \ln |u_{m-1}|_{m-1}\end{pmatrix}\Big| \quad \text{ if $m>1$}.
\end{equation}
Based upon the fixed $k$-basis $\beta_{\circ,r}$ of $K^r$ in \eqref{eqn: beta-0-r} which identifies $K^r$ with $k^{nr}$, let $Y_{\Lcal} \subset k^{nr}$ correspond to $\Lcal$ with respect to $\beta_{\circ,r}$, i.e.
\begin{equation}\label{eqn: YL}
\Lcal = \left\{\sum_{1\leq \ell \leq r}\sum_{1\leq t \leq n} x_{t \ell} \lambda_{t}e_{\ell}\ \Bigg|\  (x_{11},\cdots \hspace{-0.03cm}, x_{1r}, \cdots \hspace{-0.03cm}, x_{n1}, \cdots x_{nr}) \in Y_{\Lcal}
\right\}.
\end{equation}
We then obtain that:

\begin{thm}\label{Key}
Let $q_{\Ocal}\assign \#(\FF_\Ocal)$.
Given $\boz =(z^{(1)},\cdots \hspace{-0.03cm}, z^{(m)})\in \prod_{i=1}^m \Omega_{\infty_{i}}^r$ and $s \in \CC$ with $\re(s)>1$, 
the following equality holds:
$$
\frac{(rs)^{m-1}}{q_{\Ocal}-1} \cdot \frac{n_1\cdots n_m}{n}  \cdot \int_{\Tfk(\Ocal)} \EE^{Y_{\Lcal}}\big(g_{\circ,r}\cdot {\bf z}(\boz;{\bf t}), s\big) d \mathbf{t}
=
\dfk(\Ocal/A)^{\frac{rs}{2}}\cdot \EE_{\Ocal}^{\Lcal}(\boz,s),
$$
where
\begin{equation}\label{eqn: d(O/A)}
\dfk(\Ocal/A) \assign \cfk(\Ocal)^2\cdot q^{(2g(K)-2)[\FF_K:\FF_q]-2n(g(k)-1)} \cdot \prod_{i=1}^m \dfk(K_{\infty_i}/k_{\infty}).
\end{equation}
\end{thm}

\begin{proof}
Keep the notation in the above discussion.
We have that
\begin{align*}
\EE^{Y_{\Lcal}}\big(g_{\circ,r} \cdot \mathbf{z}(\boz;\mathbf{t}),s\big)
&= \sum_{0\neq x \in Y_{\Lcal}} \frac{\Vert Y_{\Lcal}\Vert_A^s \cdot  |\det g_{\circ,r}|_{\infty}^s \cdot D\big(\res(\nu_{\boz})(\mathbf{t})\big)^s}{\big(\res(\nu_{\boz})(\mathbf{t})(xg_{\circ,r})\big)^{nrs}} \\ 
&= \Vert \Lcal\Vert_{\beta_{\circ,r}}^s \cdot |\det g_{\circ,r}|_{\infty}^s \cdot \sum_{0\neq \lambda \in \Lcal} \frac{D\big(\res(\nu_{\boz})(\mathbf{t})\big)^s}{\Big(\res(\nu_{\boz})(\mathbf{t})\big(\iota_{K/k,r}(\lambda)\big)\Big)^{nrs}}, \quad \re(s)>1.
\end{align*}
Hence for $\re(s)>1$, we obtain that
\begin{align*}
&\hspace{-0.5cm} \int_{\Tfk(\Ocal)} \EE^{Y_{\Lcal}}\big(g_{\circ,r}\cdot \mathbf{z}(\boz;\mathbf{t}),s\big)d\mathbf{t}\\
&= \Vert \Lcal\Vert_{\beta_{\circ,r}}^s \cdot |\det g_{\circ,r}|_{\infty}^s \cdot \int_{\Tfk(\Ocal)}
\left(\sum_{0\neq \lambda \in \Lcal} \frac{D\big(\res(\nu_{\boz})(\mathbf{t})\big)^s}{\Big(\res(\nu_{\boz})(\mathbf{t})\big(\iota_{K/k,r}(\lambda)\big)\Big)^{nrs}}\right)d\mathbf{t} \\
&= (q_{\Ocal}-1)\cdot \Vert \Lcal\Vert_{\beta_{\circ,r}}^s \cdot |\det g_{\circ,r}|_{\infty}^s \cdot
\sum_{\lambda \in (\Lcal-\{0\})/\Ocal^{\times}}
\int_{\RR^m/d(\RR)} 
\frac{D\big(\res(\nu_{\boz})(\mathbf{t})\big)^s}{\Big(\res(\nu_{\boz})(\mathbf{t})\big(\iota_{K/k,r}(\lambda)\big)\Big)^{nrs}} d\mathbf{t} \\
&=
\Vert \Lcal\Vert_{\beta_{\circ,r}}^s \cdot |\det g_{\circ,r}|_{\infty}^s \cdot
(rs)^{1-m} \cdot \frac{(q_{\Ocal}-1)\cdot n}{n_1\cdots n_m} \cdot 
\sum_{\lambda \in (\Lcal-\{0\})/\Ocal^{\times}}
\prod_{i=1}^m \dfk(K_{\infty_i}/k_{\infty})^{\frac{rs}{2}} \cdot \frac{D(\nu_{z^{(i)}})^s}{\nu_{z^{(i)}}(\lambda)^{rs}} \\
&= \frac{q_{\Ocal}-1}{(rs)^{m-1}} \cdot \frac{n}{n_1\cdots n_m} \cdot \dfk(\Ocal/A)^{\frac{rs}{2}} \cdot \EE^{\Lcal}_{\Ocal}(\boz,s),
\end{align*}
where the third equality comes from Corollary~\ref{cor: sev-one} and the last equality is by Lemma~\ref{lem: norm-comp}.
Therefore the proof is complete.
\end{proof}

\begin{rem}\label{rem: HC}
${}$

(1) For each $\boz \in \prod_{i=1}^m \Omega_{\infty_{i}}^r$, 
we introduce the \lq\lq Heegner cycle\rq\rq\ associated with $\boz$ and $\Lcal$ on $\GL(Y_{\Lcal}) \backslash \Omega_{\infty}^{nr}$ in the following:
$$\Zfk_{\Lcal}(\boz) \assign \left\{ [g_{\circ,r}\cdot {\bf z}(\boz;{\bf t})] \in \GL(Y_{\Lcal}) \backslash \Omega_{\infty}^{nr} \ \bigg|\ \mathbf{t} \in \Tfk(\Ocal)\right\}.$$
The construction depends on the fixed identification $\iota_{K/k,r}$ in \eqref{eqn: iota-r} and the $k$-basis $\beta_{\circ,r}$ of $K^r$ in \eqref{eqn: beta-0-r}.
Theorem~\ref{Key} says that the several-variable Eisenstein series $\EE_{\Ocal}^{\Lcal}(\boz,s)$ can be expressed in terms of an integral of the corresponding one-variable Eisenstein series along the Heegner cycle $\Zfk_{\Lcal}(\boz)$.

(2) Suppose $K_{\infty_i}/k_{\infty}$ is tamely ramified for each $i$ with $1\leq i \leq m$, which says in particular that $K/k$ is separable.
By the Riemann--Hurwitz formula, we have that $\dfk(\Ocal/A)$ equals to the norm of the discriminant ideal of $\Ocal$ over $A$.
\end{rem}

Theorem~\ref{Key} also provides the meromorphic continuation of $\EE^L_{\Ocal}(\boz,s)$ to the whole complex $s$-plane.
Consequently, for a rank $nr$ projective $A$-module $Y\subset k^{nr}$ and $z \in \Omega_{\infty}^{nr}$, recall that the absolute discriminant $\eta_{nr}^Y(z)$ is given in Definition~\ref{defn: eta} by
\begin{equation}\label{eqn: eta}
\eta_{nr}^Y(z)= \Vert Y\Vert_{A} \cdot \im(z) \cdot |\Delta^Y(z)|^{\frac{nr}{q_{\infty}^{nr}-1}}.    
\end{equation}
Then we arrive at the following Kronecker-type limit formula of $\EE_{\Ocal}^{\Lcal}(\boz,s)$:

\begin{thm}\label{mainthm}
Keep the above notation.
For each $\boz \in \prod_{i=1}^m \Omega_{K_{\infty_i}}^r$,
we have 
$$\ord_{s=0}\EE_{\Ocal}^{\Lcal}(\boz,s) = m-1.$$
Moreover, let
$$\widetilde{\EE}_{\Ocal}^{\Lcal}(\boz,s) \assign \frac{\dfk(\Ocal/A)^{rs/2}}{(rs)^{m-1}} \cdot \EE_{\Ocal}^{\Lcal}(\boz,s).
$$
Then 
$$\widetilde{\EE}_{\Ocal}^{\Lcal}(\boz,0) = - \frac{ \Rcal(\Ocal)}{q_{\Ocal}-1} 
\quad  \text{and}\quad
\frac{\partial}{\partial s} \widetilde{\EE}_{\Ocal}^{\Lcal}(\boz,s)\Big|_{s=0}
=-\frac{n_1\cdots n_m}{n(q_{\Ocal}-1)} \int_{\Tfk(\Ocal)} \ln \eta_{nr}^{Y_{\Lcal}}\big(g_{\circ,r}\cdot {\bf z}(\boz;\mathbf{t})\big) d \mathbf{t}.$$
\end{thm}

\begin{rem}
From Theorem~\ref{mainthm}, we then obtain that
$$\frac{ (\widetilde{\EE}_{\Ocal}^{\Lcal})'(\boz,0)}{\widetilde{\EE}_{\Ocal}^{\Lcal}(\boz,0)}
= \frac{1}{\text{vol}(\Tfk(\Ocal), d\mathbf{t})} \cdot \int _{\Tfk(\Ocal)} \ln \eta^{Y_{\Lcal}}_{nr}\big(g_{\circ,r}\cdot \mathbf{z}(\boz;\mathbf{t})\big) d \mathbf{t},$$
which says that the logarithmic derivative of $\widetilde{\EE}_{\Ocal}^{\Lcal}(\boz,s)$ at $s=0$ is the average integral of the logarithm of the absolute discriminant $\eta^{Y_{\Lcal}}_{nr}$ along the Heegner cycles $\Zfk_{\Lcal}(\boz) \subset \GL(Y_{\Lcal})\backslash \Omega_{\infty}^{nr}$.
\end{rem}

\subsection{The second limit formula in several variable case}
Keep the notation as in the previous subsection.
For each $\bow \in K^r \setminus \Lcal$,
the {\it Jacobi-type Eisenstein series in several variables} is defined as follows: for $\boz = (z^{(1)},\cdots \hspace{-0.03cm}, z^{(m)}) \in \prod_{i=1}^m \Omega_{\infty_i}^r$,
\[
\EE^{\Lcal}_{\Ocal}(\boz,\bow,s)
\assign \Vert \Lcal \Vert^s_{\Ocal} \cdot 
\sum_{\lambda_{\bow} \in (\bow+\Lcal)/\Ocal^{(1)}_{\Cfk_{\bow}}}
\left(
\prod_{i=1}^m \frac{\im(z^{(i)})^s}{\nu_{z^{(i)}}(\lambda_{\bow})^{rs}}
\right),
\]
where $\Ocal^{(1)}_{\Cfk_{\bow}}\assign\big\{u \in \Ocal^{\times} \ \big|\  u \equiv 1 \bmod \Cfk_{\bow}\big\}$ and $\Cfk_{\bow}$ is the $\Ocal$-annihilator of $\bow+\Lcal$ in $K^r/\Lcal$, i.e.
\[
\Cfk_{\bow} \assign \{ a \in \Ocal \mid a \bow \in \Lcal\}.
\]
In particular, $\Cfk_{\bow}$ is a proper ideal of $\Ocal$ (i.e. $\Cfk_{\bow}\subsetneq \Ocal$) and $\Ocal^{(1)}_{\Cfk_{\bow}}$ is torsion free.

Note that with respect to the chosen $k$-basis $\beta_{\circ,r}$ of $K^r$ in \eqref{eqn: beta-0-r}, we may write
\begin{equation}\label{eqn: w}
\bow = \Big(\sum_{1\leq t \leq n} w_{1t}\lambda_t, \cdots \hspace{-0.03cm}, \sum_{1\leq t \leq n} w_{nt} \lambda_t\Big) = \sum_{1\leq \ell \leq r}\sum_{1\leq t \leq n}w_{\ell t}\lambda_t e_\ell  
\end{equation}
with $w_{11},...,w_{1n},...,w_{r1},...,w_{rn} \in k$.
Take $w \assign (w_{11},\cdots \hspace{-0.03cm}, w_{1n}, \cdots \hspace{-0.03cm}, w_{r1}, \cdots \hspace{-0.03cm}, w_{rn}) \in k^{nr}$, and
set
\[
\Tfk(\Ocal^{(1)}_{\Cfk_{\bow}}) \assign \frac{\RR^m}{\ln (\Ocal^{(1)}_{\Cfk_{\bow}})+d(\RR)},
\]
which has volume
\[
{\rm vol}\big(\Tfk(\Ocal^{(1)}_{\Cfk_{\bow}}),d\mathbf{t}\big) = \frac{n}{n_1\cdots n_m} \cdot \Rcal(\Ocal)\cdot \frac{[\Ocal^{\times}:\Ocal^{(1)}_{\Cfk_{\bow}}]}{q_{\Ocal}-1}.
\]
Following the same argument in Theorem~\ref{Key}, we obtain that:

\begin{thm}\label{Key-2}
Keep the above notation. For every $\boz \in \prod_{i=1}^m \Omega_{\infty_i}^r$ and $\bow \in K^r\setminus \Lcal$, the Jacobi-type Eisenstein series $\EE^{\Lcal}_{\Ocal}(\boz,\bow,s)$ converges absolutely for $s \in \CC$ with $\re(s)>1$, and the following integral representation holds:
\[
(rs)^{m-1}\cdot \frac{n_1\cdots n_m}{n} \cdot 
\int_{\Tfk(\Ocal^{(1)}_{\Cfk_{\bow}})}
\EE^{Y_{\Lcal}}\big(g_{\circ,r}\cdot \mathbf{z}(\boz;\mathbf{t}),w,s\big)d \mathbf{t}
= \dfk(\Ocal/A)^{\frac{rs}{2}} \cdot \EE^{\Lcal}_{\Ocal}(\boz,\bow,s).
\]
\end{thm}

The above theorem provides the meromorphic continuation of $\EE^{\Lcal}_{\Ocal}(\boz,\bow,s)$ to the whole complex $s$-plane.
Moreover, by Theorem~\ref{thm: KLM-1} and the reformulation in Remark~\ref{rem: KLF-reform}, we arrive at the Kronecker-type second limit formula in several variables:

\begin{thm}\label{mainthm-2}
Let $\widetilde{\EE}_{\Ocal}^{\Lcal}(\boz,\bow,s)\assign (rs)^{1-m} \cdot \dfk(\Ocal/A)^{rs/2} \cdot \EE_{\Ocal}^{\Lcal}(\boz,\bow,s)$. Then $\widetilde{\EE}_{\Ocal}^{\Lcal}(\boz,\bow,0)=0$ and
\[
\frac{\partial}{\partial s}\widetilde{\EE}_{\Ocal}^{\Lcal}(\boz,\bow,s)\bigg|_{s=0} = -\frac{(n_1\cdots n_m) \cdot r}{q_{\infty}^{nr}-1} \cdot 
\int_{\Tfk(\Ocal^{(1)}_{\Cfk_{\bow}})} \ln \big|\ufk_{w}^{Y_{\Lcal}}\big(g_{\circ,r}\cdot\mathbf{z}(\boz;\mathbf{t})\big)\big| d \mathbf{t},
\]
where $\ufk_{w}^{Y_{\Lcal}}$ is the Drinfeld--Siegel unit introduced in Remark~\ref{rem: KLF-reform}~(2).
\end{thm}

\begin{rem}\label{rem: HC-2}
Given a nonzero ideal $\Cfk \subsetneq \Ocal$, recall that $\Ocal_{\Cfk}^{(1)} = \{u \in \Ocal^{\times} \mid u \equiv 1 \bmod \Cfk\}$.
With respect to the fixed $k$-basis $\beta_{\circ,r}$ of $K^r$, we may embed $\Ocal^{\times}$ into $\GL(Y_{\Lcal})$.
In particular, let $\cfk= \Cfk \cap A$ and put
\[
\Gamma^{Y_{\Lcal}}(\Cfk) \assign \big\{ \gamma \in \GL(Y_{\Lcal}) \ \big|\ \gamma \equiv u \bmod \cfk \quad \text{for some $u \in \Ocal_{\Cfk}^{(1)}$}\big\},
\]
which is a subgroup of $\GL(Y_{\Lcal})$ containing
$\Gamma^{Y_{\Lcal}}(\cfk) \assign \big\{ \gamma \in \GL(Y_{\Lcal}) \ \big| \ \gamma \equiv 1 \bmod \cfk\big\}$.
We call $\Gamma^{Y_{\Lcal}}(\Cfk)$ the \emph{congruence subgroup modulo $\Cfk$}.

For each $\boz \in \prod_{i=1}^m \Omega_{\infty_{i}}^r$ and $\bow \in K^r \setminus \Lcal$, recall that $\Cfk_{\bow}=\big\{\alpha \in \Ocal\ \big|\ \alpha \cdot \bow \in \Lcal\}$.
We may also consider the following \lq\lq Heegner cycle\rq\rq\ associated with $\boz$, $\Lcal$, and $\Cfk_{\bow}$ on $\Gamma^{Y_{\Lcal}}(\Cfk_{\bow}) \backslash \Omega_{\infty}^{nr}$:
$$\Zfk_{\Lcal}(\boz;\Cfk_{\bow}) \assign \left\{ [g_{\circ,r}\cdot {\bf z}(\boz;{\bf t})] \in \Gamma^{Y_{\Lcal}}(\Cfk_{\bow}) \backslash \Omega_{\infty}^{nr} \ \bigg|\ \mathbf{t} \in \Tfk(\Ocal^{(1)}_{\Cfk_{\bow}})\right\}.$$
Then Theorem~\ref{Key} says that the derivative of the several-variable Jacobi-type Eisenstein series $\widetilde{\EE}_{\Ocal}^{\Lcal}(\boz,\bow,s)$ at $s=0$ can be expressed in terms of an integral of $\ln \big|\ufk_{w}^{Y_{\Lcal}}(\cdot)\big|$ along the Heegner cycle $\Zfk_{\Lcal}(\boz;\Cfk_{\bow})$.    
\end{rem}

\section{Kronecker term of the Dirichlet \texorpdfstring{$L$}{L}-functions}\label{sec: KT-app}

Let $K$ be a finite extension of $k$ and $\Ocal$ be an $A$-order in $K$. 
In this section, we shall apply the Kronecker-type limit formula in Theorem~\ref{mainthm} to determine the ``Kronecker term'' of the Dirichlet $L$-function associated with the characters on the Picard group $\Pic(\Ocal)$.
We start with the case of the trivial character, i.e.~the Dedekind--Weil zeta function in the following.

\subsection{Dedekind--Weil zeta function}\label{sec: app-zeta}
We first note that every invertible ideal $\Ifk$ of $\Ocal$ can be regarded as a projective $\Ocal$-submodule of $K$.
Keep the notation as in the previous section, let $\zeta_{\Ocal}(s)$ be the Dedekind-Weil zeta function of the $A$-order $\Ocal$ in $K$: for $\re(s)>1$,
\begin{eqnarray}
\zeta_{\Ocal}(s)
&\assign& \sum_{\subfrac{\text{invertible ideal}}{\text{$\Ifk$ of $\Ocal$}}}\frac{1}{\Vert \Ifk \Vert_{\Ocal}^s}
\ = \sum_{\Ical \in \Pic(\Ocal)} \zeta_{\Ocal}(s; \Ical), \nonumber 
\end{eqnarray}
where for each $\Ical \in \Pic(\Ocal)$,
\begin{eqnarray}\label{eqn: L-E-1}
\zeta_{\Ocal}(s;\Ical) &\assign &\sum_{\Ifk \in \Ical} \frac{1}{\Vert \Ifk \Vert_{\Ocal}^s} \nonumber \\
&=& \Vert \Ifk \Vert_{\Ocal}^{-s} \cdot \sum_{ \lambda \in (\Ifk^{-1}-\{0\})/\Ocal^\times} \frac{1}{\#(\Ocal/\lambda \Ocal)^{s}}\nonumber \\
&=& \Vert \Ifk^{-1} \Vert_{\Ocal}^s \cdot \sum_{\lambda \in (\Ifk^{-1}-\{0\})/\Ocal^\times} \frac{1}{|N_{K/k}(\lambda)|_{\infty}^s}\nonumber \\
&=& \Vert \Ifk^{-1} \Vert_{\Ocal}^s \cdot \sum_{\lambda \in (\Ifk^{-1}-\{0\})/\Ocal^\times} \left(\prod_{i=1}^m \frac{1}{|\lambda|_i^{n_is}}\right) \nonumber \\
&=& \EE_{\Ocal}^{\Ifk^{-1}}(\boz_K,s),
\end{eqnarray}
where $\boz_K$ is the unique point in $\prod_{i=1}^m \Omega_{\infty_i}^1$ represented by 
\[
\nu_K \assign \big(|\cdot|_1^{n_1},\cdots \hspace{-0.03cm},|\cdot|_m^{n_m}\big) \in \prod_{i=1}^m \Ncal^1(K_{\infty_i}).
\]
On the other hand,
For each $\mathbf{t} \in \RR^m$, the norm $\res(\nu_K)(\mathbf{t}) \in \Ncal^n(k_{\infty})$ corresponds to a point $\tilde{\mathbf{z}}(\boz_K;\mathbf{t})\assign \tilde{\varsigma}\big(\res(\nu_K)(\mathbf{t})\big)$ on $\widetilde{\Omega}_{\infty}^{n}$ by Remark~\ref{rem: nu-z}.
Let $\mathbf{z}(\boz_K;\mathbf{t}) \in \Omega_{\infty}^n$ be the point represented by $\tilde{\mathbf{z}}(\boz_K;\mathbf{t})$.
Recall in \eqref{eqn: beta-0} that we have fixed a $k$-basis $\beta_{\circ}=\{\lambda_1,...,\lambda_n\}$ of $K$ (regarding as a $k$-vector space of dimension $n$).
From the identification $\iota_{K/k}:\prod_{i=1}^m K_{\infty_i} \cong \prod_{i=1}^m k_{\infty}^{n_i} = k_{\infty}^{n}$ fixed in
\eqref{eqn: iota-K} , we may regard $\lambda_i$ as a row vector $\iota_{K/k}(\lambda_i)$ in $k_{\infty}^n$ for each $1\leq i \leq n$.
As in \eqref{eqn: g-beta}, we put
\begin{equation*}
g_{\circ}= \ (g_{\circ,1} =)\ 
\begin{pmatrix}
    \iota_{K/k}(\lambda_1) \\
    \vdots \\
    \iota_{K/k}(\lambda_n)
\end{pmatrix}
=
\begin{pmatrix}
\iota_{1}(\lambda_1) & \cdots & \iota_{m}(\lambda_1) \\
\vdots & &  \vdots \\
\iota_{1}(\lambda_n) & \cdots & \iota_{m}(\lambda_n)
\end{pmatrix}
\in \GL_n(k_{\infty}).
\end{equation*}
For each invertible ideal $\Ifk$ of $\Ocal$, let $Y_{\Ifk}$ be the $A$-submodule of $k^n$ corresponding to $\Ifk$ via the identification between $K$ and $k^n$ with respect to the basis $\beta_{\circ}$ in \eqref{eqn: beta-0}.
Then by Theorem~\ref{Key} we get:
\begin{prop}\label{prop: zeta-O}
Let $K$ be a finite extension of $k$ and $\Ocal$ be an $A$-order in $K$.
For each ideal class $\Ical \in \Pic(\Ocal)$ and $\Ifk \in \Ical$, we have that
$$
\zeta_{\Ocal}(s;\Ical) = \frac{s^{m-1}}{q_{\Ocal}-1} \cdot \frac{n_1\cdots n_m}{n} \cdot \dfk(\Ocal/A)^{-\frac{s}{2}} \cdot 
\int_{\Tfk(\Ocal)}\EE^{Y_{\Ifk^{-1}}}(g_{\circ} \cdot \mathbf{z}(\boz_K;\mathbf{t}),s)d \mathbf{t}.
$$
\end{prop}
Consider the following modified zeta function:
$$
\widetilde{\zeta}_{\Ocal}(s) \assign \frac{\dfk(\Ocal/A)^{s/2}}{s^{m-1}} \cdot \zeta_{\Ocal}(s).
$$
Then Proposition~\ref{prop: zeta-O} leads us to:

\begin{cor}\label{cor: zeta}
$$
\widetilde{\zeta}_{\Ocal}(0) = -\frac{\#\Pic(\Ocal)\cdot  \Rcal(\Ocal)}{q_{\Ocal}-1}
$$
and
$$\frac{\partial}{\partial s} \widetilde{\zeta}_{\Ocal}(s)\bigg|_{s=0} = 
-\frac{n_1\cdots n_m}{n(q_{\Ocal}-1)} \sum_{[\Ifk] \in \Pic(\Ocal)} \int_{\Tfk(\Ocal)} \ln \eta_n^{Y_{\Ifk}}\big(g_{\circ}\cdot \mathbf{z}(\boz_K;\mathbf{t})\big)d\mathbf{t}.
$$
\end{cor}

\begin{rem}
From the above Corollary, we obtain that
$$\frac{ \widetilde{\zeta}'_{\Ocal}(0)}{\widetilde{\zeta}_{\Ocal}(0)}
= \frac{1}{\#\Pic(\Ocal) \text{\rm vol}(\Tfk(\Ocal), d\mathbf{t})} \cdot \sum_{[\Ifk] \in \Pic(\Ocal)} \int _{\Tfk(\Ocal)} \ln \eta_n^{Y_{\Ifk}}\big(g_{\circ}\cdot \mathbf{z}(\boz_K;\mathbf{t})\big) d \mathbf{t}.$$
In other words,
the logarithmic derivative of $\zeta_{\Ocal}(s)$ at $s=0$ is equal to the average of the ``log-discriminant'' $\ln \eta_n^{Y_{\Ifk}}$ along the Heegner cycles associated with the ideal classes of $\Ocal$.
\end{rem}

\subsection{Dirichlet \texorpdfstring{$L$}{L}-function associated with ring class characters}\label{sec: app-ring}

Keep the notation as above. For each character $\chi: \Pic(\Ocal)\rightarrow \CC^{\times}$, the Dirichlet $L$-function associated with $\chi$ is defined by
\begin{eqnarray*}
L_{\Ocal}(s,\chi)&\assign& \sum_{\subfrac{\text{invertible ideal}}{\Ifk \subset \Ocal}} \frac{\chi([\Ifk])}{\Vert \Ifk \Vert_{\Ocal}^s}, \quad \re(s)>1 \\
&=& \sum_{\Ical \in \Pic(\Ocal)} \chi(\Ical) \cdot \zeta_{\Ocal}(s;\Ical).
\end{eqnarray*}
In particular, $L_{\Ocal}(s,\chi) = \zeta_{\Ocal}(s)$ when $\chi$ is trivial.
Moreover, 
we normalize $L_{\Ocal}(s,\chi)$ by
\[
\widetilde{L}_{\Ocal}(s,\chi)\assign \frac{\dfk(\Ocal/A)^{s/2}}{s^{m-1}}L_{\Ocal}(s,\chi).
\]
From Proposition~\ref{prop: zeta-O} we have the following Lerch-type formula:
\begin{cor}\label{cor: D-L}
Let $\chi:\Pic(\Ocal)\rightarrow \CC^{\times}$ be a nontrivial character. Then
$\widetilde{L}_{\Ocal}(0,\chi) = 0$
and
\[
\frac{\partial}{\partial s} \widetilde{L}_{\Ocal}(s,\chi)\bigg|_{s=0} = 
-\frac{n_1\cdots n_m}{n(q_{\Ocal}-1)} \sum_{[\Ifk] \in \Pic(\Ocal)} \chi([\Ifk]) \cdot \int_{\Tfk(\Ocal)} \ln \eta_n^{Y_{\Ifk^{-1}}}\big(g_{\circ}\cdot \mathbf{z}(\boz_K;\mathbf{t})\big)d\mathbf{t}.
\]
\end{cor}

\begin{rem}
${}$
\begin{itemize}
\item[(1)] Suppose $K/k$ is imaginary, i.e.~the infinite place $\infty$ of $k$ does not split in $K$.
Then $m=1$ and $\res(\nu_K)$
actually corresponds to a ``CM'' point $\mathbf{z}_K$ associated with $K$ on $\Omega_{\infty}^n$.
Therefore for each $A$-order $\Ocal$ in $K$, we have
\[
\frac{\widetilde{\zeta}_{\Ocal}'(0)}{\widetilde{\zeta}_{\Ocal}(0)} = \frac{1}{\#\Pic(\Ocal)} \cdot \sum_{[\Ifk] \in \Pic(\Ocal)} \ln \eta_n^{Y_{\Ifk}}(g_{\circ}\cdot \mathbf{z}_K),
\]
which agrees with the key formula \cite[equation~(5.3)]{Wei20} for deriving the Colmez-type formula for Drinfeld modules in \cite[Theorem~5.3]{Wei20}.
Moreover, for each nontrivial character $\chi:\Pic(\Ocal)\rightarrow \CC^{\times}$,
we have $\widetilde{L}_{\Ocal}(0,\chi) =0$ and
\[
\widetilde{L}_{\Ocal}'(0,\chi) = - \frac{1}{q_{\Ocal}-1} \sum_{[\Ifk] \in \Pic(\Ocal)}\chi([\Ifk]) \cdot \ln \eta_n^{Y_{\Ifk^{-1}}}(g_{\circ}\cdot \mathbf{z}_K).
\]

\item[(2)] Suppose $K/k$ is totally real, i.e.~$\infty$ splits completely in $K$.
Then $m=n$ and $n_i=1$ for $1\leq i \leq n$.
The formulas in Corollary~\ref{cor: zeta} and \ref{cor: D-L}
are in much simpler forms than \cite[Theorem~6.4]{Wei20} because of  the different normalizations of $\widetilde{L}_{\Ocal}(s,\chi)$ 
and the integration regions $\Tfk(\Ocal)$.


\end{itemize}
\end{rem}

\subsection{Dirichlet \texorpdfstring{$L$}{L}-function associated with ray class characters}\label{sec: app-ray}

Let $\Ocal_K$ be the integral closure of $A$ in $K$.
Given a proper ideal $\Cfk$ of $\Ocal_K$, let $\Jcal(\Cfk)$ be the group of fractional ideals of $\Ocal_K$ which are coprime to $\Cfk$, and $\Pcal_{\cfk}$ be the subgroup of $\Jcal(\Cfk)$ generated by the principal ideals $(\alpha)$ for nonzero $\alpha$ satisfying $\ord_{\Pfk}(\alpha -1)\geq \ord_{\Pfk}(\Cfk)$ for every nonzero prime ideal $\Pfk$ of $\Ocal_K$ with $\Pfk \mid \Cfk$. The {\it ray class group of $\Ocal_K$ for the modulus $\Cfk$} is ${\rm Cl}(\Ocal_K,\Cfk)\assign \Jcal(\Cfk)/\Pcal_{\Cfk}$, which is known to be a finite group. 
Moreover, 
for each integral ideal $\Ifk \in \Jcal(\Cfk)$,
we have the following bijection (cf.~\cite[(9.4) Lemma in Chap.~III]{Neu}):
\[
\begin{tabular}{ccl}
$(1+\Cfk \Ifk^{-1})/\Ocal_{K,\Cfk}^{(1)}$ & $\overset{\sim}{\longrightarrow}$ & $\big\{\text{nonzero ideals $\Jfk$ of $\Ocal_K$ with $[\Jfk] = [\Ifk] \in {\rm Cl}(\Ocal_K,\Cfk)$}\big\}$ \\
$\alpha$ & $\longmapsto$ & $\alpha \cdot \Ifk$
\end{tabular}
\]

Let $\chi: {\rm Cl}(\Ocal_K,\Cfk)\rightarrow \CC^{\times}$ be a ray class character.
Extending $\chi$ to a function on the group of fractional ideals of $\Ocal_K$ by  $\chi(\Ifk)\assign 0$ if $\Ifk$ is not coprime to $\Cfk$,
the {\it Dirichlet $L$-function associated with $\chi$} is defined by
\begin{align}\label{eqn: ray-L}
L_{\Ocal_K}^{\Cfk}(\chi,s) & \assign \sum_{\subfrac{\text{nonzero ideal}}{\Ifk \subset \Ocal_K}} \frac{\chi(\Ifk)}{\ \ \Vert \Ifk \Vert_{\Ocal_K}^s} \quad \quad \re(s)>1 \notag \\
&= \sum_{\Ical \in {\rm Cl}(\Ocal_K,\Cfk)}
\chi(\Ical) \cdot \sum_{\subfrac{\Ifk \subset \Ocal_K}{\Ifk \in \Ical}} \frac{1}{\ \ \Vert \Ifk \Vert_{\Ocal_K}^s} \notag \\
&= \sum_{[\Ifk] \in {\rm Cl}(\Ocal_K,\Cfk)}
\frac{\chi([\Ifk])}{\ \ \Vert \Ifk \Vert_{\Ocal_K}^s} \cdot \sum_{1+\lambda \in (1+\Cfk  \Ifk^{-1})/\Ocal_{K,\Cfk}^{(1)}} \left(\prod_{i=1}^m\frac{1}{|1+\lambda|_i^{n_i s}}\right) \notag \\
&= 
\Vert \Cfk \Vert_{\Ocal_K}^{-s} \cdot \sum_{[\Ifk] \in {\rm Cl}(\Ocal_K,\Cfk)}
\chi([\Ifk]) \cdot \EE_{\Ocal_K}^{\Cfk\Ifk^{-1}}(\boz_K,\bow_K,s),
\quad \quad \text{where $\bow_K = 1 \in K$.}
\end{align}

The last equality also provides the meromorphic continuation of $L_{\Ocal_K}^{\Cfk}(s,\chi)$ to the whole complex plane.
Note that $\Cfk_{\bow_K} = \Cfk$ as the representative ideals $\Ifk$ are chosen to be integral and coprime to $\Cfk$.
Take $w_{K} = (w_1,\cdots \hspace{-0.03cm}, w_n) \in k^{n}$ so that
\[
\bow_K = 1 = w_1\lambda_n+ \cdots +w_n\lambda_n \in K.
\]
From the second-limit formula in Theorem~\ref{mainthm-2}, we have the following Lerch-type formula:

\begin{cor}\label{cor: LF-2}
Put
\[
\widetilde{L}_{\Ocal_K}^{\Cfk}(s,\chi) \assign \frac{\dfk(\Ocal_K/A)^{s/2}}{s^{m-1}} \cdot L_{\Ocal_K}^{\Cfk}(s,\chi).
\]
Then $\widetilde{L}_{\Ocal_K}^{\Cfk}(0,\chi)=0$ and 
\[
\frac{\partial}{\partial s} \widetilde{L}_{\Ocal_K}^{\Cfk}(s,\chi)\bigg|_{s=0} = - \frac{n_1\cdots n_m}{q_{\infty}^n-1} \cdot 
\sum_{[\Ifk] \in {\rm Cl}(\Ocal_K,\Cfk)}
\chi([\Ifk]) \cdot
\int_{\Tfk(\Ocal_K^{(1)}(\Cfk))}\ln
\big|\ufk_{w_{K}}^{Y_{\Cfk \Ifk^{-1}}}\big(g_{\circ}\cdot \mathbf{z}(\boz_K;\mathbf{t})\big)\big|d\mathbf{t}.
\]
\end{cor}

\appendix

\section{An ``Automorphic'' interpretation of several-variable Eisenstein series}\label{sec: Appendix}

In this appendix, we provide a conceptual and explicit description of how the several-variable Eisenstein series studied in this paper relate to mirabolic Eisenstein series on general linear groups.

\subsection{Mirabolic Eisenstein series}

Let $K$ be a finite extension of $k$ as before, and let $\AA_K$ denote the adele ring of $K$.
We write $|\cdot|_{\AA_K}$ for the adelic norm on the idele group $\AA_K^{\times}$.
For a positive integer $r$, let $S(\AA_K^r)$ be the space of Schwartz functions on $\AA_K^r$ (i.e.~locally constant and compactly supported $\CC$-valued functions).
Given a unitary Hecke character $\chi: K^{\times}\backslash \AA_K^{\times} \rightarrow \CC^{\times}$ and a Schwartz function $\phi \in S(\AA_K^r)$,
the \emph{Goldman section associated with $\chi$ and $\phi$} is defined by (see \cite[(4.1)]{J-S} or \cite[p.~119]{S-S}): for every $g \in \GL_r(\AA_K)$,
\[
\Phi(g,s,\chi,\phi)\assign
|\det g|_{\AA_K}^{s} \cdot \int_{\AA_K^{\times}}\phi\big(a \cdot (0,\cdots \hspace{-0.03cm},0,1)g\big)\chi(a)|a|_{\AA_K}^{-rs}\,  d^{\times}a,
\]
which converges absolutely for $\re(s)>1/r$ and admits meromorphic continuation to the whole complex $s$-plane.
Here we normalize the Haar measure $d^\times a$ on $\AA_K^{\times}$ so that $\text{vol}(O_{\AA_K}^{\times},d^{\times} a) = 1$,
where $O_{\AA_K}$ is the maximal compact subring of $\AA_K$.
Let
\[
P_r = 
\left\{
\begin{pmatrix} a & * \\ 0 & d \end{pmatrix} \in \GL_r \ \bigg|\
a \in \GL_{r-1},\ d \in \GL_1
\right\}.
\]
Then for every $b = \begin{pmatrix} a&*\\ 0 & d\end{pmatrix} \in P_{r}(\AA_K)$ and $g \in \GL_r(\AA_K)$, one has
\[
\Phi(bg,s,\chi,\phi) = |\det a|_{\AA_K}^{s}\chi(d) |d|_{\AA_K}^{(1-r)s} \cdot \Phi(g,s,\chi,\phi).
\]
The \emph{mirabolic Eisenstein series on $\GL_r(\AA_K)$ associated with $\chi$ and $\phi$} is then defined by
\[
\Ecal(g,s,\chi,\phi)\assign \sum_{\gamma \in P_{r}(K)\backslash \GL_r(K)}\Phi(\gamma g,s,\chi,\phi), \quad \re(s)>1.
\]
This series admits meromorphic continuation and satisfies the functional equation (see \cite[Sec.~4]{J-S} or \cite[p.~120]{S-S})
\[
\Ecal(g,s,\chi,\phi) = \Ecal({}^{t}g^{-1},1-s,\chi^{-1},\widehat{\phi}),
\]
where $\widehat{\phi} \in S(\AA_K^r)$ is the Fourier transform of $\phi$:
\[
\widehat{\phi}(x)\assign \int_{\AA_K^r}\phi(y)\psi_K(x\cdot {}^t y)dy.
\]
Here $\psi_K: \AA_K \rightarrow \CC^{\times}$ is a nontrivial continuous additive character on $\AA_K$ trivial on $K$,
and the Haar measure $dy$ on $\AA_K^r$ is self-dual with respect to $\psi_{K}$.
Moreover, the Eisenstein series $E(g,s,\chi,\phi)$ can also be expressed as the following form (see \cite[p.~119]{S-S}):
\begin{equation}\label{eqn: m-E-1}
\Ecal(g,s,\chi,\phi) = |\det g|_{\AA_K}^s \cdot \int_{K^{\times}\backslash \AA_K^{\times}}\Big(\sum_{0\neq x \in K^r} \phi(a x g) \Big)\chi(a)|a|_{\AA_K}^{-rs} d^{\times}a, \quad \re(s)>1.
\end{equation}
In particular, its behavior under right translation is given by
\begin{equation}\label{eqn: m-E-2}
\Ecal(gg',s,\chi,\phi) = |\det g'|_{\AA_K}^s \cdot \Ecal(g,s,\chi,\varrho(g')\phi), \quad \forall g,g' \in \GL_r(\AA_K),
\end{equation}
where $\varrho(g')\phi(x) \assign \phi(xg')$ for every $x \in \AA_K^r$.

\subsection{Conceptual link to several-variable Eisenstein series}

Let $\infty_1,...,\infty_m$ be the places of $K$ lying above the infinite place $\infty$ of $k$.
Then $\AA_K$ is decomposed as $\AA_K^{\infty} \times \prod_{i=1}^m K_{\infty_i}$, where $\AA_K^{\infty}$ is the ring of finite adeles of $K$, i.e.
\[
\AA_K^{\infty} \assign \prod_{v \nmid \infty}\hspace{-0.1cm}{}^{'} K_v, \quad \text{ and\quad $K_v$ is the completion of $K$ at $v$.}
\]
Let $S\big((\AA_K^{\infty})^r\big)$ and $S(K_{\infty_i}^r)$ denote the Schwartz spaces on $(\AA_K^{\infty})^r$ and $K_{\infty_i}^r$, respectively.
We then have a canonical decomposition
\[
S(\AA_K^r) = S\big((\AA_K^{\infty})^r\big) \otimes_{\CC} \Big(\overset{m}{\underset{i=1}{\otimes}} S(K_{\infty_i}^r)\Big).
\]

\begin{rem} \label{rem: Sch-span}
For each place $v$ of $K$, let $O_{K_v}$ be the ring of integers in $K_v$.
Then the maximal compact subring of $\AA_K^{\infty}$ is $\widehat{O}_K\assign \prod_{v\nmid \infty} O_{K_v}$.
We may identify $\AA_K^{\infty}$ with $K\otimes_{O_K}\widehat{O}_K$.
For each projective $O_K$-module $\Lcal$ of rank $r$ in $K^r$, let 
\[
\widehat{\Lcal}\assign \Lcal \otimes_{O_K} \widehat{O}_K \subset K^r \otimes_{O_K}\widehat{O}_K = (\AA_K^{\infty})^r.
\]
Then the Schwartz space $S\big((\AA_K^{\infty})^r\big)$ is spanned by the characteristic functions $\mathbf{1}_{\hat{\bow}+\widehat{\Lcal}}$, where $\Lcal$ runs through all projective $O_K$-modules of rank $r$ in $K^r$ and $\hat{\bow} \in (\AA_K^{\infty})^r$.
In particular, for every $g^{\infty} \in \GL_r(\AA_K^{\infty})$,
the intersection $K^r \cap \widehat{\Lcal}g^{\infty}$ is still a projective $O_K$-module of rank $r$ in $K^r$.
Hence the Schwartz space $S\big((\AA_K^{\infty})^r\big)$
can be spanned by
and $\varrho(g^{\infty})\mathbf{1}_{\hat{\bow}+\widehat{O}_K^r}$
for $g^{\infty}$ varying in $\GL_r(\AA_K^{\infty})$ and $\hat{\bow} \in (\AA_K^{\infty})^r$,
where 
\[
\varrho^{\infty}(g^{\infty})\phi^{\infty}(x) \assign \phi(xg^{\infty}) \quad 
\text{ for every $\phi^{\infty} \in S\big((\AA_K^{\infty})^r\big)$
and $x \in (\AA_K^{\infty})^r$.}
\]
\end{rem}

Put $\mathbb{I}_K^{\infty}:= K^{\times} \backslash (\AA_K^{\infty})^{\times} \cong K^{\times}\backslash \AA_K^{\times}/ \prod_{i=1}^m K_{\infty_i}^{\times}$,
which is compact with volume (under the quotient measure)
\[
\text{vol}(\mathbb{I}_K^{\times},d^{\times}\bar{a}) = \frac{\Pic(O_K)}{q_K-1}.
\]
Let $\widehat{\mathbb{I}}_K^{\infty}$ be the Pontryagin dual group of $\mathbb{I}_K^{\infty}$.
For each $\phi^{\infty} \in S\big((\AA_K^{\infty})^r\big)$, define
\[
\Ecal^{\infty}(g,s,\phi^{\infty})
\assign \frac{1}{\text{vol}(\mathbb{I}_K^{\times},d^{\times}\bar{a})}
\cdot \sum_{\chi \in \widehat{\mathbb{I}}_K^{\infty}}
\Ecal\Big(s,g,\chi,\phi^{\infty}\otimes (\overset{m}{\underset{i=1}{\otimes}} \mathbf{1}_{O_{K_{\infty_i}}^r})\Big), \quad \forall g \in \GL_r(\AA_K),
\]
which is well-defined since the expression in \eqref{eqn: m-E-1} implies that $\Ecal\Big(g,s,\chi,\phi^{\infty}\otimes (\overset{m}{\underset{i=1}{\otimes}} \mathbf{1}_{O_{K_{\infty_i}}^r})\Big)=0$ for all but finitely many $\chi \in \widehat{\mathbb{I}}_K^{\infty}$.
In particular, one checks that
\[
\Ecal\Big(g,s,\chi,\phi^{\infty}\otimes (\overset{m}{\underset{i=1}{\otimes}} \mathbf{1}_{O_{K_{\infty_i}}^r})\Big)
= \int_{\mathbb{I}_K^{\infty}} \overline{\chi(a)} \Ecal^{\infty}(ag,s,\phi^{\infty}) d^{\times} a, \quad \forall \chi \in \widehat{\mathbb{I}}_K^{\infty},
\]
and for every $g = (g^{\infty}, g_{\infty_1},\cdots \hspace{-0.03cm}, g_{\infty_m}) \in \GL_r(\AA_K) = \GL_r(\AA_K^{\infty})\times \prod_{i=1}^m \GL_r(K_{\infty_i})$,
\begin{equation}\label{eqn: m-E-3}
\Ecal^{\infty}(g,s,\phi^{\infty}) = |\det g^{\infty}|_{\AA_K}^s \cdot \Ecal^{\infty}\big((1, g_{\infty_1},\cdots \hspace{-0.03cm}, g_{\infty_m}),s,\varrho(g^{\infty})\phi^{\infty}\big).
\end{equation}
Moreover:

\begin{prop}\label{prop: m-E-L}
Given $g=(g^{\infty},g_{\infty_1},\cdots \hspace{-0.03cm},g_{\infty_m})\in \GL_r(\AA_K) = \GL_r(\AA_K^{\infty})\times \prod_{i=1}^m \GL_r(K_{\infty_i})$,
put $\Lcal_{g^{\infty}} \assign K^r \cap \widehat{O}_K^r (g^{\infty})^{-1}$.
For $\re(s)>1$, we have that
\[
\Ecal^{\infty}\big(g,s,\mathbf{1}_{\widehat{O}_K}\big)
=\frac{(q_K-1) \cdot \Vert \Lcal_{g^{\infty}} \Vert_{O_K}^s}{\displaystyle \prod_{i=1}^m (1-q_{\infty_i}^{-rs})} \cdot \sum_{x \in (\Lcal_{g^{\infty}}-\{0\})/O_K^{\times}}\left(\prod_{i=1}^m \frac{|\det g_{\infty_i}|_{\infty_i}^s}{\nu_{O_{K_{\infty_i}}^r}(xg_{\infty_i})^{rs}}
\right).
\]
Here $q_{\infty_i} \assign q^{\deg \infty_i}$ for $1\leq i \leq m$, and the norm $\nu_{O^r_{K_{\infty_i}}}\in \Ncal^r(K_{\infty_i})$ corresponds to the standard lattice $O_{K_{\infty_i}}^r$ in $K_{\infty_i}^r$, i.e.
\begin{align*}
\nu_{O^r_{K_{\infty_i}}}(a_1,\cdots \hspace{-0.03cm},a_r) & \assign \max\big(|a_i|_{\infty_i}\ \big|\ 1\leq i \leq r\big) \\
&\, = \min\big(|c|_{\infty_i}\ \big|\  (a_1,\cdots \hspace{-0.03cm}, a_r) \in c O^r_{K_{\infty_i}}\big), \quad \forall (a_1,\cdots \hspace{-0.03cm},a_r) \in K_{\infty_i}^r.
\end{align*}
Moreover, for $\hat{\bow} \in K^r\setminus O_K^r$, let $\bow \in K^r \setminus \Lcal_{g^{\infty}}$ such that $\bow -\hat{\bow} g^{\infty} \in \widehat{\Lcal}_{g^{\infty}}$.
Then
\[
\Ecal^{\infty}\big(g,s,\mathbf{1}_{\hat{\bow}+\widehat{O}_K^r}\big)
=\frac{\Vert \Lcal_{g^{\infty}} \Vert_{O_K}^s}{\displaystyle \prod_{i=1}^m (1-q_{\infty_i}^{-rs})} \cdot \sum_{x \in (\bow+\Lcal_{g^{\infty}})/O_{K,\Cfk_{\bow}}^{(1)}}\left(\prod_{i=1}^m \frac{|\det g_{\infty_i}|_{\infty_i}^s}{\nu_{O^r_{K_{\infty_i}}}(xg_{\infty_i})^{rs}}
\right).
\]
\end{prop}

\begin{proof}
Observe that 
\[
|\det g^{\infty}|_{\AA_K}^s = \Vert \Lcal_{g^{\infty}}\Vert_{O_K}, \quad 
\varrho(g^{\infty})\mathbf{1}_{\widehat{O}_K^r}(x) = \mathbf{1}_{\widehat{\Lcal}_{g^{\infty}}}(x)
\]
and
$\mathbf{1}_{\widehat{\Lcal}_{g^{\infty}}}(ux) = \mathbf{1}_{\widehat{\Lcal}_{g^{\infty}}}(x)$ for every $u \in \widehat{O}_K^{\times}$ and $x \in (\AA_K^{\infty})^r$.
Let $h = \#\Pic(O_K)$ and take $a_1,...,a_h \in \AA_K^{\infty}$ be representatives of the cosets in $\mathbb{I}_K^{\infty}/\widehat{O}_K^{\times} \cong \Pic(O_K)$.
Without loss of generality, assume that $a_1=1$.
Put $K_{\infty}\assign \prod_{i=1}^m K_{\infty_i}$.
Then we have the following identification:
\[
K^{\times}\backslash \AA_K^{\times} \stackrel{\sim}{\longleftrightarrow} \bigcup_{1\leq j \leq m}^{\cdot} a_j \widehat{O}_K^{\times} \times \Big( K_{\infty}^{\times}/O_K^{\times}\Big).
\]
Let $g_{\infty} = (g_{\infty_1},\cdots \hspace{-0.03cm}, g_{\infty_m}) \in \GL_r(K_{\infty}) = \prod_{i=1}^m \GL_r(K_{\infty_i})$ and $O_{K_{\infty}}\assign \prod_{i=1}^m O_{K_{\infty_i}}$.
Then
\begin{align*}
\Ecal^{\infty}(g,s,\mathbf{1}_{\widehat{O}_K^r}) 
& = |\det g^{\infty}|_{\AA_K}^{rs} \cdot \Ecal^{\infty}\big((1,g_{\infty}),s,\varrho(g^{\infty})\mathbf{1}_{\widehat{O}_K^r}\big)
\hspace{1.3cm} \text{(by \eqref{eqn: m-E-3})} \\
& \hspace{-1.5cm} = \Vert \Lcal_{g^{\infty}}\Vert_{O_K}^{s} \cdot \frac{\prod_{i=1}^m|\det g_{\infty_i}|_{\infty_i}^s}{\text{vol}(\mathbb{I}_K^{\infty}, d^{\times}\bar{a})} \\
& \hspace{-1.1cm} \cdot \sum_{\subfrac{\chi \in \widehat{\mathbb{I}}_K^{\infty}}{\chi(\widehat{O}_K^{\times})=1}}
\sum_{j=1}^h
\int_{K_{\infty}^{\times}/O_K^{\times}}
\left(\sum_{0\neq x \in K^r} \mathbf{1}_{\widehat{\Lcal}_{g^{\infty}}}(a_i x)
\cdot \chi(a_i)|a_i|_{\AA_K}^{-rs} \cdot 
\mathbf{1}_{O_{K_{\infty}}^r}(a_{\infty}xg_{\infty})\right) |a_{\infty}|_{\AA_K}^{-rs} d^{\times}a_{\infty}.
\end{align*}
As $\text{vol}(\mathbb{I}_K^{\infty}, d^{\times}\bar{a}) = \#\Pic(O_K)/(q_K-1)$ and 
\[
\sum_{\subfrac{\chi \in \widehat{\mathbb{I}}_K^{\infty}}{\chi(\widehat{O}_K^{\times})=1}}
\sum_{j=1}^h \chi(a_j) = 
\begin{cases}
\#\Pic(O_K) & \text{ if $j=1$,} \\
0 & \text{ otherwise,}
\end{cases}
\]
we get
\begin{align*}
\Ecal^{\infty}(g,s,\mathbf{1}_{\widehat{O}_K^r})
& = (q_K-1)\cdot \Vert \Lcal_{g^{\infty}}\Vert_{O_K}^s \cdot \prod_{i=1}^m|\det g_{\infty_i}|_{\infty_i}^s  \\
& \ \ \ \cdot 
\int_{K_{\infty}^{\times}/O_K^{\times}}
\left(\sum_{0\neq x \in K^r}\mathbf{1}_{\widehat{\Lcal}_{g^{\infty}}}(x)
\cdot 
\mathbf{1}_{O_{K_{\infty}}^r}(a_{\infty}xg_{\infty}) \right)|a_{\infty}|_{\AA_K}^{-rs} d^{\times}a_{\infty} \\
& = (q_K-1)\cdot \Vert \Lcal_{g^{\infty}}\Vert_{O_K}^s \cdot \prod_{i=1}^m|\det g_{\infty_i}|_{\infty_i}^s  \\
& \ \ \ 
\cdot
\sum_{x \in (\Lcal_{g^{\infty}}-\{0\})/O_K^{\times}}
\int_{K_{\infty}^{\times}}
\mathbf{1}_{O_{K_{\infty}}^r}(a_{\infty}xg_{\infty}) |a_{\infty}|_{\AA_K}^{-rs} d^{\times}a_{\infty}.
\end{align*}
Finally, as for $1\leq i \leq m$,
\begin{align*}
\int_{K_{\infty_i}^{\times}} \mathbf{1}_{O_{K_{\infty_i}}^r}(a_{\infty_i}xg_{\infty_i})|a_{\infty_i}|_{\infty_i}^{-rs} d^{\times}a_{\infty_i} 
& = \frac{1}{1-q_{\infty_i}^{-rs}} \cdot \nu_{O_{K_{\infty_i}}^r}(xg_{\infty_i})^{-rs},
\end{align*}
we obtain
\[
\Ecal^{\infty}(g,s,\mathbf{1}_{\widehat{O}_K^r})
=\frac{(q_K-1) \cdot \Vert \Lcal_{g^{\infty}} \Vert_{O_K}^s}{\displaystyle \prod_{i=1}^m (1-q_{\infty_i}^{-rs})} \cdot \sum_{x \in (\Lcal_{g^{\infty}}-\{0\})/O_K^{\times}}\left(\prod_{i=1}^m \frac{|\det g_{\infty_i}|_{\infty_i}^s}{\nu_{O_{K_{\infty_i}}^r}(xg_{\infty_i})^{rs}}
\right).
\]

For the second statement, we note that $\varrho(g^{\infty})\mathbf{1}_{\hat{\bow}+\widehat{O}_K^r} = \mathbf{1}_{\bow+\widehat{\Lcal}_{g^{\infty}}}$.
Therefore following the same argument, we obtain the desired formula for $\Ecal^{\infty}(g,s,\mathbf{1}_{\hat{\bow}+\widehat{O}_K^r})$.
\end{proof}

Finally, for each $1\leq i \leq m$ and 
$z^{(i)} \in \Omega_{\infty_i}^r$ satisfying that $\nu_{z^{(i)}}$ is equivalent to the norm $g_{\infty_i}*\nu_{O_{K_{\infty_i}}^r}$, i.e.~$\Bscr(z) = [g_{\infty_i}*\nu_{O_{K_{\infty_i}}^r}] \in \Xcal^{r}(K_{\infty_i})$ where $\Bscr: \Omega_{\infty_i}^r \rightarrow \Xcal^r(K_{\infty_i})$ is the building map introduced in Section~\ref{sec: build},
one has that
\[
\frac{|\det g_{\infty_i}|_{\infty_i}}{\nu_{O_{K_{\infty_i}}^r}(xg_{\infty_i})^r} = \frac{\im(z^{(i)})}{\nu_{z^{(i)}}(x)^r}, \quad \forall x \in K_{\infty_i}^r-\{0\}.
\]
As a consequence, Proposition~\ref{prop: m-E-L} implies that:

\begin{cor}\label{cor: m-E-L}
Let $g = (g^{\infty},g_{\infty_1},\cdots \hspace{-0.03cm},g_{\infty_m}) \in \GL_r(\AA^{\infty}_K)\times \prod_{i=1}^m \GL_r(K_{\infty_i})$.
For every $\boz = (z^{(1)},\cdots \hspace{-0.03cm},z^{(m)}) \in \prod_{i=1}^m \Omega_{\infty_i}^r$ with $\Bscr(z^{(i)}) = [g_{\infty_i}*\nu_{O_{K_{\infty_i}}^r}] \in \Xcal^{r}(K_{\infty_i})$ for $1\leq i \leq m$,
we have
\[
\Ecal^{\infty}(g,s,\mathbf{1}_{\widehat{O}^r_K})
= \frac{q_K-1}{\displaystyle \prod_{i=1}^m(1-q_{\infty_i}^{-rs})} \cdot \EE_{O_K}^{\Lcal_{g^{\infty}}}(\boz,s).
\]
Moreover, for every $\hat{\bow} \in (\AA_K^\infty)^r \setminus \widehat{O}_K^r$ and $\bow \in K^r \setminus \Lcal_{g^{\infty}}$ such that $\bow-\hat{\bow} \in \widehat{\Lcal}_{g^{\infty}}$,
we get
\[
\Ecal^{\infty}(g,s,\mathbf{1}_{\hat{\bow}+\widehat{O}_K^r})
= \frac{1}{\displaystyle \prod_{i=1}^m(1-q_{\infty_i}^{-rs})} \cdot \EE_{O_K}^{\Lcal_{g^{\infty}}}(\boz,\bow,s).
\]
\end{cor}

\begin{rem}
As the Schwartz space $S\big((\AA_K^{\infty})^r\big)$ is spanned by $\varrho(g^{\infty})\mathbf{1}_{\hat{\bow}+\widehat{O}_K^r}$ for $g^{\infty} \in \GL_r(\AA_K^{\infty})$ and $\hat{\bow} \in (\AA_K^{\infty})^r$ (see Remark~\ref{rem: Sch-span}), combining the above Corollary with the transformation property \eqref{eqn: m-E-2} allows us to extend the mirabolic Eisenstein series $\Ecal^{\infty}(g,s,\phi)$ on $\GL(\AA_K)$ to a broader space $\GL_r(\AA_K^{\infty})\times \prod_{i=1}^m \Omega_{\infty_i}^r$. This would lead to an ``adelic formulation'' of our Kronecker limit formula, generalizing the one-variable case in \cite[Theorem 1.1]{Wei20}.
Given the essential role of mirabolic Eisenstein series in the study of automorphic $L$-functions, we expect that this perspective--particularly through the lens of the Berkovich structure on the Drinfeld period domains--will provide new insights into the arithmetic of special values of automorphic $L$-functions. These directions will be explored in future work.
\end{rem}


\end{document}